%% file: MainText.tex
\documentclass[11pt,usenames,dvipsnames]{article}
\usepackage[export]{adjustbox} 
\usepackage{fancyhdr}         
\usepackage{subcaption} 
\usepackage{datetime}    
\usepackage{array}
\usepackage{tikz}
\usetikzlibrary{arrows.meta,calc,positioning,decorations.pathreplacing}
\usepackage{amsmath,amssymb}   
\DeclareMathOperator*{\argmin}{arg\,min}
\usepackage[bottom]{footmisc}
\usepackage{mathtools}        
\usepackage{commath}         
\usepackage{enumitem}

\usepackage{caption}         
\usepackage[includeheadfoot,
            left=0.8in,
            right=0.8in,
            top=0.20in,
            bottom=0.20in,
            headheight=8pt]{geometry}
\usepackage{adjustbox}
\usepackage{comment}
\usepackage{float}
\usepackage{authblk}

\usepackage{booktabs}
\usepackage{siunitx}
\renewcommand{\headrulewidth}{0pt}
\renewcommand{\footrulewidth}{0pt}

\usepackage[
  backend=biber,
  style=numeric,
  sorting=none,
  giveninits=true,
  maxcitenames=2,
  mincitenames=1,
  maxbibnames=99,
  date=year,
  doi=false,isbn=false,url=false
]{biblatex}

\usepackage{hyperref}
\usepackage[open,openlevel=2,depth=3,atend]{bookmark}

\DefineBibliographyStrings{english}{%
  andothers = {et\addabbrvspace al\adddot}
}

\renewcommand*{\multicitedelim}{\addcomma\space}

\DeclareCiteCommand{\parencite}[\mkbibparens]
  {\usebibmacro{prenote}}
  {%
    \printnames{labelname}%
    \addspace
    \printtext[bibhyperref]{\mkbibbrackets{\printfield{labelnumber}}}%
  }
  {\multicitedelim}
  {\usebibmacro{postnote}}

\DeclareCiteCommand{\textcite}[\mkbibparens]
  {\usebibmacro{prenote}}
  {%
    \printnames{labelname}%
    \addspace
    \printtext[bibhyperref]{\mkbibbrackets{\printfield{labelnumber}}}%
  }
  {\multicitedelim}
  {\usebibmacro{postnote}}

\DeclareNameAlias{author}{family-given}
\DeclareDelimFormat{multinamedelim}{\addcomma\space}
\DeclareDelimFormat{finalnamedelim}{\addspace\&\space}
\DeclareFieldFormat{pages}{#1}
\newbibmacro*{journalname}{%
  \iffieldundef{shortjournal}
    {\printtext{\mkbibemph{\printfield{journaltitle}}}}
    {\printtext{\mkbibemph{\printfield{shortjournal}}}}}
\DeclareBibliographyDriver{article}{%
  \usebibmacro{author}%
  \setunit{\labelnamepunct}\newblock
  \printfield[title]{title}%
  \newunit\newblock
  \usebibmacro{journalname}%
  \setunit{\space}%
  \printfield[volume]{volume}%
  \setunit{\space}%
  \printtext[parens]{\printfield{year}}%
  \setunit{\addcomma\space}%
  \printfield{pages}%
  \finentry}
\DeclareFieldFormat[article,inproceedings,incollection,periodical]{volume}{\mkbibbold{#1}}
\DeclareFieldFormat{doilink}{\href{https://doi.org/#1}{https://doi.org/#1}}
\DeclareBibliographyDriver{book}{%
  \usebibmacro{author/editor}%
  \setunit{\labelnamepunct}\newblock
  \printtext{\mkbibemph{\printfield{title}}}%
  \iffieldundef{series}{}{%
    \setunit{\addcomma\space}\printfield{series}%
    \iffieldundef{volume}{}{%
      \addcomma\space \printfield{volume}}}%
  \setunit{\addcomma\space}%
  \printtext[parens]{%
    \printlist{publisher}\addcomma\space\printlist{location}\addcomma\space\printfield{year}}%
  \iffieldundef{doi}{}{%
    \par\noindent\printfield[doilink]{doi}}%
  \finentry
}
\DeclareBibliographyDriver{thesis}{%
  \usebibmacro{author}%
  \setunit{\labelnamepunct}\newblock
  \printtext{\mkbibemph{\printfield{title}}}%
  \setunit{\addcomma\space}%
  \printfield{type}% 
  \setunit{\addcomma\space}%
  \printlist{institution}%
  \setunit{\space}%
  \printtext[parens]{\printfield{year}}%
  \finentry}
\AtEveryBibitem{%
  \clearfield{labelyear}%
  \clearfield{extrayear}%
  \clearfield{month}%
  \clearlist{language}%
}

\usepackage{amsthm}
\usepackage{mathrsfs,bbm,url}
\usepackage[super]{nth}
\usepackage[open,openlevel=2,depth=3,atend]{bookmark}
\hypersetup{pdfstartview=XYZ}
\usepackage[colorinlistoftodos]{todonotes}
\usepackage{mathabx}
\usepackage{tikz-cd}
\tikzcdset{row sep/normal=3.6em,column sep/normal=3.6em}
\usepackage{subcaption}
\usepackage{epstopdf}

\theoremstyle{plain}
\newtheorem{theorem}{Theorem}
\newtheorem{lemma}{Lemma}

\newtheorem{proposition}{Proposition}

\theoremstyle{definition}

\theoremstyle{remark}
\newtheorem{remark}{Remark}

\newcommand{\lip}{\mathrm{Lip}}

\begin{document}

\pagestyle{fancy}
\fancyhf{}                   
\fancyfoot[C]{\thepage}        
\fancypagestyle{plain}{%
  \fancyhf{}%
  \fancyfoot[C]{\thepage}%
  \renewcommand{\headrulewidth}{0pt}%
  \renewcommand{\footrulewidth}{0pt}%
}

\title{Data-Driven Reduced Modeling of Delayed Dynamical Systems via Spectral Submanifolds \\[0.5em] (Submitted to Chaos) }
\author[1]{Giacomo Abbasciano}
\author[1]{Gergely Buza}
\author[1,$\dagger$]{George Haller}
\affil[1]{Department of Mechanical and Process Engineering, ETH Zurich, 8092 Zurich, Switzerland}

\date{}  
\maketitle
\begingroup
\renewcommand\thefootnote{\textdagger}
\footnotetext{Corresponding author: \href{mailto:georgehaller@ethz.ch}{georgehaller@ethz.ch}}
\endgroup

\input{Sections/0_Abstract}
\newpage
\input{Sections/1_Introduction}
\input{Sections/2_Section2}
\input{Sections/3_Section3}
\input{Sections/4_Section4}

\input{Sections/5_Conclusions}
\input{Sections/6_Aknowledgments}
\appendix
\input{Sections/AppendixA}
\input{Sections/AppendixB}
\input{Sections/AppendixC}

\input{Sections/AppendixD}
\addcontentsline{toc}{section}{References}
\printbibliography
\end{document}

%% file: Sections/0_Abstract.tex
\section*{Abstract}
We show how the recent extension of spectral submanifold (SSM) theory to delay differential equations (DDEs) enables data-driven model reduction of nonlinear delay systems. First, using a scalar DDE with a single discrete delay, we compare equation-based and data-driven SSM reductions, to illustrate the need for the latter. 
We then use the same algorithm to obtain purely data-driven, SSM-reduced, delay-free ODE models for several nonlinear delayed systems. Our approach requires no information about the form of the underlying DDE, or about the number and magnitude of the delays it contains.
Our SSM-reduced, low-dimensional models remain predictive even for chaotic dynamics. We also illustrate the use of parametric SSM-reduction to capture bifurcations in systems with both distributed and discrete delays. Finally we extend the theoretical underpinning of delayed SSM-reductions to non-autonomous systems with periodic delays, and apply these results to experimental data from a control system with feedback delay and quantization.

%% file: Sections/1_Introduction.tex
\section{Introduction  }
Many systems in nature and engineering exhibit dynamics that depend on past states. In epidemiology, the spread of diseases is delayed by the incubation time \parencite{AlTuwairqiAlHarbi2022}, and delays in vaccine production can destabilize the disease-free equilibrium. In autonomous vehicle control \parencite{SzakszOroszStepan2024Traffic, SzakszOroszStepan2025Discretization, OroszMolnar2025ConnectedVehicles}, delays arise from human reaction times. More generally, feedback-control systems involve delays due to measurement, computation, and zero-order hold (ZOH) effects \parencite{Stepan2017, InspergerStepan2011, Vizi2024DigitalDetection, Insperger2015Semi-discretizationBalance}.

Such delays can often be neglected without compromising model performance. Indeed, for sufficiently small delays, the dominant dynamics remain close to those of the corresponding delay-free system. More specifically, as shown in \parencite{BuzaHaller2025}, the spectrum of the delayed system retains the eigenvalues of the undelayed system, which remain dominant for small delays, while infinitely many additional eigenvalues originate from the leftmost part of the complex plane, whose associated dynamics decay rapidly. This explains the effectiveness of low-dimensional models and linear controllers that do not account for small delays, enabling the construction of such low-dimensional yet accurate models for computational cost reduction. However, as the delay increases, the additional eigenvalues move toward the imaginary axis and may become dominant, eliminating the spectral gap between delay-induced and undelayed eigenvalues. These may then cross the imaginary axis, leading to bifurcations that are absent in the undelayed system \parencite{Insperger2015Semi-discretizationBalance}.

At the same time, delay differential equations (DDEs) are inherently infinite-dimensional, because their phase space is the Banach space of continuous history functions. Consequently, they can exhibit richer dynamics than ordinary differential equations (ODEs) of the same nominal dimension. Indeed, even scalar delay equations may display chaotic dynamics \parencite{MackeyGlass1977}.

Beyond delay effects, nonlinearizable phenomena, such as the coexistence of isolated steady states and bifurcations, are ubiquitous in applied sciences and engineering. Spectral submanifolds (SSMs) enable the construction of mathematically justified reduced models for such nonlinearizable dynamics on the slowest family of low-dimensional, normally attracting invariant manifolds emanating from a hyperbolic fixed point. These manifolds form a family whose members all have the same dimension and are tangent to a direct sum of dominant eigenspaces at the fixed point, i.e., to a slow spectral subspace (see \parencite{Haller2025, Haller2016, Cenedese2022}).

Prior work based on SSM has analyzed and modeled such behaviors in systems including turbulent flows \parencite{Kaszs2024, KaszasHaller2025}, fluid sloshing \parencite{Cenedese2022Data-drivenSubmanifolds, Axs2024}, nonsmooth dynamics \parencite{Bettini2024, Morsy2025Data-DrivenProblem, AbbascianoEndreszStepanHaller2026JSV}, control of soft robots \parencite{Alora2022, Alora2023, Alora2025, Kaundinya2025, BettiniEtAl2026SoftRoboticsSSM}, alongside global and local bifurcations \parencite{AbbascianoEndreszStepanHaller2026JSV, KingEtAl2026ParametricSSM}.

Equation-driven approaches for nonlinear delay systems have been extensively employed. Most pertinent to this work, classic center-manifold reduction \parencite{StepanHaller1995, Diekmann1995} has recently been formally extended to SSMs directly in the function space \parencite{Szaksz2025SpectralSystems}, without the need for discretization of the delay equation. A rigorous existence and regularity result for such SSMs in delay systems has subsequently been established in \parencite{BuzaHaller2025}. Despite their interpretability, equation-driven SSM-reduction is computationally demanding and therefore typically limited to low-dimensional systems, as the cost increases with both the SSM dimension and the polynomial order desired to parametrize the manifold and reduced dynamics. Moreover, equation-driven SSM reduction requires full knowledge of the governing equations, parameters, and delays, which is often unavailable in real-world applications.

These limitations of equation-driven model reduction motivate the development of data-driven methods for modeling delay systems. Approaches such as neural delay differential equations (NDDEs) and SINDy-based \parencite{BruntonEtAl2016PNAS} methods have been successfully applied to systems with multiple discrete delays \parencite{JiOrosz2022, JiOrosz2024SingleTraj, BredaEtAl2025DDMethods, PecileEtAl2025, SandozEtAl2023}. These methods are formulated under different assumptions on the available prior knowledge about the delays. For instance, when the number and magnitude of the delays are not known a priori, E-SINDy estimates the delay magnitudes from data after specifying an upper bound on their number \parencite{BredaEtAl2025DDMethods}. In the absence of any prior knowledge of the number of delays, P-SINDy instead estimates the maximum delay directly from data. This leads to a trajectory-based formulation that focuses on reproducing the observed trajectory and extending the simulation beyond the measurement window, rather than explicitly identifying the right-hand side of the system \parencite{BredaEtAl2025DDMethods}.

In addition, these approaches rely on the explicit form of the DDE to select basis functions. The choice of these functions may in turn strongly impact the performance of the algorithm \parencite{BredaEtAl2025DDMethods}. This restriction is particularly relevant when the full dynamics include non-polynomial terms, such as those in the Mackey--Glass equation \parencite{BredaEtAl2025DDMethods}.
Finally, NDDEs provide a neural-network-based framework and likewise require delay values to be identified from data when they are not known a priori \parencite{JiOrosz2022, JiOrosz2024SingleTraj, BredaEtAl2025DDMethods}. 

Here, we show how SSM-based reduction can be extended to data-driven settings for DDEs. This approach enables the modeling of nonlinear phenomena, including coexisting isolated steady states and bifurcations, directly from observed data. 
In particular, the specific form of the DDE plays no role and number and magnitude of the delays need not be a priori known. As long as the delayed system has class $C^1$ governing equations, SSMs always exist with at least $C^1$ regularity and are tangent to finite-dimensional slow spectral subspaces in the infinite-dimensional phase space of generally unknown DDE. In order to handle control applications with ZOH effects, we also prove an extension of the SSM results of \parencite{BuzaHaller2025} to systems with periodic delays.

This paper is organized as follows. Section \ref{section2} recalls the settings for nonlinear autonomous DDEs and the main theoretical results on their SSMs, and extends the theory to periodically delayed DDEs. Section \ref{sec:methodologies} presents the data-driven procedures used for SSM-based model reduction, with particular emphasis on the estimation of the quantities required directly from data. Section \ref{section4} demonstrates the proposed methodology through a sequence of examples. It begins by comparing equation-driven and data-driven SSM reductions for the Hutchinson equation, thereby motivating the need for data-driven approaches. The section then applies data-driven SSM reduction to numerically simulated systems of increasing complexity, including single-delay systems with limit cycles, multiple-delay systems exhibiting chaotic dynamics, and a distributed-delay system with multiple saddle-node bifurcations of limit cycles, analyzed through parametric SSMs. Finally, the theoretical extension to periodically delayed DDEs is illustrated on experimental data involving delay and quantization effects typical of digital controllers.

%% file: Sections/2_Section2.tex
\section{Existence of SSMs for nonlinear DDEs }\label{section2}
In this section, we provide a high-level overview of delay differential equations on infinite-dimensional spaces, and introduce the relevant theory, following \parencite{Pazy1983, Diekmann1995}, before recalling the main results of \parencite{buza2025existence} on the existence and regularity of SSMs in delay systems. Finally, we show how these results can be applied to justify and carry out data-driven SSM-reduction in such systems.

\subsection{Analogies between ODEs and differential equations on infinite dimensional spaces }
We recall the notion of strongly continuous semigroups  \parencite{Pazy1983}, fundamental to the study of linear evolutionary systems on Banach spaces (in particular, delay equations). In finite dimensions, they coincide with matrix exponentials arising from first-order linear ODEs. This extends naturally to a Banach space \(X\). Specifically, if \(A \in \mathcal{L}(X)\) is bounded, the Cauchy problem
\[
\dot{x} = Ax, \qquad x(0) = x_0,
\]
admits the unique solution
\[
x(t) = e^{tA}x_0 = \sum_{k=0}^{\infty} \frac{t^k}{k!}A^k x_0, \qquad t \in \mathbb{R}.
\]

For unbounded operators, typical in PDEs and DDEs, this representation may fail and solutions may exist only for \(t \geq 0\). It is therefore preferable to describe the dynamics via solution operators, leading to the notion of strongly continuous semigroups, namely a family of linear operators $T:[0,\infty)\rightarrow\mathcal{L}(X)$ satisfying

\[
T(0) = I, \qquad T(t+s) = T(t)T(s), \quad t,s \geq 0,
\]
and 
\begin{equation}\label{eq:strongContinuity}
\lim_{t \downarrow 0} \|T(t)x - x\| = 0, \quad \text{for all } x\in X.
\end{equation}
The classical notion of continuous dependence on time and initial conditions are encoded in \eqref{eq:strongContinuity} and in the boundedness of $T(t)$, respectively.

The infinitesimal generator $A:\mathrm{dom}(A)\rightarrow X$ of the semigroup $T(t)$ is the linear operator defined on
\[
\mathrm{dom}(A) \doteq \left\{ x \in X \;\middle|\; \lim_{\tau \downarrow 0} \frac{T(\tau)x - x}{\tau} \text{ exists} \right\},
\]
as
\[
Ax \doteq \lim_{\tau \downarrow 0} \frac{T(\tau)x - x}{\tau}, \qquad x \in \mathrm{dom}(A).
\]

The fundamental qualitative behavior of solutions to nonlinear ordinary differential equations induced by vector fields on open subsets of \(\mathbb{R}^n\) is classically described through the notion of the flow map. In the setting of Banach spaces, we have to distinguish the forward time direction which leads to the notion of a semiflow. Indeed, the backward initial value problem is usually not well posed: backward solutions do not exist, or may fail to be unique \parencite{Diekmann1995}. 

To introduce the concept of a semiflow following \parencite{BuzaHaller2025}, we let $X$ be a Banach space. A family of maps
$\{\varphi_t\}_{t\geq 0}$, $\varphi_t:X\to X$, is called a \emph{semiflow} on $X$ if
\[
\varphi_0=\mathrm{id}_X, 
\]
\[
\varphi_t\circ\varphi_s=\varphi_{t+s}
\qquad \text{for all } t,s>0,
\]
and $\{\varphi_t\}_{t\geq 0}$ is jointly continuous in space and time. It is known that DDEs, under mild assumptions, generate semiflows \parencite{Diekmann1995}.

\subsection{Existence of SSMs for nonlinear autonomous DDEs}\label{seciont2_Existence}
We briefly recall the setting for nonlinear autonomous delay differential equations and the main results from \parencite{BuzaHaller2025} that will be used in the examples below. We consider the phase space
\[
X := C([-\tau,0];\mathbb{R}^n),
\]
namely the Banach space of continuous $\mathbb{R}^n$-valued functions on $[-\tau,0]$, with the supremum norm
\[
\|\cdot\|_X := \|\cdot\|_\infty,
\]
where $\tau>0$ represents the maximal delay and $n\in\mathbb{N}$. Let $O\subset X$ be an open set, and let
$f:O\to\mathbb{R}^n$ be locally Lipschitz continuous. On $X$, we study autonomous delay differential equations of the form
\begin{subequations}\label{eq:IVP}
\begin{align}
\dot{x}(t) &= f(x_t), \label{eq:IVP-a}\\
x_0 &= u. \label{eq:IVP-b}
\end{align}
\end{subequations}
Here, for each $t$, the history segment $x_t\in X$ is defined by $x_t(\vartheta)=x(t+\vartheta)$, $\vartheta\in[-\tau,0]$, while $u\in O$ is the prescribed initial datum. A solution of \eqref{eq:IVP} is a map
\[
x:[-\tau,t^*)\to\mathbb{R}^n,\qquad t^*\in (0,\infty)\cup\{\infty\},
\]
such that \eqref{eq:IVP-b} is satisfied, $x$ is continuously differentiable on $(0,t^*)$, and \eqref{eq:IVP-a} holds throughout that interval. In addition, the right derivative of $x(t)$ at $t=0$ must exist and coincide with $f(x_0)$. Classic well-posedness results \parencite{Diekmann1995} show that, when $f$ is locally Lipschitz, every initial condition $u\in O$ gives rise to a unique solution of \eqref{eq:IVP}. We write $\{\varphi_t\}_{t\geq 0}$ for the semiflow induced by \eqref{eq:IVP}.

Assume now that the semiflow $\varphi$ has an equilibrium at the origin, i.e., $0\in O$ and $f(0)=0$. Let $A$ be the infinitesimal generator of the semigroup $t\mapsto D\varphi_t(0)$, and let $\Sigma\subset \sigma(A)$ be a nonempty bounded set of the spectrum with associated spectral projection $P_\Sigma$. For simplicity, we also assume that $f$ is of class $C^k$ for some $k\geq 1$. We will use the notation $\sigma(A)$ for the spectrum of a linear operator $A$, we will denote by $P_\Sigma$ the spectral projection associated with a bounded spectral subset $\Sigma \subset \sigma(A)$ and its image by $\operatorname{im}(P_{\Sigma})$.
\bigskip

\begin{theorem} \label{thm:main1}
\parencite{BuzaHaller2025}.
\textit{Suppose that a spectral subset $\Sigma$ is of the form
$\Sigma=\{\lambda\in\sigma(A)\mid \operatorname{Re}\lambda\geq \gamma\}$ for some $\gamma\in\mathbb{R}$. Then, there exists a neighbourhood $U \subset O$ of the stationary state
$0\in O$ and a locally invariant (under the semiflow $\varphi$) $C^{\ell}$ manifold $W^{\Sigma}\subset U$ (an
SSM) tangent to the spectral subspace $\operatorname{im}(P_{\Sigma})$ at $0$. Here the integer $1\leq \ell\leq k$ is
determined by the spectral gap condition}
\begin{equation}\label{eq:spectralGap}
\sup \operatorname{Re}\,(\sigma(A)\setminus \Sigma)< \ell \inf \operatorname{Re}\Sigma.
\end{equation}
\textit{Moreover, $W^{\Sigma}$ attracts trajectories that remain in $U$ at an exponential rate of $\gamma$,
synchronized along the leaves of a $C^0$ foliation.}
\end{theorem}

Trajectories approaching an attracting SSM synchronize exponentially fast with trajectories on the SSM. For each on-SSM trajectory, the off-SSM initial conditions that converge to it most rapidly form its stable fiber. These fibers define a foliation over the SSM that is invariant under the semiflow, although individual leaves need not be invariant \parencite{Szalai2020, Haller2025ModelingControls, Bettini2025, BettiniEtAl2026ObliqueProjections}.

We emphasize that the sole assumption for the existence of a $C^1$ SSM $W^{\Sigma}$ is $f$ being $C^1$. A more detailed and precise treatment can be found in \parencite{BuzaHaller2025}.

\subsection{Existence of SSMs for nonlinear nonautonomous DDEs with periodic delay}\label{seciont2_ExistencePerDel}
We now recall the setting for nonautonomous DDEs with periodic delay and state the corresponding existence and regularity result. Its proof, together with further details, is provided in Appendix \ref{appendixC}. This result will be applied to a control example in which the periodic delay arises from ZOH.

Let $X : = C([-\tau,0];\mathbb{R}^n)$; here $\tau>0$ denotes the maximal delay and $n \in \mathbb{N}$.
Let $O \subset X$ denote an open set containing the origin. 
We consider systems of the form
\begin{subequations} \label{eq:DDE_classicala} 
\begin{align}
&\dot{x}(t) = f(t,x_t), \qquad t \geq s \in \mathbb{R}  \label{eq:DDE_classical1a} \\ 
& x_s = u,  \label{eq:DDE_classical2a}
\end{align}
\end{subequations}
for $f : \mathbb{R} \times O \to \mathbb{R}^n$ continuous, time-periodic, i.e., $f(t+p,u) = f(t,u)$ for some $p > 0$ and all $(t,u) \in \mathbb{R} \times O$, and $f(t,0) = 0$ for all $t \in \mathbb{R}$. This last requirement is usually achieved by mapping a nontrivial periodic orbit to the origin via a time-periodic change of coordinates. We suppose moreover that $f(t,\cdot)$ is $f(t,\cdot)$ is $C^k$ with $k \geq 1$ for all $t \in \mathbb{R}$.
We shall also suppose $(t,u) \mapsto D_2^jf(t,u)$ is continuous for all $j \leq k$. % this is needed for IMPFT, smoothness of Rhat

Equation \eqref{eq:DDE_classicala} determines a time-dependent semiflow $S:\mathcal{D}(S)\rightarrow X$ (see \eqref{eq:AIE}-\eqref{eq:81}). We denote by $U$ its linear part and by $\sigma(U)$ the associated Floquet spectrum.

To make the statement more concise, we set, for a subset $V\subset O$,
\begin{displaymath}
    \mathcal{U}(V) = \big\{ (t,s,u) \in \mathcal{D}(S) \; \big| \; S([t,s] \times \{s\} \times \{u \}) \subset V \big\},
\end{displaymath}
and $\mathcal{U}_{s,u}(V) = \{t \in [s,\infty) \; | \; (t,s,u) \in \mathcal{U}(V)\}$.
Moreover, let us denote
\begin{displaymath}
    X = X_\Sigma(t) \oplus X_{\Sigma'}(t) := \mathrm{im}(P_\Sigma(t))\oplus \ker(P_\Sigma(t)), \qquad t \in \mathbb{R},
\end{displaymath}
with projection $P_\Sigma$ as in Proposition~\ref{prop:spect}.

\begin{theorem} \label{thm:main}
Let $\Sigma = \{ \lambda \in \sigma(U) \; | \; |\lambda| > \gamma \}$ for some $\gamma > 0$, $\gamma \notin |\sigma(U)|$. The following hold.
        \begin{enumerate}[label=\upshape{(\roman*)}]
        \item \label{st1} \textup{(Invariant manifold)} There exists an open neighbourhood $V \subset O$ of $0$ and a continuous family of $C^1$ submanifolds $W^\Sigma_t \subset V$ tangent to $X_\Sigma(t)$ at $0$, which is locally invariant under $S$: If $u \in W^\Sigma_s$, then $S(t,s;u) \in W^\Sigma_t$ for all $(t,s,u) \in \mathcal{U}(V)$.
        \item \label{st2} \textup{(Smoothness)}  If  $k > 1$, the smoothness of $W^\Sigma_t$ is $C^\ell$ for all $t \in \mathbb{R}$, where $\ell$ is the maximal $\ell \leq k$ for which
        \begin{displaymath}
            (\sup|\sigma(U)\setminus \Sigma|)< (\inf|\Sigma|)^\ell
        \end{displaymath}
        holds. 
        \item \label{st3} \textup{(Attractivity)} There exist a continuous maps $\pi_t : V \to W^\Sigma_t$, $t \in \mathbb{R}$, such that 
        \begin{equation}
            \pi_t \circ S(t,s;u) = S(t,s; \pi_s (u)), \qquad \text{for all } (t,s,u) \in \mathcal{U}(V).
            \label{eq:foliation_invariance}
        \end{equation}
        Moreover, for any $u \in U$ and $s \in \mathbb{R}$, there exists $C > 0$ such that 
        \begin{equation}
            | S(t,s;u) - S(t,s;\pi_s(u))| < C e^{  (t-s) \ln(\gamma)/p}, \qquad \text{for all } t \in \mathcal{U}_{s,u}(V).
            \label{eq:foliation_decay}
        \end{equation}
        The preimages $\{\pi_t^{-1}(w)\}_{w \in W^\Sigma_t}$ are Lipschitz submanifolds foliating $V$ for each $t \in \mathbb{R}$.
    \end{enumerate}
\end{theorem}

\subsection{Data-driven approximation of SSMs of nonlinear autonomous DDEs with bounded delay }
We leverage the existence of finite-dimensional SSMs in the phase space of a DDE to construct data-driven approximations of such manifolds and their topologically conjugate reduced dynamics. Based on Takens' theorem \parencite{Takens1981}, a diffeomorphic copy of a $d-$dimensional smooth manifold is embedded with probability one in a finite-dimensional delay embedding space \(\mathbb{R}^k\) via delayed nondegenerate observables of dimension $k > 2d$.

Specifically, time series of observables are used through the SSMLearn algorithm \parencite{Cenedese2022} to reconstruct a manifold at \(\vartheta=0\), that is, at the present time, together with its \(\vartheta\)-independent reduced dynamics. The result is a low-dimensional system of ODEs that captures the dominant behavior after transient decay to the SSM, which takes place exponentially fast along the stable foliation over the attracting invariant manifold \parencite{BuzaHaller2025}.

As an advantage, reduction to an SSM enables backward-time advection of initial conditions, now represented as finite-dimensional vectors rather than history functions. This is not possible in the full phase space of a DDE, where the solution operator defines only a forward semiflow, since backward-time advection would require specifying future histories as initial conditions. An overview of data-driven SSM modeling for autonomous systems via delay embedding is provided in Appendix~\ref{appendixA}.

%% file: Sections/3_Section3.tex
\section{Methodology}\label{sec:methodologies}

We describe the data-driven procedures used in this paper to construct SSM-reduced models of delayed dynamical systems. Since the phase space of a delay differential equation (DDE) is infinite dimensional and is therefore not accessible from measurements, the reduction is performed in a finite-dimensional delay-coordinate embedding space constructed from observed time series. Within this space, we identify an SSM, learn the reduced dynamics, and validate the resulting model using trajectory prediction errors or, in the presence of chaotic dynamics, invariant quantities.

\subsection{Correlation dimension estimation from data}

For systems exhibiting chaotic dynamics, the dimension of the reduced model must be sufficiently high to embed the chaotic invariant set of interest. We estimate the fractal dimension of the attractor from data using the correlation dimension $m$ introduced by \parencite{GrassbergerProcaccia1983}. Since the full phase space of a DDE is infinite dimensional, we first construct a delay-coordinate embedding from a scalar observable \(x\), following \parencite{Takens1981}, and then estimate the correlation dimension in the resulting finite-dimensional delay-embedding space.

Choosing the embedding dimension $k$ to be at least twice the SSM dimension $d$, ensures that the SSM $W \subset X$ can be embedded into the delay space \parencite{Takens1981}.
Furthermore, assuming the chaotic attractor is contained in one of the SSMs in $X$, once in the delay embedding space we invoke \parencite{SauerYorkeCasdagli1991} to ensure that an orthogonal projection embeds the attractor into a $d$-dimensional spectral subspace, provided that $d > 2m$. This condition determines the sufficient, not necessary, dimension of the SSM.

Given \(N\) embedded data points \(x_i\), the correlation integral is defined as
\[
C(l)=\frac{1}{N^2}
\#\left\{(i,j): 1\le j<i\le N,\ \|x_i-x_j\|_2<l\right\}.
\]
For sufficiently small \(l\), the correlation integral scales as
\[
C(l) \sim l^m,
\qquad l\to 0.
\]
Hence, \(m\) is estimated via linear regression as the slope of the approximately linear scaling region in the log-log plane \((\log l,\log C(l))\).

It suffices to carry out this computation in the delay-embedding space, since \parencite{SauerYorkeCasdagli1991} guarantees that the delay coordinate map $F$ preserves the correlation dimension $m$ of the attractor $A$ in $X$, i.e., $\mathrm{dim}(F(A))=\mathrm{dim}(A)$.

We compare the correlation dimension estimated from trajectories in the delay-embedding space with that estimated after orthogonal projection onto the reduced coordinates to verify that no topological information is lost during the projection \parencite{KaszasHaller2025}.

\subsection{Lyapunov exponent estimation from data}
In the presence of chaotic dynamics, we estimate the leading Lyapunov exponent from data in order to validate the SSM-reduced model. The procedure follows \parencite{Nastac2017}, as also used in \parencite{KaszasHaller2025}. The leading Lyapunov exponent quantifies the average exponential separation rate of nearby trajectories, and its reciprocal provides an estimate of the predictability horizon of a chaotic solution.

Consider an autonomous system of the form \eqref{eq:IVP-a}, with initial
conditions \(x_0\) and \(x_0^*\) of the form \eqref{eq:IVP-b}, both close
to the attractor. Let
\[
\varepsilon(0)=x_0^*-x_0,
\qquad
\delta(0)=\|\varepsilon(0)\|_X
\]
be the initial perturbation and its norm in the phase space \(X\). The evolution of the perturbation can be locally
approximated by
\[
\varepsilon(t)=D\varphi_t(x_0)\varepsilon(0).
\]
Thus \[
\delta(t)
=
\|x^*(t)-x(t)\|_X
\approx
\|D\varphi_t(x_0)(x_0^*-x_0)\|_X.
\]
For sufficiently small perturbations, the dominant exponential growth rate
is characterized by
\[
\delta(t)\approx \delta(0)e^{\lambda t},
\]
where \(\lambda\) denotes the leading Lyapunov growth rate. If
\(\lambda>0\), nearby trajectories separate exponentially due to a
stretching mechanism, while the folding mechanism of the chaotic attractor
keeps trajectories bounded in a compact set.

For DDEs, distances in the full phase space are generally unavailable from data. We therefore compute separations in the delay-coordinate embedding space $\mathbb{R}^k$, using the Euclidean norm. \parencite{SauerYorkeCasdagli1991} justifies estimating the Lyapunov exponent from the image $F(A)$ of the attractor $A$. To estimate the leading Lyapunov exponent, we first advect an initial condition until the trajectory reaches the chaotic attractor. Denoting this state by \(x(t_{Ch})\), we then reset the time origin to \(t_{Ch}=t_0\). We initialize \(N\) nearby trajectories, typically \(N\approx 50\), in a small neighborhood of \(x(t_{Ch})\). For the full DDE, we prescribe these initial conditions as constant history functions with values chosen in a small ball around \(x(t_{Ch})\). All trajectories are then advanced forward in time.
For each pair of trajectories \(1\leq j<i\leq N\), we compute the separation
\[
\delta_{ij}(t)=\|x_i(t)-x_j(t)\|_2
\]
in the delay-embedding space. The mean normalized separation is then defined as
\[
\frac{\delta(t)}{\delta(0)}
=
\frac{1}{N_t}
\sum_{(i,j)}
\frac{\delta_{ij}(t)}{\delta_{ij}(0)},
\]
where \(N_t\) is the number of trajectory pairs. The leading Lyapunov exponent is obtained by linear regression of $\log\left(\frac{\delta(t)}{\delta(0)}\right)$ against \(t\), restricted to the initial time interval over which the separation grows approximately exponentially. For validation, we apply the same procedure both to trajectories generated by the full system, when the equations are available, and to trajectories predicted by the SSM-reduced model. When only experimental data are available, this estimation becomes more challenging because one usually has limited control over the initialization of nearby trajectories, especially for DDEs, whose initial conditions are history functions.

\subsection{Invariant measure estimation via probability density functions}
To further validate SSM-reduced models in the presence of chaotic dynamics, we compare statistical properties of the chaotic attractor. In particular, we estimate probability density functions (PDFs) of training trajectories in the reduced space and use them to estimate a physically relevant invariant measure of the chaotic attractor, following \parencite{Liu2024, Xu2024}. Agreement between the PDFs obtained from the full system and from the SSM-reduced model indicates that the reduced dynamics reproduces the long-term statistical behavior of the original system, even when test trajectory prediction is no longer meaningful due to high sensitivity to initial conditions.

\subsection{Systematic data-driven SSM reduction for autonomous DDEs}

We now summarize the data-driven procedure used to construct SSM-reduced models for autonomous DDEs with bounded delay.

\begin{itemize}
\item[\textbf{Step 1)}] \textbf{Verification of the hypothesis of Theorem~\ref{thm:main1}.}
For autonomous DDEs of the form introduced in Section~\ref{seciont2_Existence}, we first verify the assumptions of Theorem~\ref{thm:main1}. This theorem guarantees the existence of a finite-dimensional invariant manifold emanating from an equilibrium in the infinite-dimensional phase space \(X\), tangent to the selected spectral subspace. As long as the delayed system has
class $C^1$ governing equations, low-dimensional SSMs always exist with at least $C^1$ regularity. This follows from the fact that we can always choose the spectral subset $\Sigma$ in Theorem~\ref{thm:main1} containing the rightmost eigenvalues associated to eigenvectors spanning the spectral subspace $\operatorname{im}(P_{\Sigma})$ at $0$ with eq. \eqref{eq:spectralGap} satisfied with $\ell=1$.

\item[\textbf{Step 2)}] \textbf{Selection of the SSM dimension.}
The slow SSM dimension \(d\) is determined by the selected slow spectral subspace associated with \(\Sigma\). This is therefore determined by the configuration of the rightmost eigenvalues in the complex plane. In equation-free settings, this information may be inferred from data using, for example, spectrograms or the observed behavior of trajectories. For chaotic behaviour, assuming the chaotic attractor is contained in one of the SSMs in $X$, once in the delay embedding space we invoke \parencite{SauerYorkeCasdagli1991} to ensure that an orthogonal projection embeds the attractor into a $d$-dimensional spectral subspace, provided that \(d>2m\).

\item[\textbf{Step 3)}] \textbf{Assessment of SSM smoothness.}
After selecting \(d\), the smoothness of the SSM can be estimated from data by comparing the reduced dynamics on a \(d\)-dimensional SSM with that on a larger \(d^*\)-dimensional SSM, where \(d^*>d\). The enlarged spectral set \(\Sigma^*\) includes the slowest real eigenvalue or complex-conjugate pair outside \(\Sigma\). The spectral gap can then be estimated from the real parts of the eigenvalues inferred from the reduced models. The smoothness of the SSM is $\ell$ respecting the spectral gap condition \eqref{eq:spectralGap}. If unstable eigenvalues are present, the selected spectral subspace must include all unstable directions for an SSM to be attracting. If \(\Sigma\) contains only unstable eigenvalues, then the corresponding SSM coincides with the (strong) unstable manifold, it is thus unique, and it is as smooth as the full system by Theorem \ref{thm:main1}.
We emphasize that the smoothness of the SSM is not required explicitly for the data-driven SSM reduction to be carried out. Indeed, both the SSM parametrization and the reduced dynamics are approximated directly from data using multivariate polynomial regression in SSMLearn \parencite{Cenedese2022}. Consequently, no Taylor expansion needs to be computed.

\item[\textbf{Step 4)}] \textbf{Delay-coordinate embedding.}
Since the full phase space \(X\) is infinite dimensional and therefore inaccessible from data, we construct a finite-dimensional delay-coordinate embedding from the available observables, as in \eqref{delayEmbCoordMap}. This yields data in a delay-embedding space \(\mathbb{R}^k\), with \(k>2d\). In this space, we seek an embedded SSM \(\widetilde{\mathcal W(E)}\) that is diffeomorphic to an SSM \(\mathcal W(E)\subset X\), whose existence has been established in Step~1, by \parencite{Takens1981} as detailed in Appendix~\ref{delayEmb}.

\item[\textbf{Step 5)}] \textbf{Selection of the SSM polynomial approximation order.}
The \(\theta=0\) slice of an SSM, corresponding to the current-time component of history functions in $X$, is identified in the delay- embedding space by multivariate polynomial regression, as detailed in Appendix~\ref{ManFit}. When the delay-embedding lag \(\Delta t\) is small, trajectories in $\mathbb{R}^k$ tend to flatten \parencite{Haller2025ModelingVideos}, which often makes linear manifold approximations accurate. Indeed, for a delay-embedded scalar time series
\[
\mathbf y(t)
=
\begin{pmatrix}
s(t)\\
s(t+\Delta t)\\
\vdots\\
s(t+(k-1)\Delta t)
\end{pmatrix}
\in \mathbb{R}^k,
\]
a Taylor expansion around $\Delta t = 0$ gives
\[
\mathbf y(t)
=
\sum_{\ell=0}^{m}
\frac{\Delta t^\ell}{\ell!}
s^{(\ell)}(t)\mathbf v_\ell
+
\mathcal O\!\left((k-1)^{m+1}\Delta t^{m+1}\right),
\qquad
\mathbf v_\ell
=
\begin{pmatrix}
0^\ell\\
1^\ell\\
\vdots\\
(k-1)^\ell
\end{pmatrix}.
\]
Therefore, up to an error of order $\mathcal O\!\left((k-1)^{m+1}\Delta t^{m+1}\right)$, the trajectory lies in the subspace
$\mathcal P_m
=
\operatorname{span}\{\mathbf v_0,\mathbf v_1,\ldots,\mathbf v_m\}
\subseteq \mathbb{R}^k$.
This yields a hierarchy of linear approximating subspaces, $\mathcal P_0
\subset
\mathcal P_1
\subset
\cdots
\subset
\mathbb{R}^k$. In particular, the first-order approximation gives
\[
\mathbf y(t)
=
s(t)\mathbf v_0
+
\Delta t \dot{s}(t)\mathbf v_1
+
\mathcal O\!\left((k-1)^2\Delta t^2\right),
\]
so for sufficiently small \(\Delta t\), the embedded trajectory lies close to the two-dimensional plane
\[
\mathcal P_1
=
\operatorname{span}
\left\{
\begin{pmatrix}
1\\
1\\
\vdots\\
1
\end{pmatrix},
\begin{pmatrix}
0\\
1\\
\vdots\\
k-1
\end{pmatrix}
\right\}.
\]
For larger embedding lags, higher-dimensional linear subspaces may be required to obtain more accurate approximations.

\item[\textbf{Step 6)}] \textbf{Learning the reduced dynamics.}
The reduced dynamics is approximated from data in the reduced space using multivariate polynomial regression, or radial basis function (RBF) interpolation in the presence of chaotic dynamics when this yields better performance, as detailed in Appendix~\ref{RedDynFit}. The graph-style parametrization of the SSM provides a one-to-one lifting from the reduced coordinates to the delay-coordinate embedding space. This avoids introducing artificial self-intersections in the reconstructed manifold and thereby preserves the uniqueness of solutions for smooth autonomous DDEs \parencite{Diekmann1995}.

When polynomial reduced dynamics are used, the polynomial order \(n_{\mathrm{red}}\) is selected together with the SSM approximation order \(n_{\mathrm{SSM}}\) by minimizing the normalized mean trajectory error (NMTE) on test data. For a test trajectory \(\mathbf y_{\mathrm{test}}\) and its SSM-based prediction \(\mathbf y_{\mathrm{model}}\), the NMTE is defined \parencite{Cenedese2022} as
\[
\mathrm{NMTE}(n_{\mathrm{SSM}\star},n_{\mathrm{red}\star})
=
\frac{1}{N_p \|\mathbf{y}\|_{\max}}
\sum_{j=1}^{N_p}
\left\|
\mathbf{y}_{\mathrm{test}}^j(t_j)
-
\mathbf{y}_{\mathrm{model}}^j(t_j;n_{\mathrm{SSM}\star},n_{\mathrm{red}\star})
\right\|.
\] 
The selected orders are then
\[
(n_{\mathrm{SSM}},n_{\mathrm{red}})
=
\arg\min_{n_{\mathrm{SSM}\star},n_{\mathrm{red}\star}}
\mathrm{NMTE}(n_{\mathrm{SSM}\star},n_{\mathrm{red}\star}),
\]
where $N_p$ is the number of points of the time series. Too high polynomials lead to overfitting, while too low polynomials lead to underfitting. For chaotic systems, short-time trajectory prediction is not a reliable validation criterion due to high sensitivity to initial conditions. In that case, \(n_{\mathrm{red}}\) is selected based on the ability of the reduced model to reproduce the statistics of the chaotic attractor, as presented in the next Step 7).

\item[\textbf{Step 7)}] \textbf{Validation of the SSM-reduced model.}
For nonchaotic dynamics, model quality is assessed using the NMTE on test trajectories. Each test trajectory is orthogonally projected onto the reduced coordinates to approximate the base point of the leaf of the foliation, the reduced initial condition is advected under the learned reduced dynamics, and the resulting prediction is lifted back to the delay-coordinate embedding space using the learned SSM parametrization. For better approximation of leaves' base points see the works \parencite{Bettini2025} on linear oblique projections, and \parencite{BettiniEtAl2026ObliqueProjections} on nonlinear oblique projections.

For chaotic dynamics, pointwise trajectory prediction is meaningful only over short times because of sensitive dependence on initial conditions. Therefore, we validate the reduced model by comparing PDFs of training trajectories in the reduced coordinates, comparing leading Lyapunov exponents, and qualitatively assessing whether test trajectories of the reduced model reconstruct the chaotic attractor in the delay-coordinate embedding space.
\end{itemize}

\subsection{Systematic data-driven SSM reduction for nonautonomous DDEs with periodic delay}

For nonautonomous DDEs with periodic delay, the data-driven reduction procedure is performed in a Poincar\'e section. In this section, the dynamics is described by an autonomous discrete-time map, i.e. the Poincar\'e map. The reduction procedure is therefore analogous to Steps 1)--7) above for autonomous DDEs, with the following modifications.

\begin{itemize}

\item[\textbf{Step 1)}] 
The assumptions of Theorem~\ref{thm:main} are verified instead of those of Theorem~\ref{thm:main1}. For \(C^1\) periodic-delay DDEs, this theorem guarantees the existence of a time-periodic family of \(C^1\) SSMs tangent to a time-periodic family of slow spectral subspaces.

\item[\textbf{Step 2)}]
The smoothness of the SSM is estimated using the eigenvalues of the discrete map approximating the Poincar\'e map, rather than the eigenvalues of a continuous-time SSM-reduced model.

\item[\textbf{Step 6)}]
The reduced dynamics in the Poincar\'e section is learned as a discrete-time reduced map. As in the autonomous case, this map may be approximated using either polynomial regression or radial basis function interpolation.

\end{itemize}

%% file: Sections/4_Section4.tex
\section{Data-driven SSM reduction for delay differential equations}\label{section4}
In this section, we apply Theorem~\ref{thm:main1} to examples of increasing complexity drawn from \parencite{BredaEtAl2025DDMethods, SzakszOroszStepan2024Traffic, BuzaHaller2025}. We begin with the Hutchinson equation \parencite{Hutchinson1948, BredaEtAl2025DDMethods}, a scalar DDE with a single discrete delay that exhibits a stable limit cycle for a suitable choice of parameters. For this system, we first derive an SSM-reduced model through the formal equation-driven calculations of \parencite{Szaksz2025SpectralSystems}, based on \parencite{Diekmann1995}, as detailed in Appendix~\ref{SectionB1}. We then show in Appendix~\ref{B2} that this result is equivalent to the equation-based formulation of \parencite{BuzaHaller2025}, which relies on the rigorous results summarized in Theorem~\ref{thm:main1}.

The equation-driven SSM reduction is used here only as a benchmark to identify the limitations that the data-driven approach aims to overcome. We therefore compare the equation-driven results with a data-driven SSM reduction obtained by applying SSMLearn \parencite{Cenedese2022} directly to trajectories generated by the governing equations, following the procedure outlined in Section~\ref{sec:methodologies}. This comparison highlights the limitations of equation-based SSM reduction and motivates the use of data-driven SSM methods.

We then analyze several numerically simulated systems with multiple discrete and distributed bounded delays, including examples exhibiting chaotic dynamics. We find that the SSM-based approach captures both local and global bifurcations, consistently with the results reported in \parencite{AbbascianoEndreszStepanHaller2026JSV, KingEtAl2026ParametricSSM}. These examples also illustrate some limitations of existing data-driven reduction approaches for DDEs discussed in \parencite{BredaEtAl2025DDMethods}. In particular, we obtain SSM-reduced ODE-models with predictive power, without prior knowledge of the specific form of the DDE, including the number and magnitude of discrete and distributed delays.

Finally, we apply Theorem~\ref{thm:main}, the new theoretical result established in this paper for nonautonomous DDEs with periodic delay, to experimental control data affected by ZOH and quantization effects.

\subsection{Hutchinson DDE - Equation-driven vs. data-driven reduced modeling}
A powerful equation-driven reduction method in the context of DDEs have been developed in \parencite{Diekmann1995} and extended to SSMs in \parencite{Szaksz2025}, whose existence for DDEs has recently been rigorously established by \parencite{BuzaHaller2025}. Two equivalent approaches have been proposed by \parencite{Szaksz2025SpectralSystems} and \parencite{BuzaHaller2025} to compute finite-dimensional SSMs emanating from equilibria of delay equations in an infinite dimensional phase space.
The core idea is to approximate the solution of the invariance PDE via arbitrarily high-order truncated Taylor expansions. This leads to homological equations, that can be solved as linear algebraic systems whose solutions yield the coefficients of both the SSM parametrization and the reduced dynamics up to the desired polynomial order. These constructions require sufficient smoothness of the DDE right-hand side and generic nonresonance conditions on the spectrum of the linearization, as detailed in \parencite{Szaksz2025SpectralSystems, BuzaHaller2025}. However, the amount of computations required by these equation-driven SSM reduction approaches is high and increases rapidly with both the dimension of the SSM and the truncation order, as the number of coefficients grows combinatorially. We apply both methods to the Hutchinson delay system, as detailed in Appendix~\ref{appendixB}, reducing it to a two-dimensional (2D) SSM to motivate the need for data-driven methods in real-world applications, where the underlying equations are unknown and parameter estimation is uncertain.

The Hutchinson DDE is of the form
\begin{equation}\label{eq:hutch1}
x'(t)=r\,x(t)\left(1-\frac{x(t-\tau)}{K}\right).
\end{equation}
as discussed in \parencite{Hutchinson1948}. After shifting the equilibrium \(x_*=K\) to the origin, the state at time \(t\) is represented by the history segment
\[
y_t(\vartheta)=y(t+\vartheta), \qquad \vartheta\in[-\tau,0],
\]
defined on the phase space \(X:=C([-\tau,0],\mathbb{R})\), as detailed in Appendix \ref{appendixB}. Throughout, we use the parameter values \(r=1.8\), \(K=10\), and \(\tau=1\).

The dominant modes arise from an unstable complex conjugate pair of eigenvectors of the linearized system about the origin equilibrium. Applying Theorem~\ref{thm:main1}, and choosing $\gamma>0$, this ensures the existence of an SSM as smooth as the full system, thus $C^\infty$.
The spectrum in Figure \ref{fig:SpectralGap_H} is computed using the toolbox developed in \parencite{AppeltansSilmMichiels2022, AppeltansMichiels2023}, a root-finding algorithm for quasi-polynomials arising from the characteristic equations of DDEs.

\begin{figure}[H]
  \centering
  \includegraphics[width=0.6\textwidth]{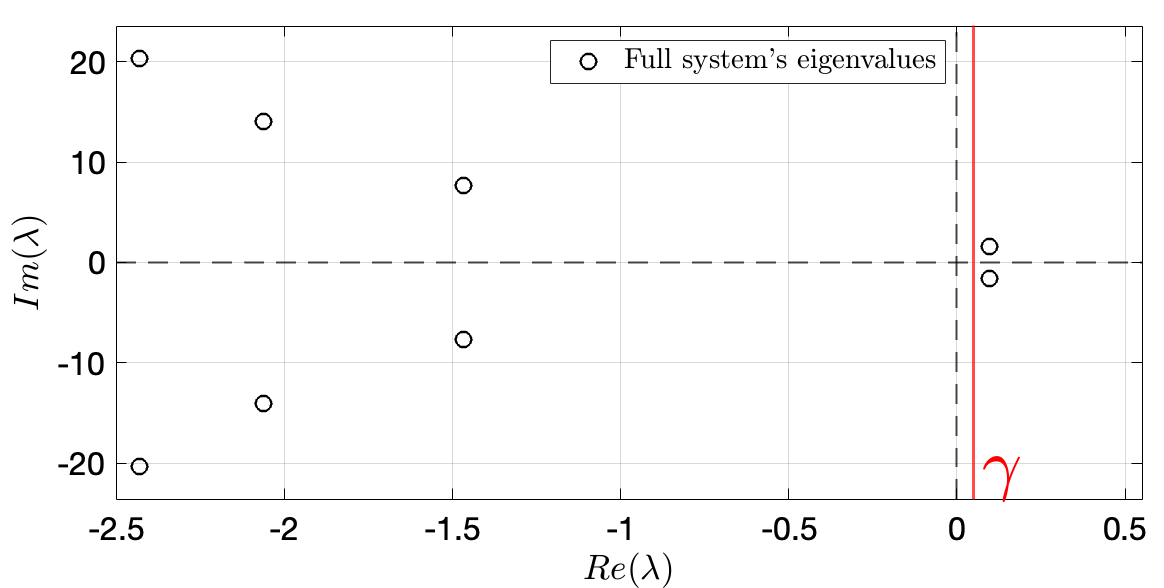}
  \caption{Eigenvalues of the linearization of the full system \eqref{eq:hutch1} at the fixed point at the origin (black markers). The red vertical line indicates the value $\gamma$ chosen to satisfy the assumptions of Theorem~\ref{thm:main1}.}
  \label{fig:SpectralGap_H}
\end{figure}

The SSM parametrization and the reduced dynamics up to third order, computed via the procedure outlined in \parencite{Szaksz2025}, are given by

\begin{equation}\label{eq:eqManifold}
\begin{aligned}
W(\vartheta = 0, z, \bar z)
&= z + \bar z \quad + \frac{1}{2}\big((0.0046 - 0.021 i)\, z^2+
(0.0024)\, z\bar z
+(0.0046 + 0.021 i)\, \bar z^2\big) \\
&\quad + \frac{1}{6}\big((-0.0019 + 0.0025 i)\, z^3+(-0.0017 + 0.0026 i)\, z^2\bar z \\
&\quad + (-0.0017 - 0.0026 i)\, z\bar z^2
+ (-0.0019 - 0.0025 i)\big)\, \bar z^3   + \mathcal{O}(|z|^4),
\end{aligned}
\end{equation}

\begin{equation}\label{eq:eqRD}
\begin{aligned}
\dot z &=
(0.097 - 1.6i)\,z \\
&\quad + (0.072 - 0.042i)\,z^{2}
+ (0.0055 + 0.0082i)\,z\bar z
- (0.066 - 0.050i)\,\bar z^{2} \\
&\quad - (0.0039 + 0.00073i)\,z^{3}
- (0.0034 - 0.00054i)\,z^{2}\bar z
+ (0.0018 - 0.0030i)\,z\bar z^{2} \\
&\quad + (0.00080 - 0.0039i)\,\bar z^{3} + O(|z|^4),
\end{aligned}
\end{equation}

\begin{equation}\label{eq:eqRDconj}
\begin{aligned}
\dot{\bar z} &=
(0.097 + 1.6i)\,\bar z \\
&\quad - (0.066 + 0.050i)\,z^{2}
+ (0.0055 - 0.0082i)\,z\bar z
+ (0.072 + 0.042i)\,\bar z^{2} \\
&\quad + (0.00080 + 0.0039i)\,z^{3}
+ (0.0018 + 0.0030i)\,z^{2}\bar z
- (0.0034 + 0.00054i)\,z\bar z^{2} \\
&\quad - (0.0039 - 0.00073i)\,\bar z^{3} + O(|z|^4).
\end{aligned}
\end{equation}

In Fig.~\ref{fig:SSMsVsTheta}, we plot the objects depicted in a three-dimensional observable space. The left panel shows the spectral subspace, an SSM obtained from the second-order Taylor truncation of the manifold parametrization, and a trajectory of the full system approaching the limit cycle shown in red. This limit cycle is obtained by lifting the prediction of the SSM-reduced model. The right panel shows SSM slices for three values of $\vartheta$ in the phase space $X$. Each red line represents a state, i.e., a history function in $X$, while markers indicate intersections of the states with the SSM slices. Advecting the state generates trajectories evolving from initial conditions near the unstable fixed point at the origin toward a stable limit cycle. A marker on the SSM at $\vartheta=0$ corresponds to the marker evaluated one second later on the SSM at $\vartheta=-1$, since for a DDE the state satisfies
\[
y_{t+1}(\vartheta)
= y(t+1+\vartheta)
= y\bigl(t+(\vartheta+1)\bigr)
= y_t(\vartheta+1).
\]
This relation holds up to an $O(|z|^4)$ correction on the manifold, as we use a third-order Taylor expansion to approximate the spectral submanifold near the origin. Time-shift invariance is incorporated into the manifold through the interior equations in Appendix~\ref{SectionB1}.

The SSM that we are going to be approximate from data using SSMLearn~\parencite{Cenedese2022} corresponds to the slice at $\vartheta=0$ (yellow) in Fig.\ref{fig:SSMsVsTheta_b}, representing the invariant manifold of states, now the red markers  $\vartheta=0$ at the current time $y_t(0)=y(t)$. Slices with $\vartheta<0$ correspond to time-shifted states and thus provide backward-shifted predictions. 

\begin{figure}[H]
  \centering

  \begin{subfigure}{0.49\textwidth}
    \centering
    \includegraphics[width=\linewidth]{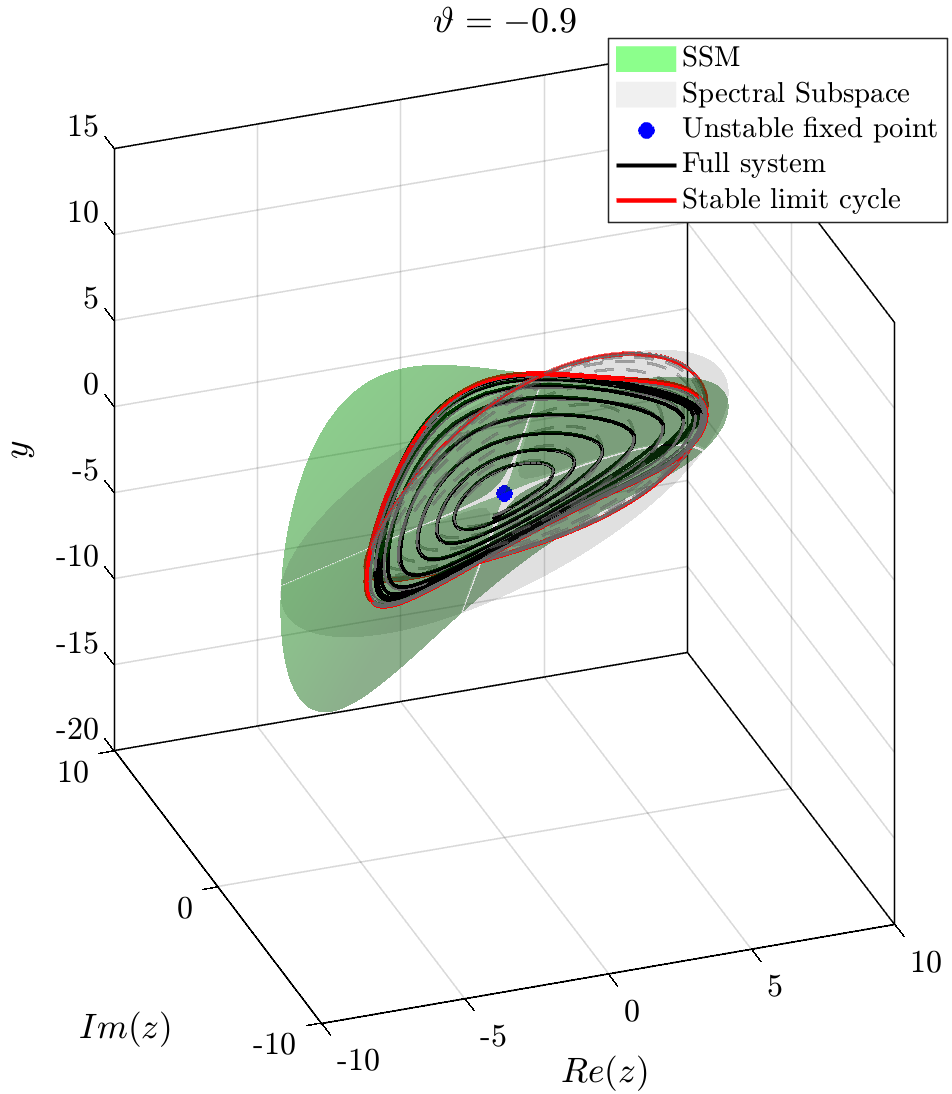}
    \caption{}
    \label{fig:SSMsVsTheta_a}
  \end{subfigure}
  \hfill
  \begin{subfigure}{0.49\textwidth}
    \centering
    \includegraphics[width=\linewidth]{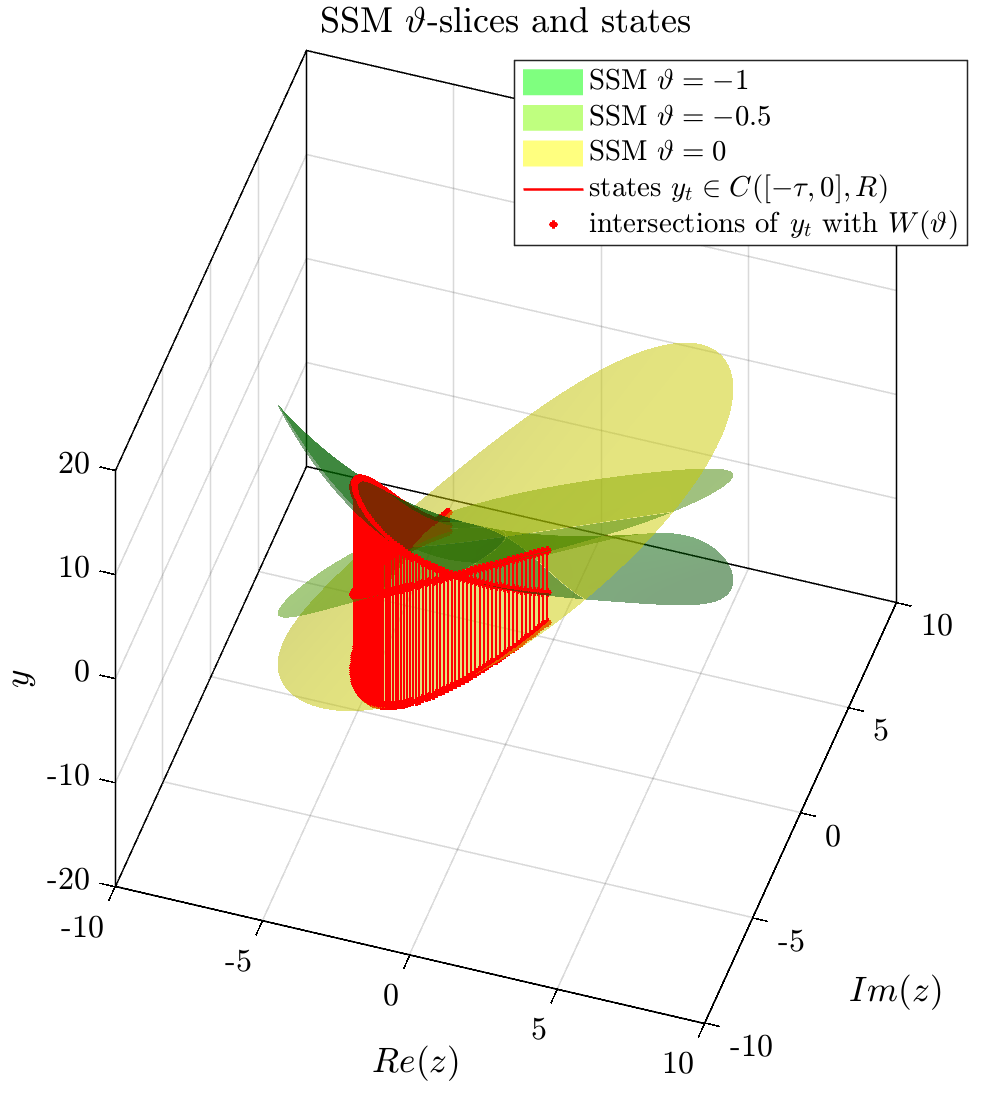}
    \caption{}
    \label{fig:SSMsVsTheta_b}
  \end{subfigure}
    \caption{(a): Comparison between a full-system trajectory (black) of system \eqref{eq:hutch1} and the SSM-reduced model \eqref{eq:eqManifold}-\eqref{eq:eqRDconj} prediction (grey) obtained projecting the full trajectory's initial condition onto the spectral subspace via the pairing, lifted with the cubic equation-driven SSM parametrization at $\vartheta=-0.9$. (b): Equation-driven SSM slices for different values of $\vartheta$ (yellow: $\vartheta=0$, medium green: $\vartheta=-0.5$, dark green: $\vartheta=-1$). Red curves denote the states in $X$, i.e. hystory functions, and red markers indicate their intersections with each slice.}
  \label{fig:SSMsVsTheta}
\end{figure}

Projecting the initial condition of a numerically advected full system trajectory onto the spectral subspace using the pairing \eqref{eq:pairing} enables a linear approximation of the base point of the nonlinear leaf of the stable foliation over the finite-dimensional dominant spectral submanifold. The reduced initial condition is then evolved using the SSM-reduced model and lifted via the manifold parametrization at $\vartheta=0$ in Figure \ref{fig:Heq_TrajPred} to compare full-system simulations with equation-driven SSM predictions.

\newcommand{\imgAAwidth}{0.49\textwidth}
\newcommand{\imgBBwidth}{0.49\textwidth}

\begin{figure}[H]
  \centering

  \begin{subfigure}[c]{\imgAAwidth}
    \centering
    \includegraphics[width=\linewidth]{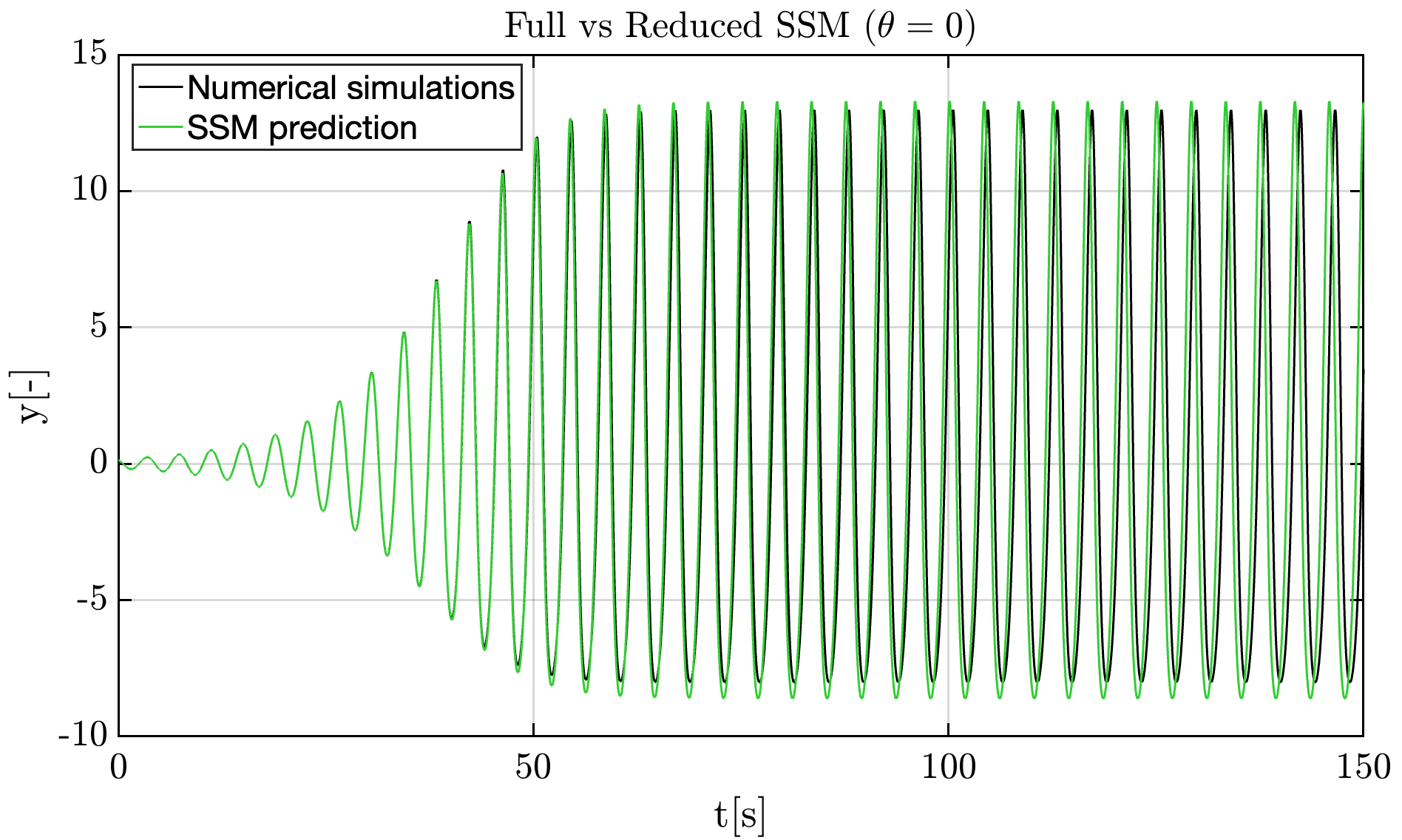}
    \caption{}
    \label{fig:H_Traj}
  \end{subfigure}
  \hfill
  \begin{subfigure}[c]{\imgBBwidth}
    \centering
    \includegraphics[width=\linewidth]{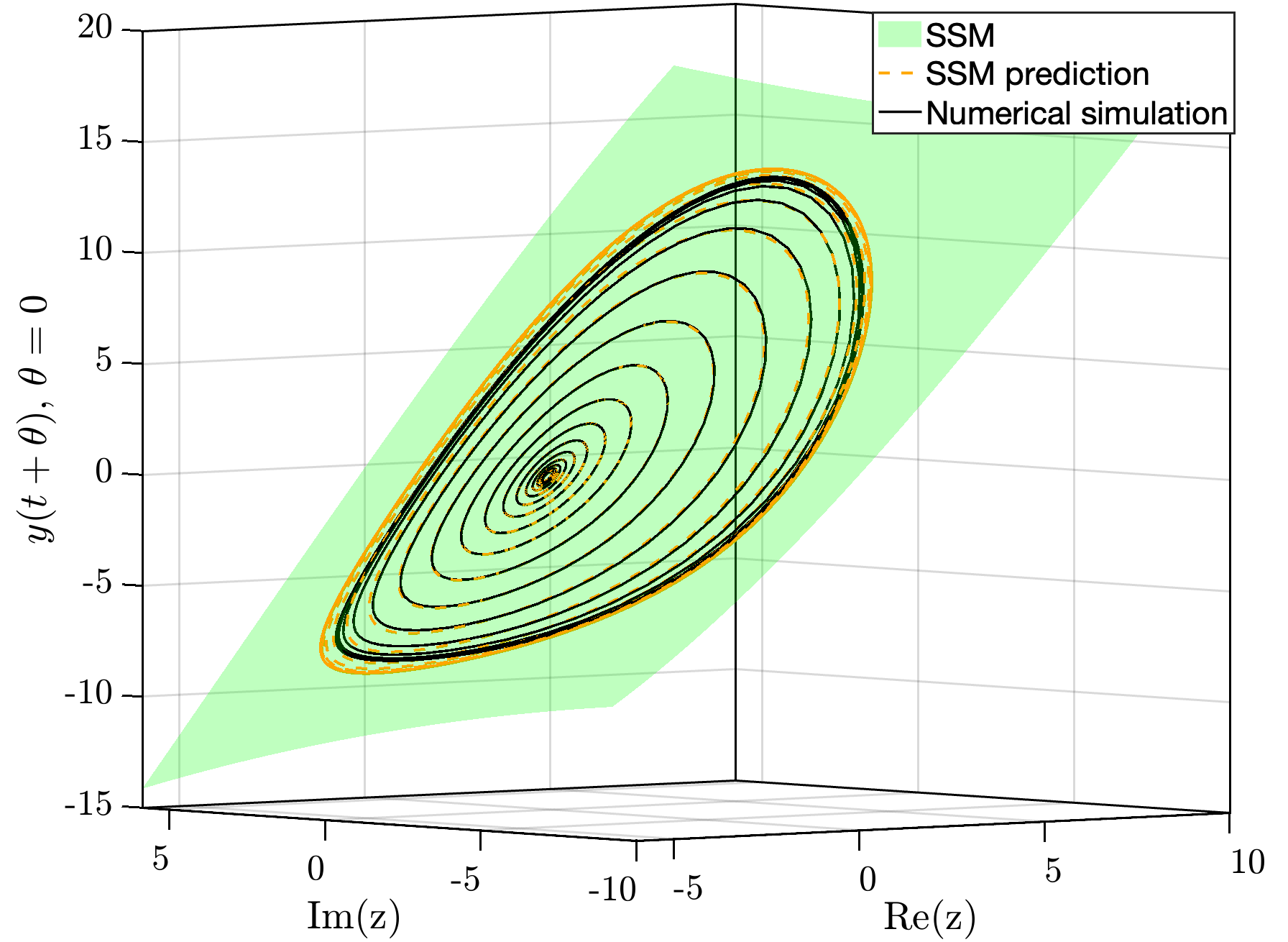}
    \caption{}
    \label{fig:SSM}
  \end{subfigure}
  \caption{(a): Comparison between a trajectory of system \eqref{eq:hutch1} and the corresponding equation-driven SSM-reduced model prediction. (b): Reconstructed SSM used to lift reduced dynamics into the full space.}
  \label{fig:Heq_TrajPred}
\end{figure}   

Prediction errors in the limit-cycle amplitude are observed, while a frequency error induces a phase discrepancy that increases toward the end of the time series (Fig.~\ref{fig:H_Traj}). These errors can be reduced by increasing the truncation order of the reduced dynamics, at the cost of a higher-order manifold expansion.

A data-driven approach can accelerate this process while eliminating the need for explicit knowledge of the governing delay equations. More generally, when higher-dimensional SSMs are required, the computational cost of equation-driven reductions increases substantially. In an equation-driven setting, deriving higher-order reduced dynamics typically requires computing the SSM parametrization to sufficiently high order, at least one order below that of the reduced dynamics, in order to solve the corresponding homological equations. This requirement is avoided in the proposed data-driven approach, where the approximation orders of the SSM and of the reduced dynamics can be selected independently. Furthermore, the data-driven approach does not require the explicit verification of nonresonance conditions, since the SSM is fitted directly using polynomial regression. Most importantly, a data-driven approach is particularly advantageous for real-world systems, where the governing equations and parameter values are rarely known with high accuracy. Identifying the number and magnitude of delays from data is even more challenging, especially in the presence of distributed delays. 

To address these limitations identified with equation-driven SSM reduction, we propose in this paper a data-driven SSM-reduction approach for DDEs, and start by applying it to the Hutchinson delay equation \eqref{eq:hutch1} following the procedure described in Section \ref{sec:methodologies}. Numerical data are generated from the DDE \eqref{eq:hutch1} with the same parameter values used for the equation driven reduction.

Finite-dimensional SSMs exist under generic conditions, i.e., $f$ in \eqref{eq:IVP} being $C^1$ \parencite{BuzaHaller2025}. In this example, inspection of a spectrogram of the initial transient indicates that the dominant eigenvalues form an unstable complex conjugate pair, ensuring the existence of a 2D $C^\infty$ SSM $W$ in $X$. Since the phase space $X$ is infinite-dimensional, it cannot be fully observed. We therefore leverage Takens’ theorem \parencite{Takens1981}, which, under generic nondegeneracy conditions on the dynamics and observables, guarantees the existence of a copy \(\tilde W\) diffeomorphic to $W$, i.e., an embedding of the selected SSM \(W\), in a delay-coordinate space of dimension $k$ strictly greater than twice that of the SSM $d$.

To embed the 2D SSM capturing the dominant dynamics, we construct a delay-embedding space of dimension at least $5$, using $y$ as the observable. We approximate $\tilde W$, the embedding of the smooth SSM $W$, directly from data via polynomials. In this particular case, being the slow manifold built on all the unstable linear modes, the SSM $W$ coincides with the unstable manifold emanating from the equilibrium, and hence is unique in any smoothness class. The eigenvalues of the reduced dynamics correctly capture the two dominant eigenvalues of the full system, computed using the root-finding toolbox \parencite{AppeltansSilmMichiels2022, AppeltansMichiels2023} (see Fig.~\ref{fig:Eigenvalues_H}).

\begin{figure}[H]
  \centering
  \includegraphics[width=0.6\textwidth]{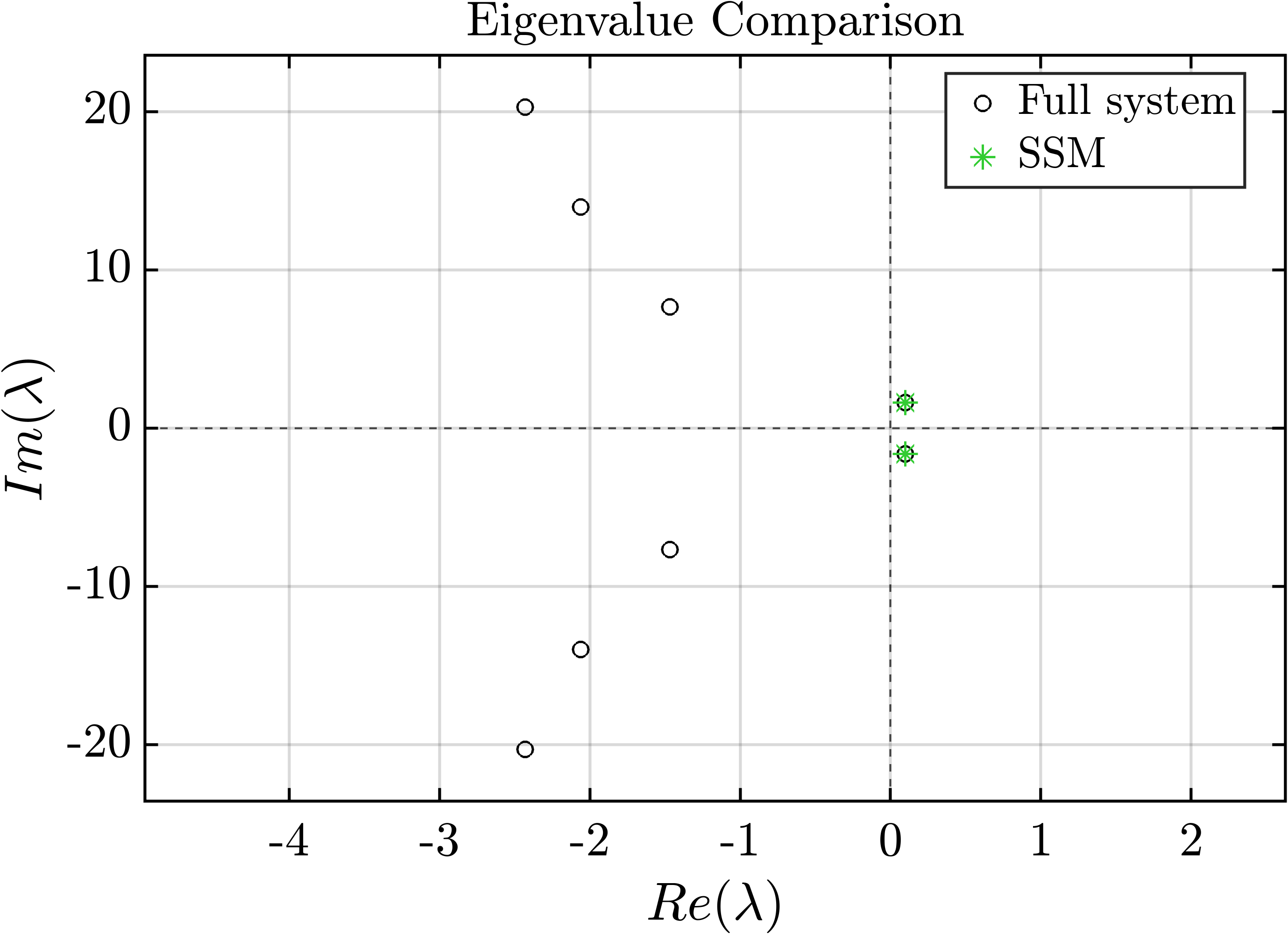}
  \caption{Comparison between true dominant eigenvalues of the data-driven 2D SSM-reduced dynamics of system \eqref{eq:hutch1} (green) and those from the full system obtained via root-finding \parencite{AppeltansSilmMichiels2022, AppeltansMichiels2023} (black).}
  \label{fig:Eigenvalues_H}
\end{figure}

\newcommand{\imgAwidth}{0.65\textwidth}
\newcommand{\imgBwidth}{0.35\textwidth}

Six training trajectories are sufficient to construct a predictive fifth-order SSM-reduced ODE model that captures the dominant dynamics of the underlying delay equation on an SSM approximated with cubic polynomials. The cubic SSM approximation significantly improves prediction accuracy as the delay-embedding lag \(\Delta t=5\) was selected to further reduce the NMTE. The initial conditions of the training trajectories, generated via numerical simulation of the DDE, are placed near the unstable anchor fixed point to accurately reconstruct unstable directions. These trajectories are eventually attracted to the stable limit cycle as they lie within its domain of attraction.

We evaluate the predictive power of the SSM model on previously unseen trajectories in Figure \ref{fig:TestTraj_H} and obtain a normalized mean trajectory error (NMTE) \parencite{Cenedese2022} of $1.7\%$.

\begin{figure}[H]
  \centering
  \includegraphics[width=0.8\textwidth]{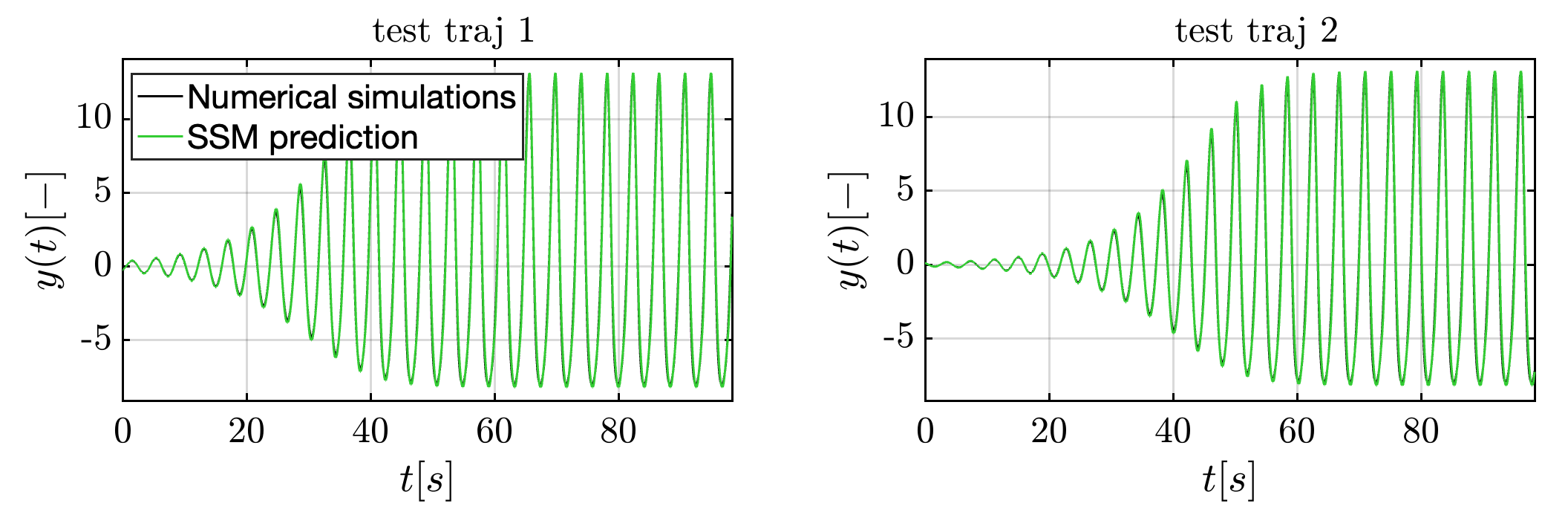}
  \caption{Comparison between two test trajectories of the full system \eqref{eq:hutch1} (black) and their data-driven SSM-based predictions (green).}
  \label{fig:TestTraj_H}
\end{figure}

Data-driven SSM reduction seeks an optimal polynomial approximation of a diffeomorphic copy of $W(z,\bar z,\vartheta=0)$ directly from data at $\vartheta=0$, rather than recovering the Taylor expansion derived from an equation-driven formulation. Figure~\ref{fig:reducedEquations_H} shows a three-dimensional projection of the regressed SSM, a test trajectory with its SSM-based prediction, and the corresponding polynomial reduced dynamics.

\begin{figure}[H] 
  \centering
  \begin{minipage}[c]{\imgAwidth}
    \centering
    \includegraphics[width=1\linewidth]{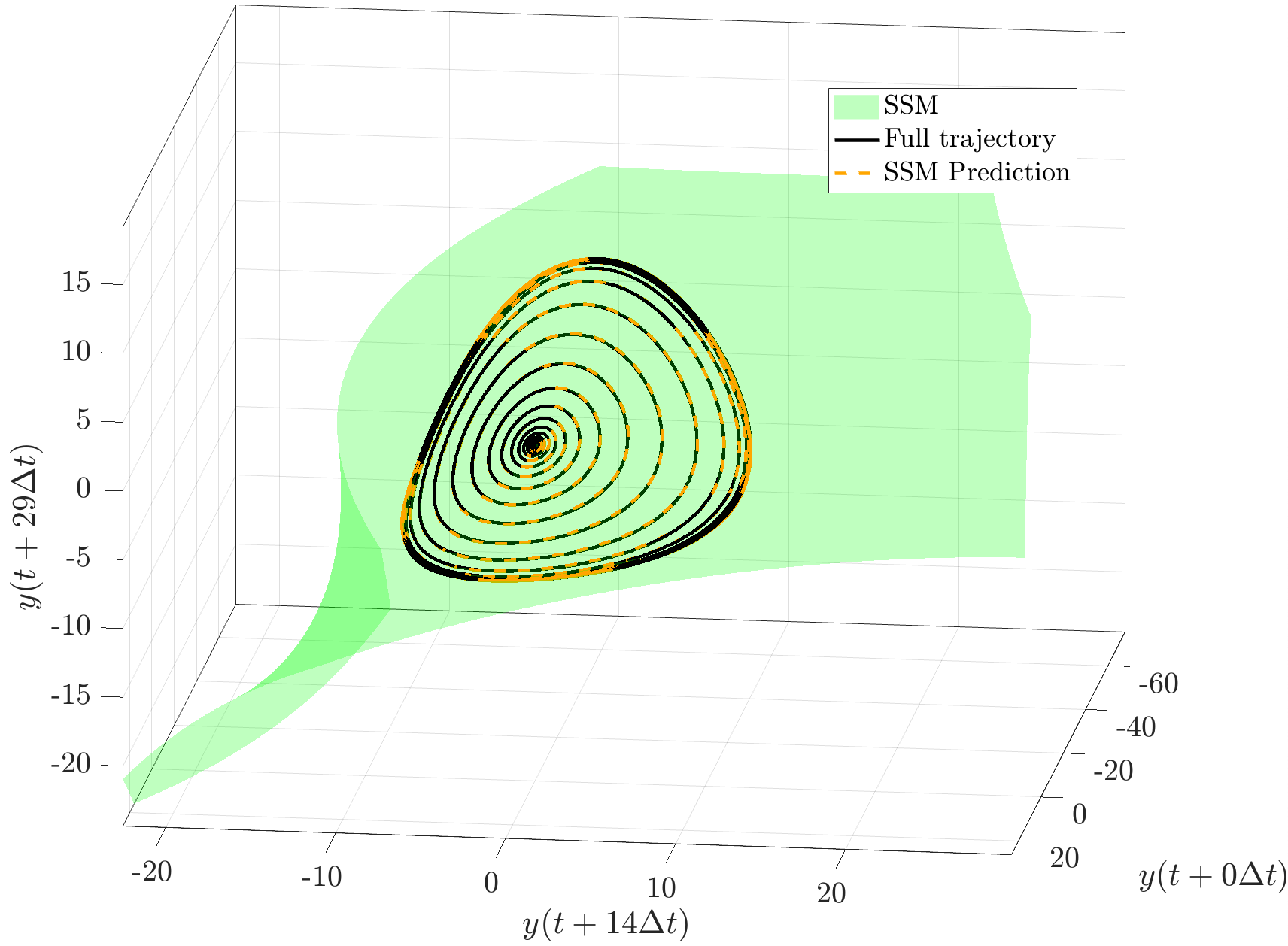}
  \end{minipage}\hfill
  \begin{minipage}[c]{\imgBwidth}
    \centering
    \includegraphics[width=\linewidth]{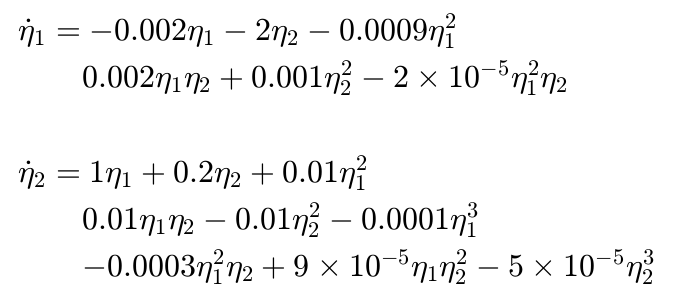}
  \end{minipage}
  \caption{Left: Comparison between a full system \eqref{eq:hutch1} test trajectory (black) together with its prediction (orange) on the data-driven SSM (green). Right: Polynomial of the SSM-reduced dynamics, with monomials whose coefficients are smaller than $10^{-5}$ omitted for plot clarity.}
  \label{fig:reducedEquations_H}
\end{figure}

The estimation of the spectral gap from data serves only to identify the smoothness class of $W$ and its associated reduced dynamics ensured by Theorem~\ref{thm:main1}. This requirement does not restrict the use of higher-order polynomial approximations for both the manifold and the reduced dynamics, since SSMLearn does not reconstruct the Taylor expansions produced by equation-driven methods, instead, it fits multivariate polynomials to data through regression.

\subsection{Chaotic attractor in a single discrete time delay system}
The Mackey--Glass system  \parencite{MackeyGlass1977}, a model for blood cell production, is described by the delay differential equation

\begin{equation}
\dot{x}(t) = \beta \, \frac{x(t-\tau)}{1 + x(t-\tau)^{\alpha}} - \gamma x(t),
\label{eq:mackey-glass}
\end{equation}
where $x(t) \in \mathbb{R}$, $\tau=1$ is a constant discrete delay, and the parameters values are $\beta=4.0$, $\gamma=2.0$, and $\alpha=9.6$ as in \parencite{BredaEtAl2025DDMethods}. The system admits an equilibrium at the origin. Despite being only one-dimensional, the system is known to exhibit chaotic behaviour for suitable parameter choices \parencite{MackeyGlass1977}, as the delay renders the phase space infinite-dimensional. This contrasts with finite-dimensional ODEs, where chaos requires a phase space of dimension at least three.

Following the steps outlined in Section \ref{sec:methodologies}, we first find the correlation dimension of the chaotic attractor to be $m \approx 2.2$ (Figure~\ref{fig:CorrDim_Eigenvalues_MG}).

Assuming that the chaotic attractor is contained in one of the SSMs whose existence in \(X\) is guaranteed by Theorem~\ref{thm:main1}, we seek $d=6$ dimensional SSM, following the procedure outlined in Section~\ref{sec:methodologies}. This choice is motivated by the oscillatory character of the dynamics and by the correlation dimension estimate \(m=2.2\). Indeed, \(d=6\) is the smallest even integer satisfying the sufficient embedding condition \(d>2m=4.4\). This choice is further supported by a preliminary identification of the rightmost eigenvalues in the complex plane, revealing three pairs of weakly damped oscillatory modes.

Accordingly, the minimum delay embedding dimension is $k=13$, which we further increase by $6$ to improve predictive accuracy on test trajectories. Finally, we verify that the estimated correlation dimension of the attractor A is preserved under projection from the delay embedding space to the reduced space (Figure~\ref{fig:CorrDim_Eigenvalues_MG}), confirming that no topological information about the attractor is lost.
\newcommand{\imgWWwidth}{0.49\textwidth} 
\newcommand{\imgWWWwidth}{0.49\textwidth} 
\begin{figure}[H] 
  \centering
  \begin{minipage}[c]{\imgWWwidth}
    \centering
    \includegraphics[width=\linewidth]{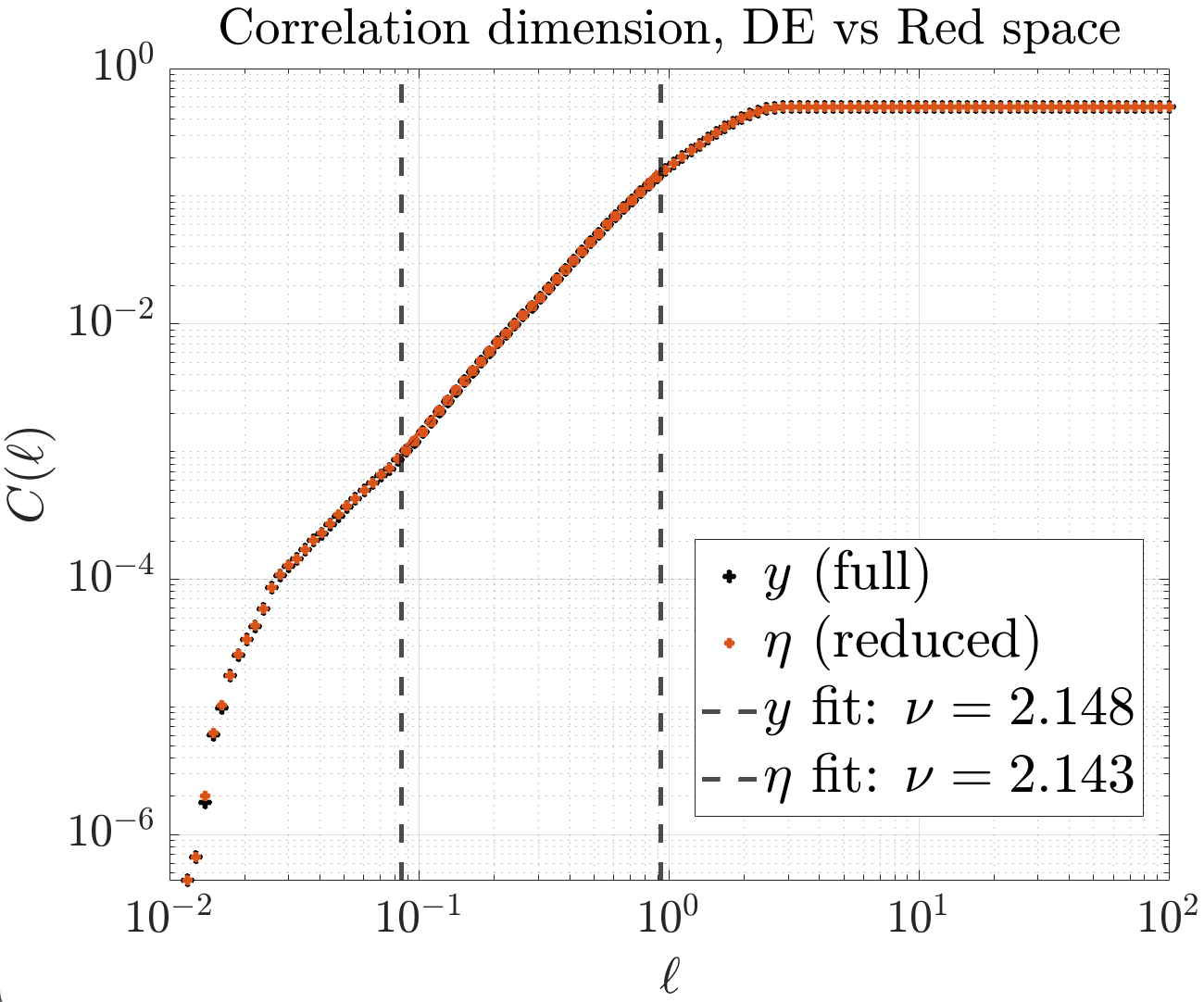}
    \label{fig:mp1a}
  \end{minipage}\hfill
  \begin{minipage}[c]{\imgWWWwidth}
    \centering
    \includegraphics[width=\linewidth]{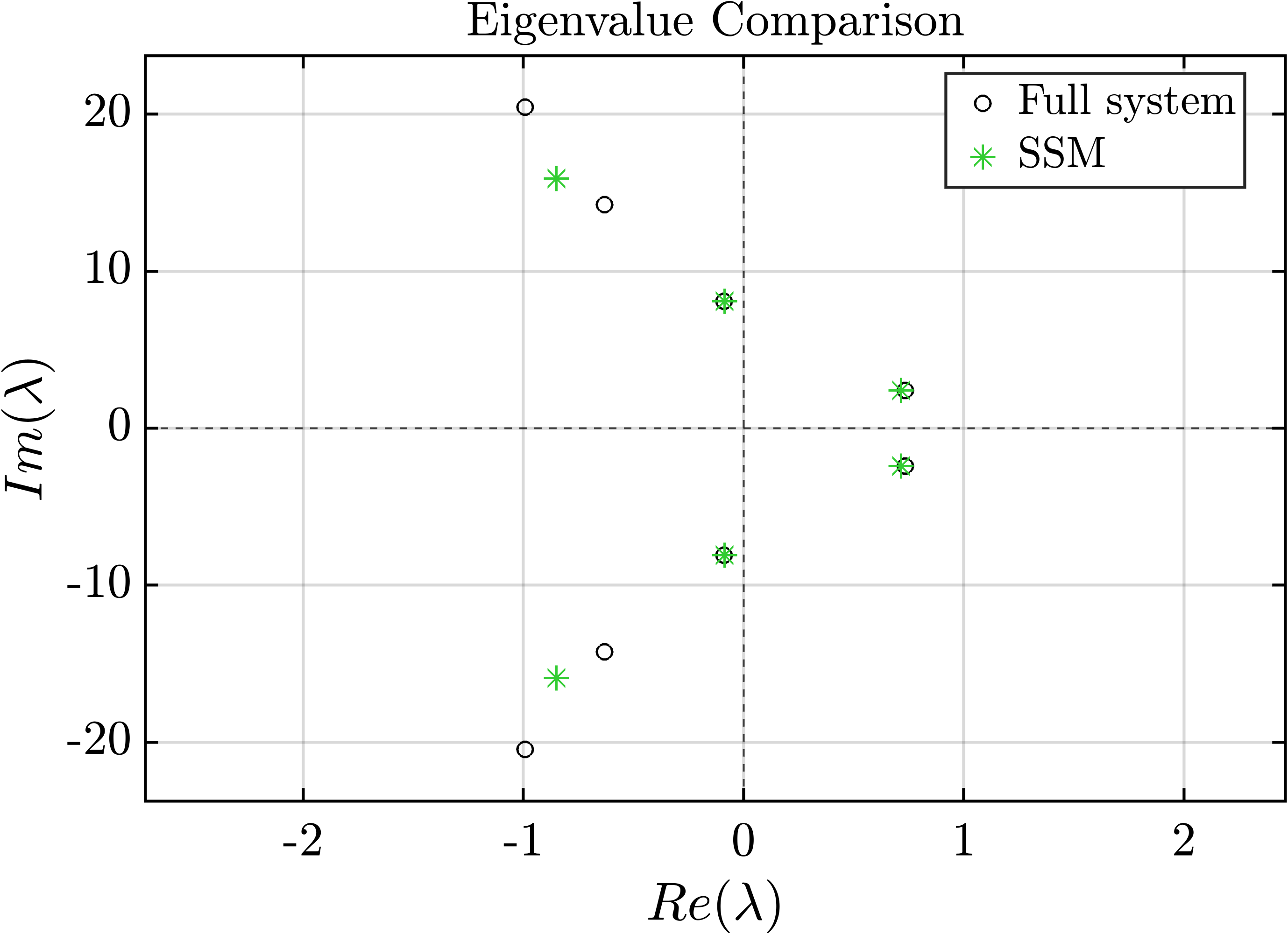}
    \label{fig:mp2a}
  \end{minipage}
  \caption{Correlation dimension estimation and spectrum validation for the Mackey–Glass system \eqref{eq:mackey-glass}. Left: log–log plot of the correlation sum versus radius \parencite{GrassbergerProcaccia1983}, showing a scaling region with slope $m \approx 2.2$. Right: comparison between eigenvalues of the reduced dynamics (green) and those from the full system obtained via root-finding  (black).
}
  \label{fig:CorrDim_Eigenvalues_MG}
\end{figure} 

We identify the reduced SSM dynamics by applying polynomial regression to delay-embedded trajectories converging to the chaotic attractor \(A\), after projecting them onto the tangent space. The graph-style parametrization of the SSM provides a one-to-one lift from the reduced coordinates to the delay-embedding space, thereby avoiding spurious self-intersections. This preserves the uniqueness of the predicted solutions, consistently with the well-posedness of smooth autonomous DDEs \parencite{Diekmann1995}.

The eigenvalues of the reduced dynamics, obtained from data, match those obtained via root-finding applied to the equation (Figure~\ref{fig:CorrDim_Eigenvalues_MG}). Additionally, by estimating eigenvalues from an 8D SSM, one can approximate the dominant spectral gap and conclude that 6D SSMs are of class $C^1$.

A data-driven 6D linear SSM with third-order polynomial reduced dynamics accurately reproduces key features of the attractor, including a physically relevant invariant measure and leading Lyapunov exponent. Although Theorem~\ref{thm:main1} guarantees only $C^1$ reduced dynamics, we employ a third-order polynomial model for two reasons: first, we seek a smooth ODE that best fits the data, without aiming to recover a Taylor expansion; second, linear ODE systems cannot exhibit chaotic dynamics.

Model performance is assessed via short-time trajectory prediction, as long-term accuracy is inherently limited by sensitivity to initial conditions. The positive leading Lyapunov exponent implies that even small projection errors grow exponentially under the flow. Exploiting the topological transitivity of $A$, two sufficiently long trajectories are used for training, while one is reserved for testing. Figure~\ref{fig:TrajTestTrain_MG} shows trajectory divergence over the first $30$ seconds, occurring on a timescale of order $1/\lambda_{\text{Lyap}}$.
A positive Lyapunov exponent can be approximated in the delay embedding space, as detailed in Section \ref{sec:methodologies}.

\begin{figure}[H] 
  \centering
  \begin{minipage}[c]{\imgWWwidth}
    \centering
    \includegraphics[width=\linewidth]{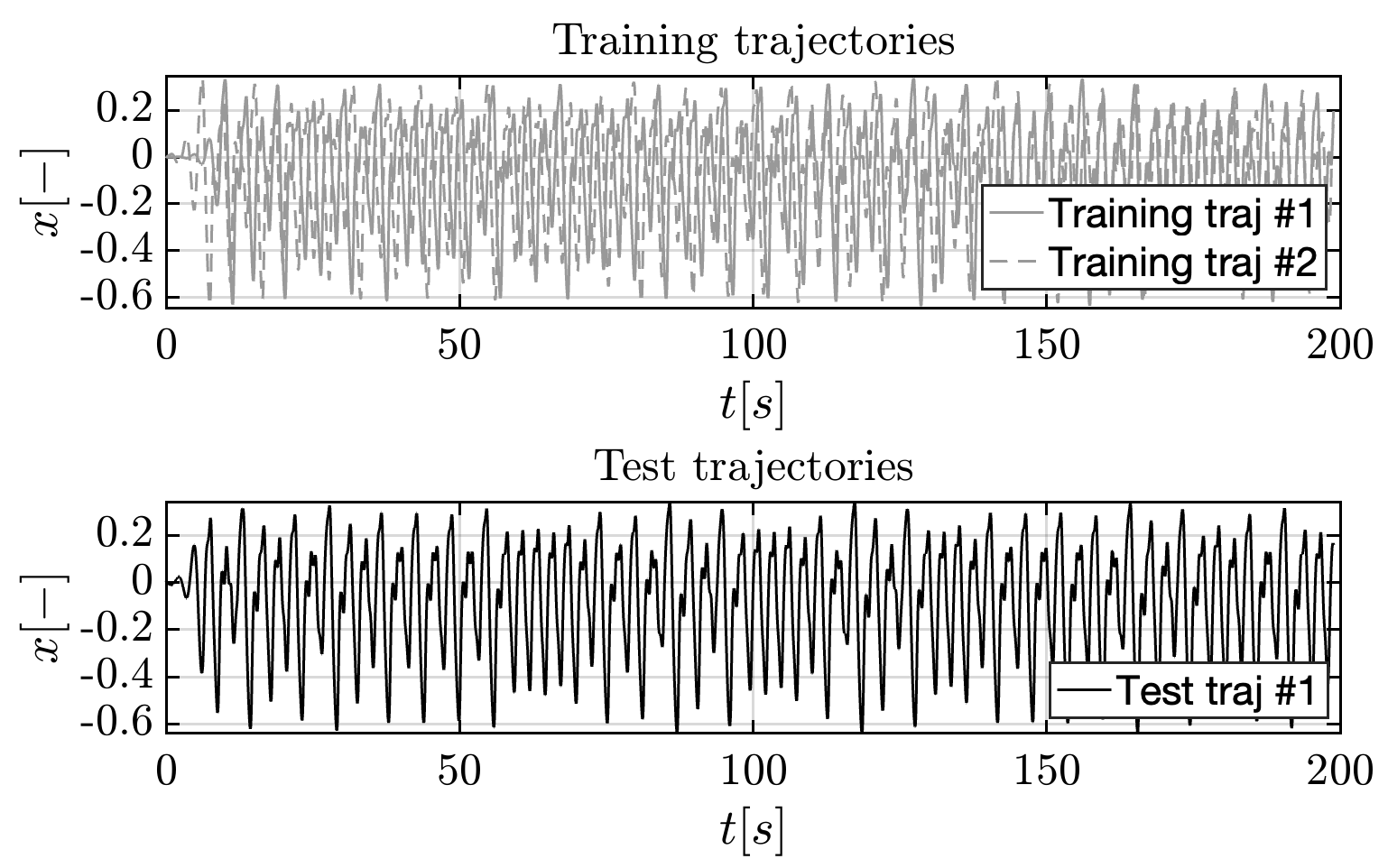}
    \label{fig:mp1b}
  \end{minipage}\hfill
  \begin{minipage}[c]{\imgWWWwidth}
    \centering
    \includegraphics[width=\linewidth]{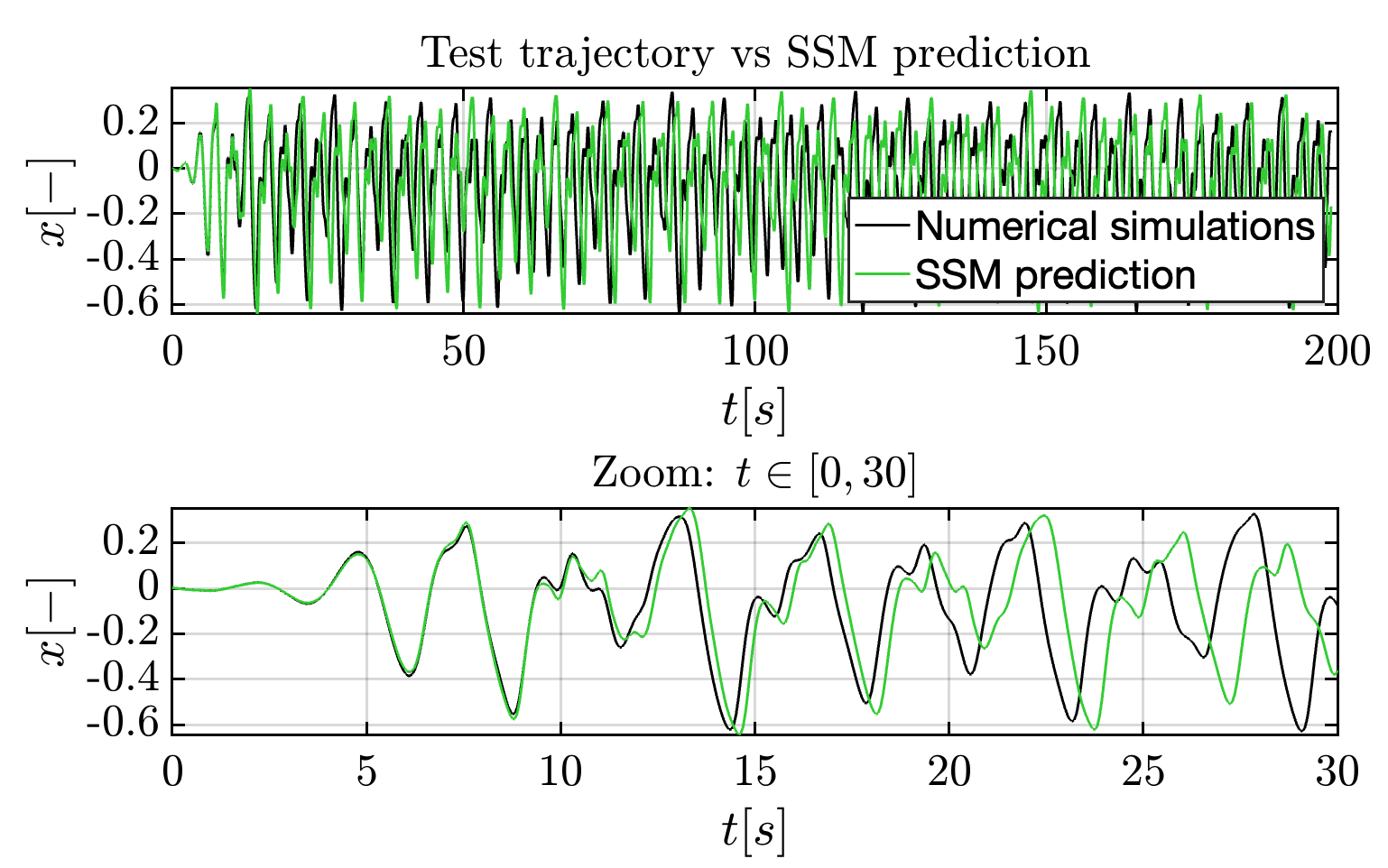}
    \label{fig:mp2b}
  \end{minipage}
  \caption{Short-time prediction performance of the SSM- reduced model of system \eqref{eq:mackey-glass}. Left: training and test trajectories in the observable $x(t)$ used for delay embedding. Right: comparison between full (black) and predicted (green) test trajectories, showing divergence only over a timescale consistent with the inverse leading Lyapunov exponent.
}
  \label{fig:TrajTestTrain_MG}
\end{figure}

A physically relevant invariant measure of the chaotic attractor can be estimated via probability density functions of the reduced coordinates and is accurately reconstructed from training trajectories (Figure~\ref{fig:PDF_LE_MG}). The leading Lyapunov exponent is estimated from data by fitting an exponential growth to the mean deviation of many trajectories, here we used $50$, initialized within an $\epsilon = 10^{-3}$ neighborhood of a point on the attractor. The initial condition for the full model is taken as a constant history function in $X$. This requires access to the full DDE, but only for validation purposes. In contrast, Lyapunov exponents can be estimated directly from the SSM-based reduced dynamics, as the data-driven model allows controlled initialization near the attractor, even when only data are available.

\newcommand{\imgWWWWwidth}{0.49\textwidth} 
\newcommand{\imgWWWWWwidth}{0.49\textwidth} 

\begin{figure}[H] 
  \centering
  \begin{minipage}[c]{\imgWWWWwidth}
    \centering
    \includegraphics[width=\linewidth]{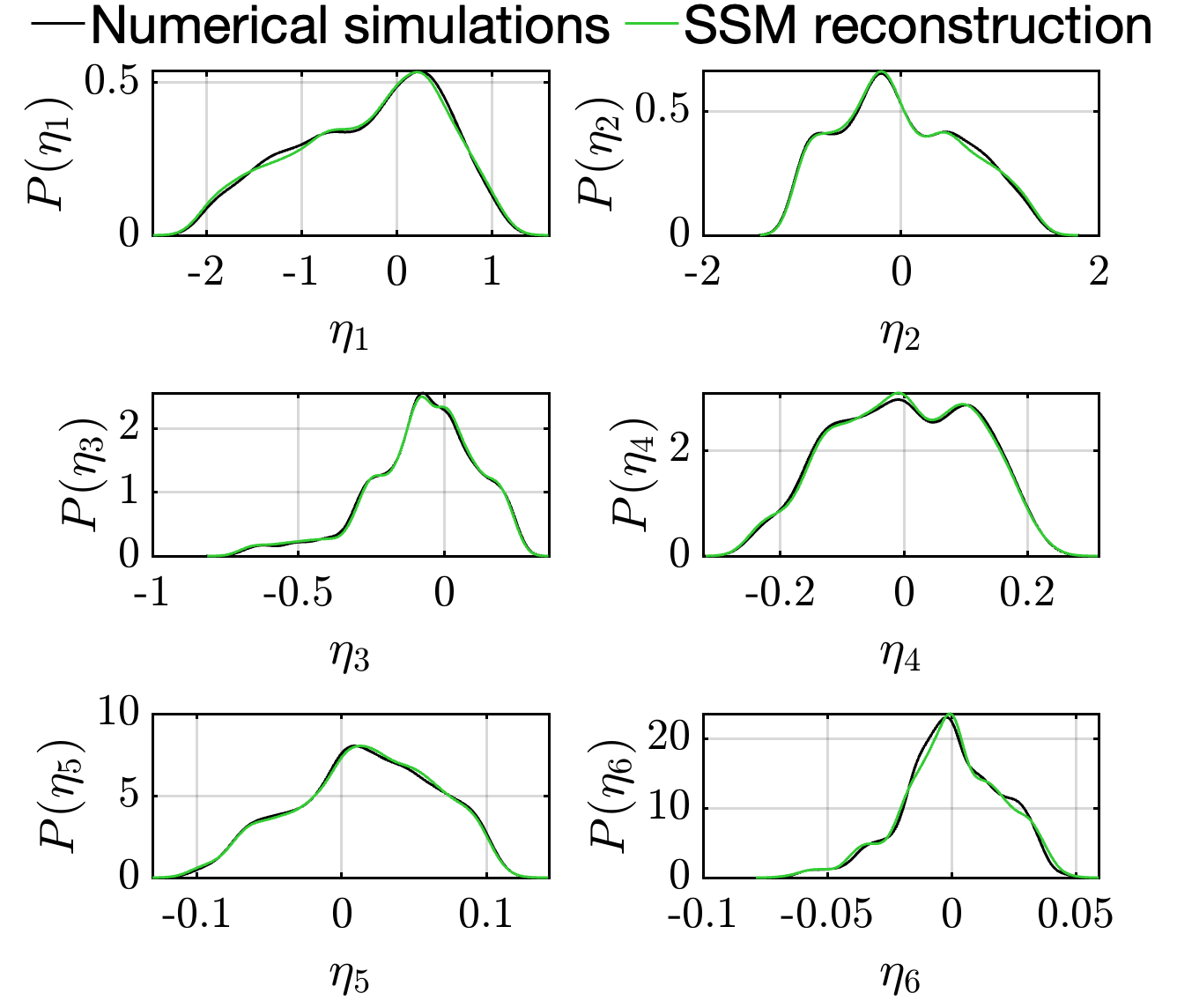}
    \label{fig:mp1c}
  \end{minipage}\hfill
  \begin{minipage}[c]{\imgWWWWWwidth}
    \centering
    \includegraphics[width=\linewidth]{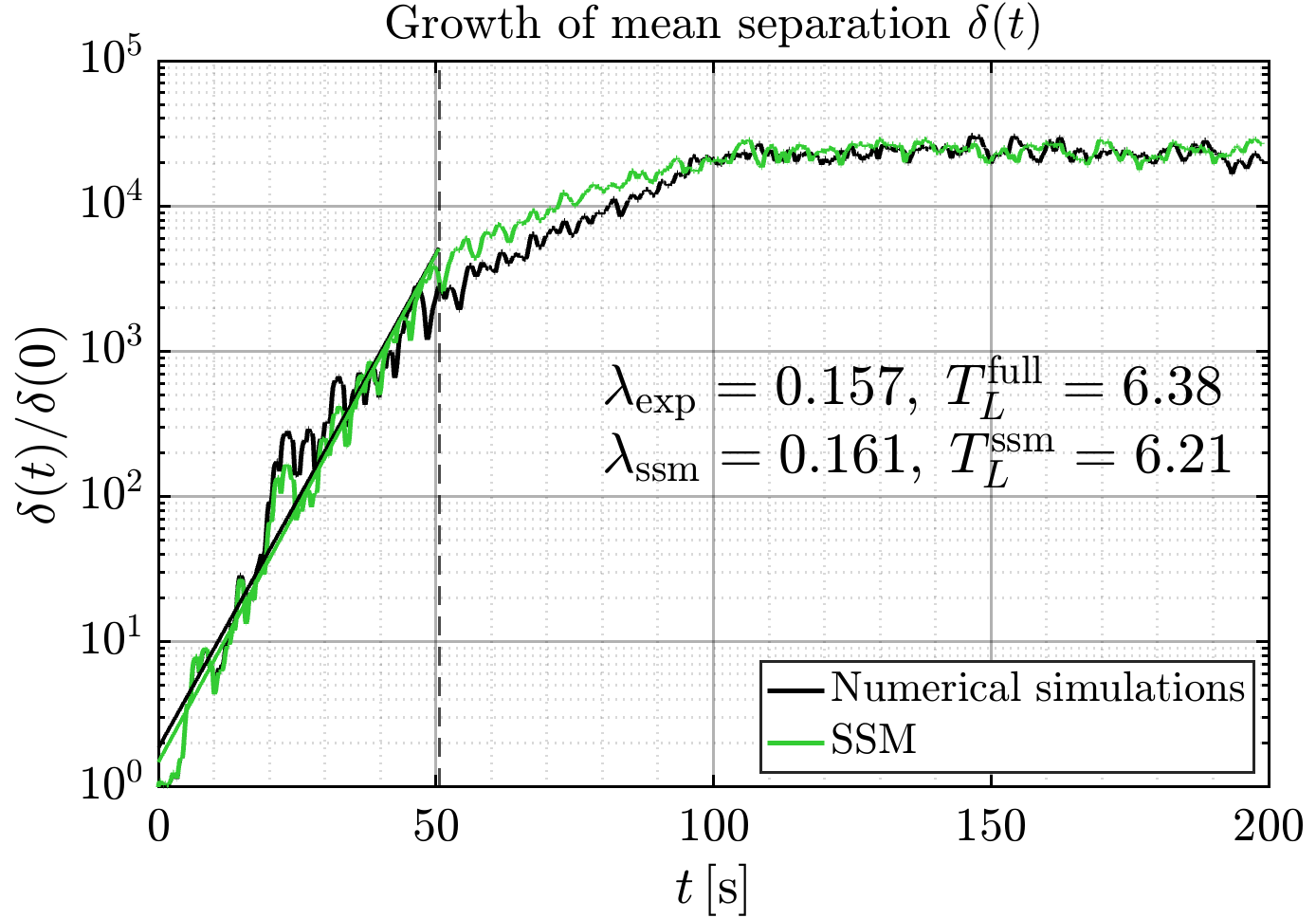}
    \label{fig:mp2c}
  \end{minipage}
  \caption{Analysis of the chaotic attractor of system \eqref{eq:mackey-glass} and its SSM-reduced model. Left: probability density functions of all the reduced coordinates based on training trajectories (black) and their reconstruction (green). 
Right: estimation of the leading Lyapunov exponent from trajectory mean separation over time $\delta(t)$ for the full system (black) and for the SSM-predictions (green).}
  \label{fig:PDF_LE_MG}
\end{figure} 

A 3D projection in the delay embedding space shows the test trajectory together with its prediction converging to the chaotic attractor from the equilibrium (Figure~\ref{fig:DE_MG}). The polynomial reduced dynamics for system~\eqref{eq:mackey-glass} is plotted in Figure~\ref{fig:DE_MG_redDyn} in Appendix~\ref{appendixD}.

\begin{figure}[H]
  \centering
  \includegraphics[width=0.5\textwidth]{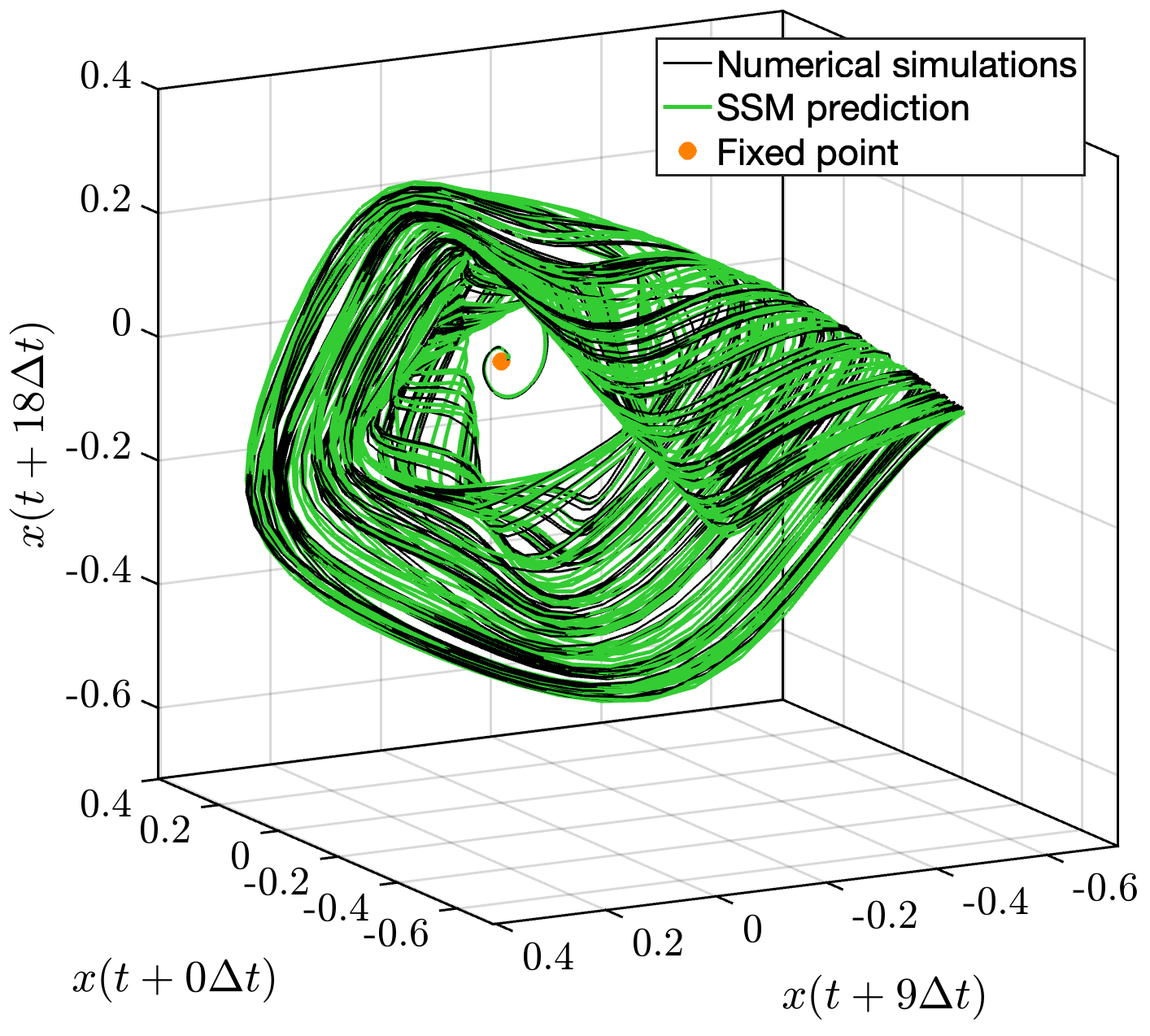}
  \caption{3D projection of the delay embedding space showing convergence of the test trajectory to the chaotic attractor for the full system \eqref{eq:mackey-glass} (black) and the SSM prediction (green).}
  \label{fig:DE_MG}
\end{figure}

We emphasize that, even in an equation-driven setting, Taylor expansions are not guaranteed to give good approximation of global invariant structures such as chaotic attractors, as the attractor can lie outside the domain of convergence of the series. Moreover, we highlight that polynomial parametrizations suffice to obtain predictive reduced models, without the need for rational terms and subsequent knowledge of the underlying equations, as required for the Mackey--Glass equation in \parencite{BredaEtAl2025DDMethods}.

\subsection{Limit cycle in a multiple discrete time delays system}
The two-neuron model with multiple discrete time delays \parencite{BredaEtAl2025DDMethods} is described by
\begin{equation}
\begin{aligned}
\dot{x}_1(t) &= -\kappa x_1(t) + \beta \tanh\big(x_1(t-\tau_s)\big) + a_{12}\tanh\big(x_2(t-\tau_2)\big), \\
\dot{x}_2(t) &= -\kappa x_2(t) + \beta \tanh\big(x_2(t-\tau_s)\big) + a_{21}\tanh\big(x_1(t-\tau_1)\big),
\end{aligned}
\label{eq:two-neuron}
\end{equation}
where $x_1(t), x_2(t) \in \mathbb{R}$, with delays $\tau_s = 1.5$, $\tau_1 = 2.0$, $\tau_2 = 2.0$ and parameters
$
\kappa = 0.5, \
\beta = -1, \
a_{12} = 1, \
a_{21} = 2.
$
The system admits the trivial equilibrium $(x_1^\ast, x_2^\ast) = (0,0)$. From the behavior of the training trajectories, we observe a 2D unstable manifold emanating from the equilibrium containing a stable limit cycle. Accordingly, we seek a 2D SSM associated with the dominant linear modes.

We follow the data-driven SSM-reduction outlined in Section \ref{sec:methodologies}. A 2D $C^\infty$ SSM exists for \eqref{eq:two-neuron} based on Theorem~\ref{thm:main1}.
We learn a cubic 2D manifold with a fifth-order reduced dynamics from the scalar observable $x_1$ evolving on six training trajectories. The dominant eigenvalues are accurately captured by the reduced-order model, as shown in Figure~\ref{fig:Eigenvalues_2N}.

\newcommand{\imgTwoNLeft}{0.34\textwidth} 
\newcommand{\imgTwoNRight}{0.64\textwidth} 

\begin{figure}[H] 
  \centering
  \begin{minipage}[c]{\imgTwoNLeft}
    \centering
    \includegraphics[width=\linewidth]{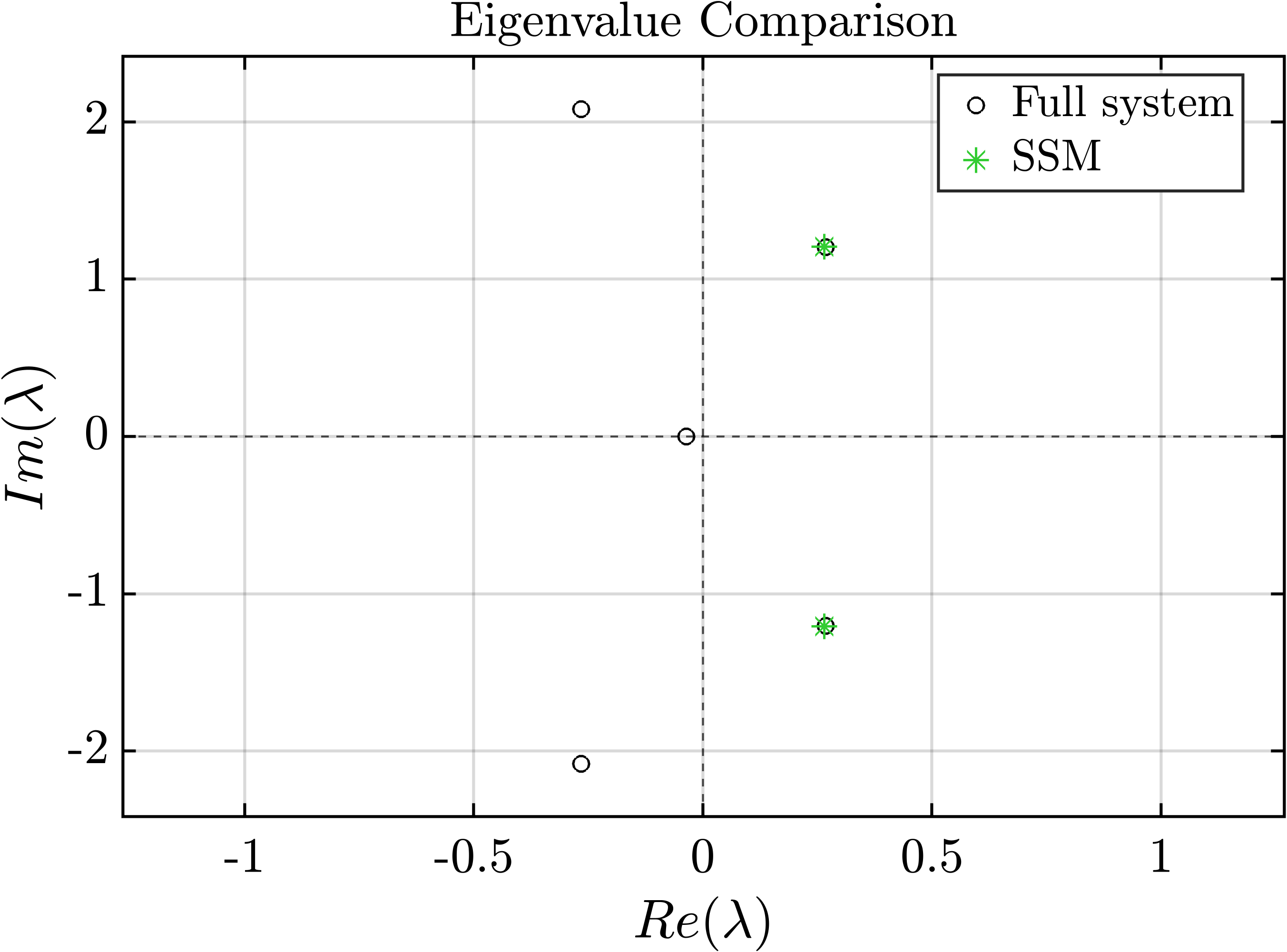}
  \end{minipage}\hfill
  \begin{minipage}[c]{\imgTwoNRight}
    \centering
    \includegraphics[width=\linewidth]{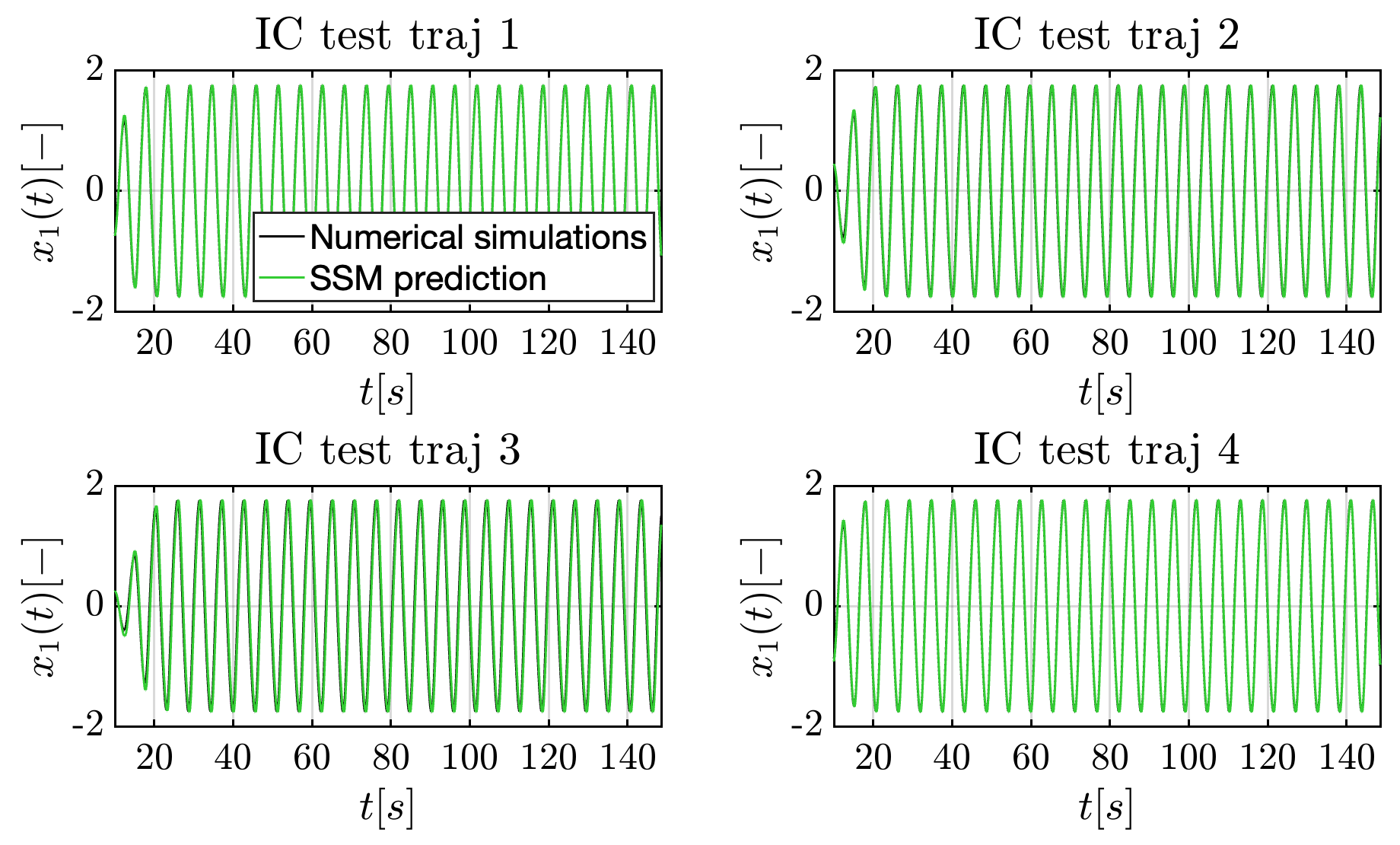}
  \end{minipage}
  \caption{
Left: Comparison between eigenvalues of the data-driven reduced dynamics (green) and those obtained from the full system \eqref{eq:two-neuron} (black) via root-finding.
Right: Four test trajectories (black) and their SSM-based predictions (green).}
  \label{fig:Eigenvalues_2N}
\end{figure}

Four test trajectories of system \eqref{eq:two-neuron}, together with their predictions from the SSM-reduced model, are shown in Figure~\ref{fig:Eigenvalues_2N}, yielding a NMTE of $4.8\%$. A 3D projection of the delay embedding space is also reported in Figure ~\ref{fig:2N_TrajPred}.

\newcommand{\imgTwoNBisLeft}{0.49\textwidth} 
\newcommand{\imgTwoNBisRight}{0.49\textwidth} 

\begin{figure}[H] 
  \centering
  \begin{minipage}[c]{\imgTwoNBisLeft}
    \centering
    \includegraphics[width=1.1\linewidth]{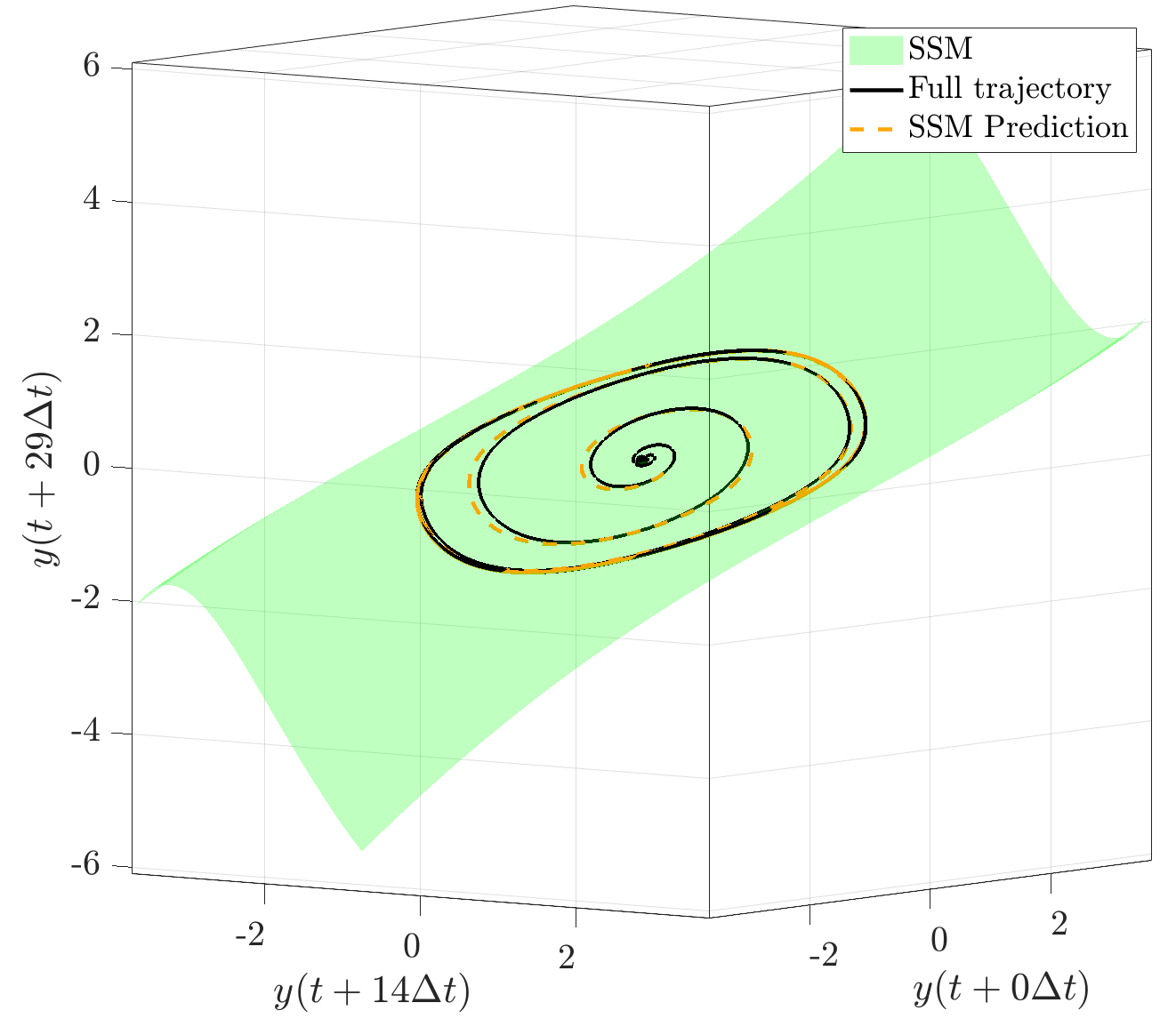}
  \end{minipage}\hfill
  \begin{minipage}[c]{\imgTwoNBisRight}
    \centering
    \includegraphics[width=0.8\linewidth]{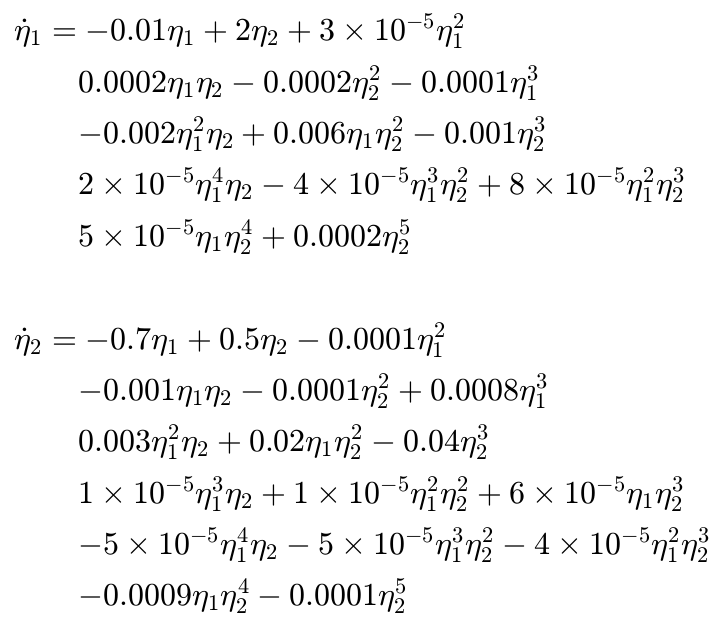}
  \end{minipage}
  \caption{Data-driven SSM reduction for system \eqref{eq:two-neuron}. Left: 3D projection of the delay embedding space showing an SSM (green), with a test trajectory (black) and its prediction (orange).
Right: Learned 5th order polynomial SSM-reduced dynamics. Monomials whose coefficients are smaller than $10^{-5}$ are omitted for plot clarity.}
  \label{fig:2N_TrajPred}
\end{figure}

We observe that data-driven SSM reduction retains predictive power despite the presence of multiple discrete delays, both in number and magnitude, which are assumed unknown throughout the reduction procedure. We next demonstrate that this performance extends to systems exhibiting chaotic attractors.

\subsection{Chaotic attractor in a multiple discrete delays system}

We now consider the system \parencite{BredaEtAl2025DDMethods} 
\begin{equation}
\begin{aligned}
\dot{x}_1(t) &= -x_2(t)-x_3(t)+\alpha_1 x_1(t-\tau_1)+\alpha_2 x_1(t-\tau_2), \\
\dot{x}_2(t) &= x_1(t)+\beta_1 x_2(t), \\
\dot{x}_3(t) &= \beta_2 + x_3(t)x_1(t)-\gamma x_3(t),
\end{aligned}
\label{eq:rossler-double-delay}
\end{equation}
where \(x_1(t), x_2(t), x_3(t) \in \mathbb{R}\). The delays,
\(\tau_1 = 1.0\) and \(\tau_2 = 2.0\), and the parameter values
$
\alpha_1=0.2,\
\alpha_2=1.0,\
\beta_1=0.2,\
\beta_2=0.2,\
\gamma=1.2.$ 
are taken from \parencite{BredaEtAl2025DDMethods}.
The system admits a nontrivial equilibrium, which we shift to the origin. This system exhibits a chaotic attractor resembling that of the R\"ossler system. As in the Mackey--Glass example, we estimate the correlation dimension of the attractor as $m \approx 2.17$ (Figure~\ref{fig:CorrDim_Eigenvalues_RO}). 

Assuming again that the chaotic attractor is contained in one of the SSMs that exist in $X$ by Theorem \ref{thm:main1}, and given the oscillatory nature of the dynamics, we seek a $d=6$ dimensional SSM following Section \ref{sec:methodologies}. Since $d > 2m$, this choice guarantees that the attractor can be embedded in the reduced space within the delay-coordinate space \parencite{SauerYorkeCasdagli1991}. We construct a delay embedding space of dimension $k=23$ using the delay-coordinate map, exceeding the minimum requirement $k=13$ to improve model accuracy.

The dominant SSM is approximated using first-order polynomials for the manifold, while the reduced dynamics are learned via radial basis functions (RBFs), as polynomial approximations are insufficiently accurate, as already noted in \parencite{KaszasHaller2025, AbbascianoEndreszStepanHaller2026JSV, Xu2024}. For the purpose of estimating the Lyapunov exponent from data, we regard the time series predicted by the RBF-based discrete map as a trajectory of an underlying continuous-time dynamical system, sampled at time instants separated by the discrete-map time step.

\newcommand{\imgROLeft}{0.49\textwidth} 
\newcommand{\imgRORight}{0.49\textwidth} 

\begin{figure}[H] 
  \centering
  \begin{minipage}[c]{\imgROLeft}
    \centering
    \includegraphics[width=\linewidth]{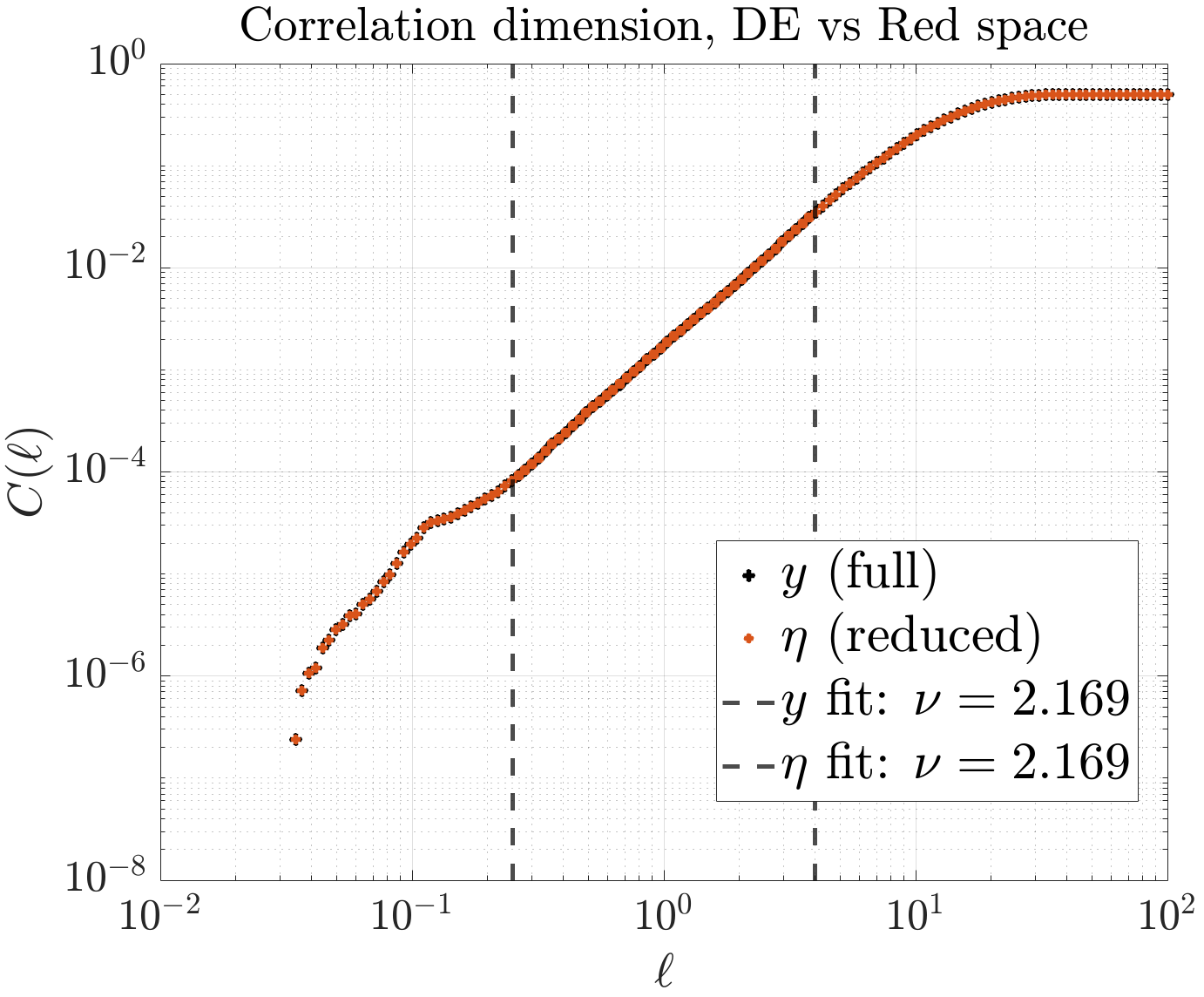}
  \end{minipage}\hfill
  \begin{minipage}[c]{\imgRORight}
    \centering
    \includegraphics[width=\linewidth]{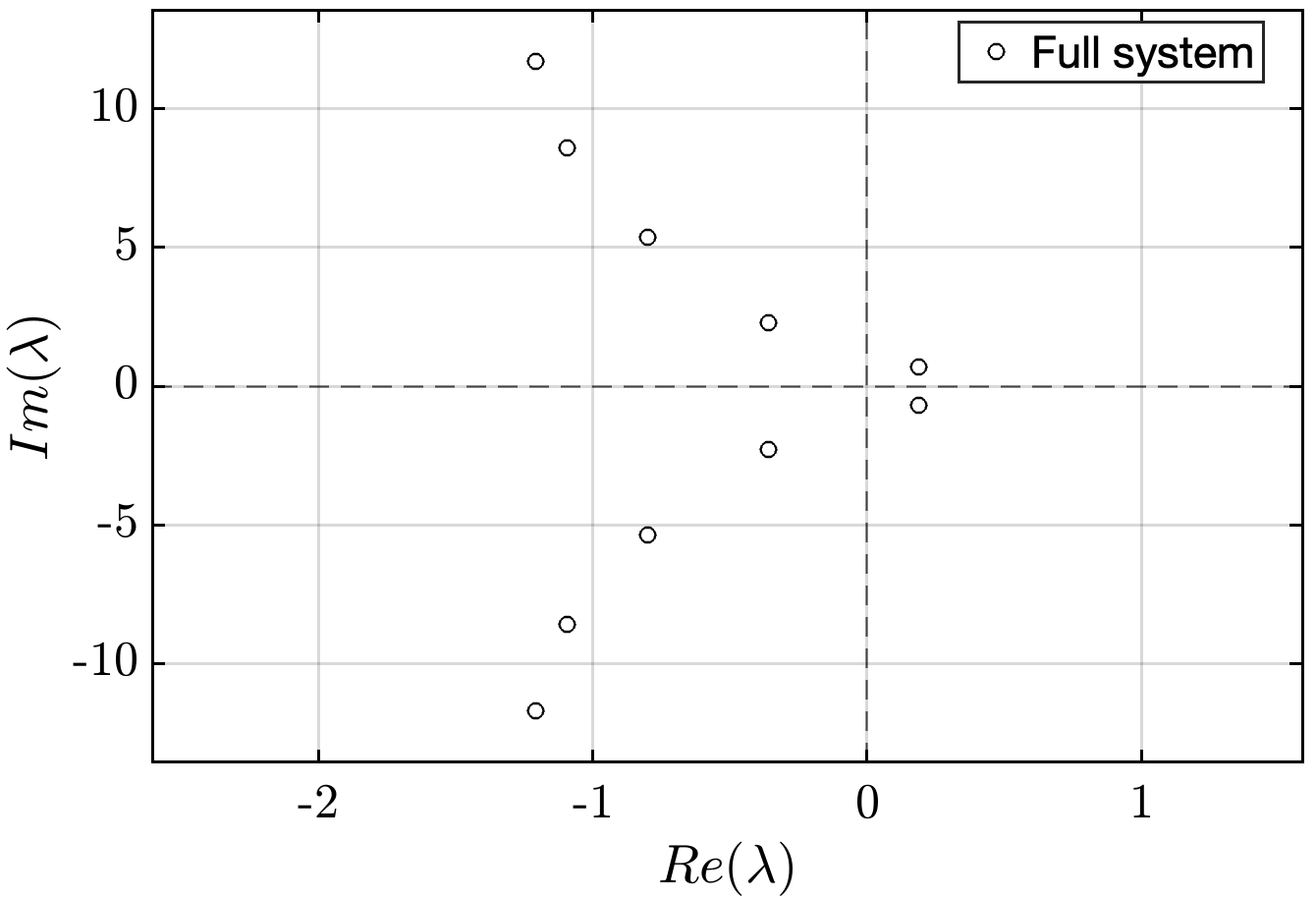}
  \end{minipage}
  \caption{
Left: Correlation dimension estimation for the attractor of system \eqref{eq:rossler-double-delay} from data in the delay embedding space (black) and in the reduced space (orange), showing agreement in correlation dimension estimation around $m \approx 2.17$. 
Right: Eigenvalues of the full system \eqref{eq:rossler-double-delay} obtained via root-finding \parencite{AppeltansSilmMichiels2022, AppeltansMichiels2023}.}
  \label{fig:CorrDim_Eigenvalues_RO}
\end{figure}

Two of the four training trajectories, together with predictions for four test trajectories, are shown in Figure~\ref{fig:RO_Traj}. As expected for chaotic dynamics, short-term predictions are accurate, while long-term prediction is limited by sensitivity to initial conditions.

\newcommand{\imgROTrajLeft}{0.49\textwidth} 
\newcommand{\imgROTrajRight}{0.49\textwidth} 

\begin{figure}[H] 
  \centering
  \begin{minipage}[c]{\imgROTrajLeft}
    \centering
    \includegraphics[width=\linewidth]{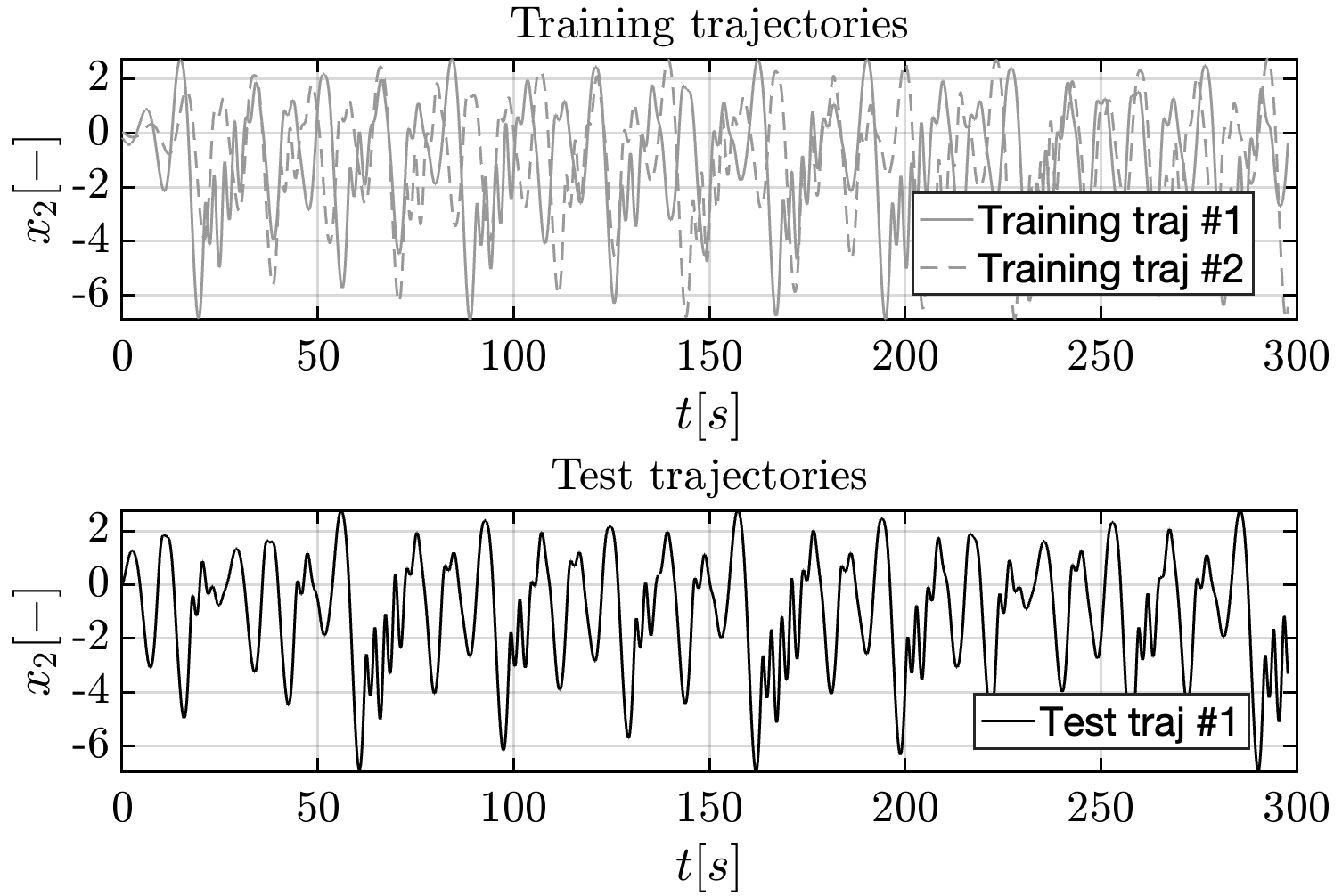}
  \end{minipage}\hfill
  \begin{minipage}[c]{\imgROTrajRight}
    \centering
    \includegraphics[width=\linewidth]{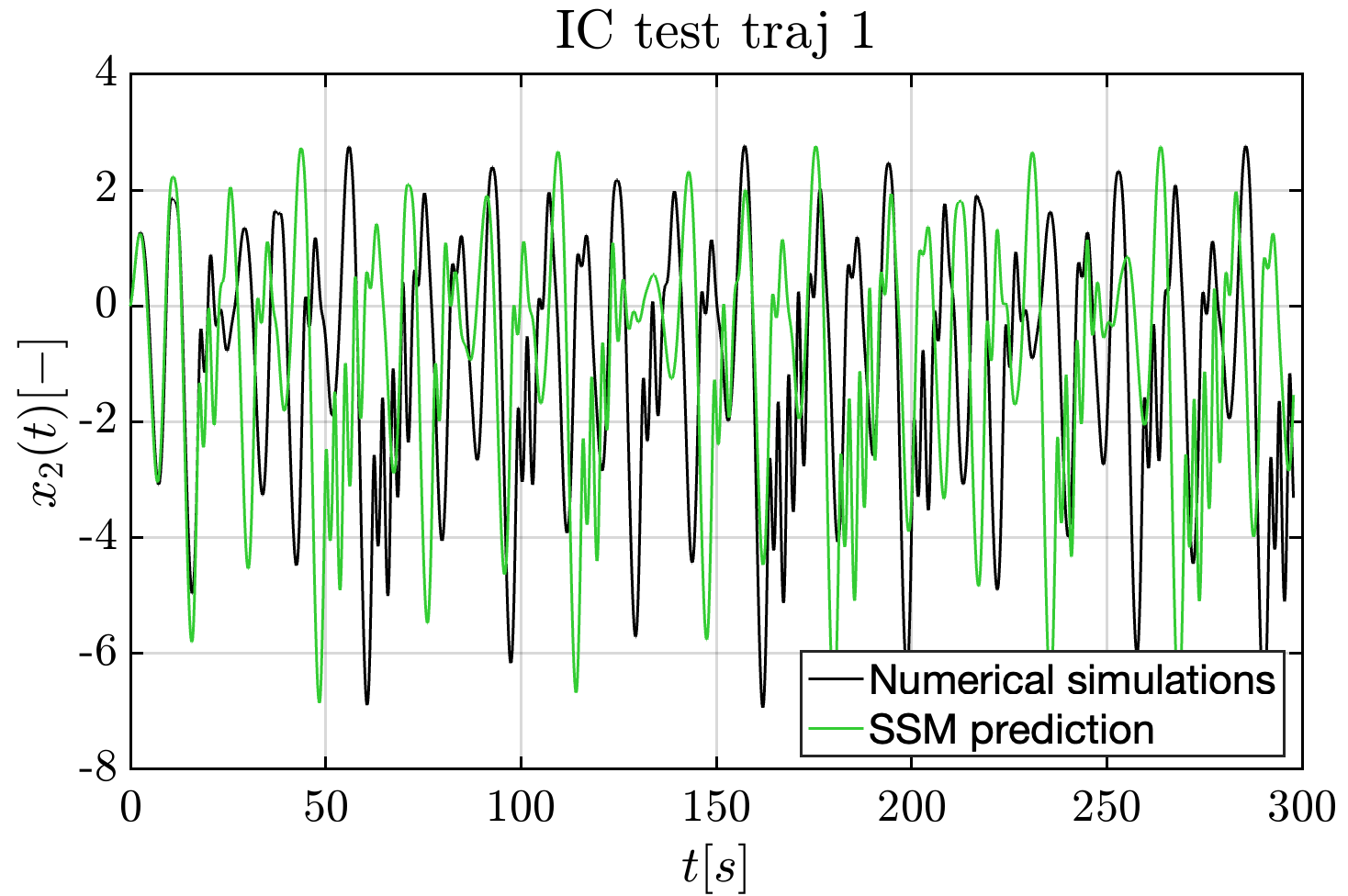}
  \end{minipage}
  \caption{
Left: Training and test trajectories for the SSM-reduction of system \eqref{eq:rossler-double-delay} in the observable $x_2$. 
Right: Predictions (green) of a full system's test trajectory (black), with divergence at longer times due to sensitivity to initial conditions.}
  \label{fig:RO_Traj}
\end{figure}

A physically relevant invariant measure and the leading Lyapunov exponent are estimated from data and reported in Figure~\ref{fig:PDF_LE_RO}.

\newcommand{\imgROPDFLeft}{0.49\textwidth} 
\newcommand{\imgROPDFRight}{0.49\textwidth} 

\begin{figure}[H] 
  \centering
  \begin{minipage}[c]{\imgROPDFLeft}
    \centering
    \includegraphics[width=\linewidth]{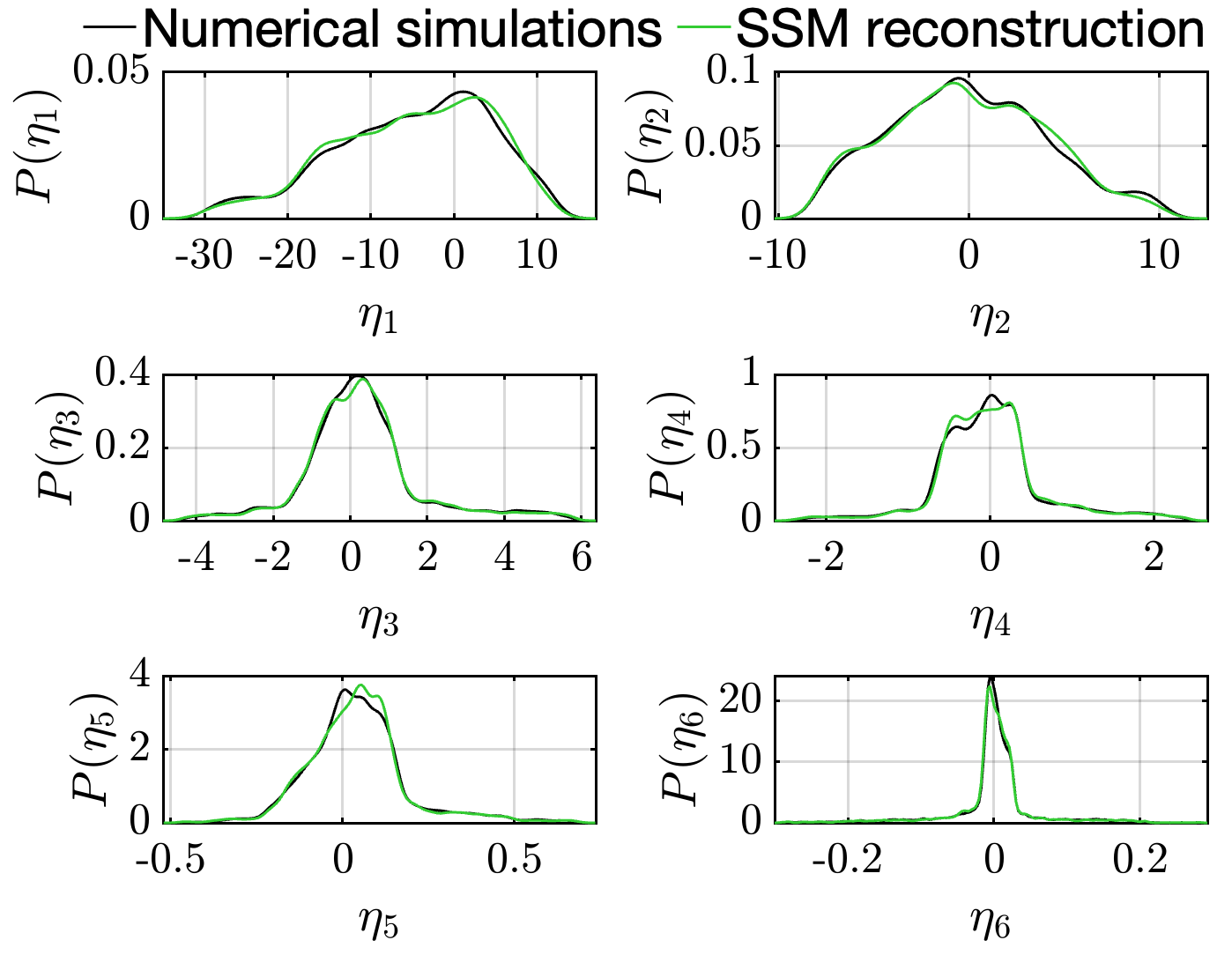}
  \end{minipage}\hfill
  \begin{minipage}[c]{\imgROPDFRight}
    \centering
    \includegraphics[width=\linewidth]{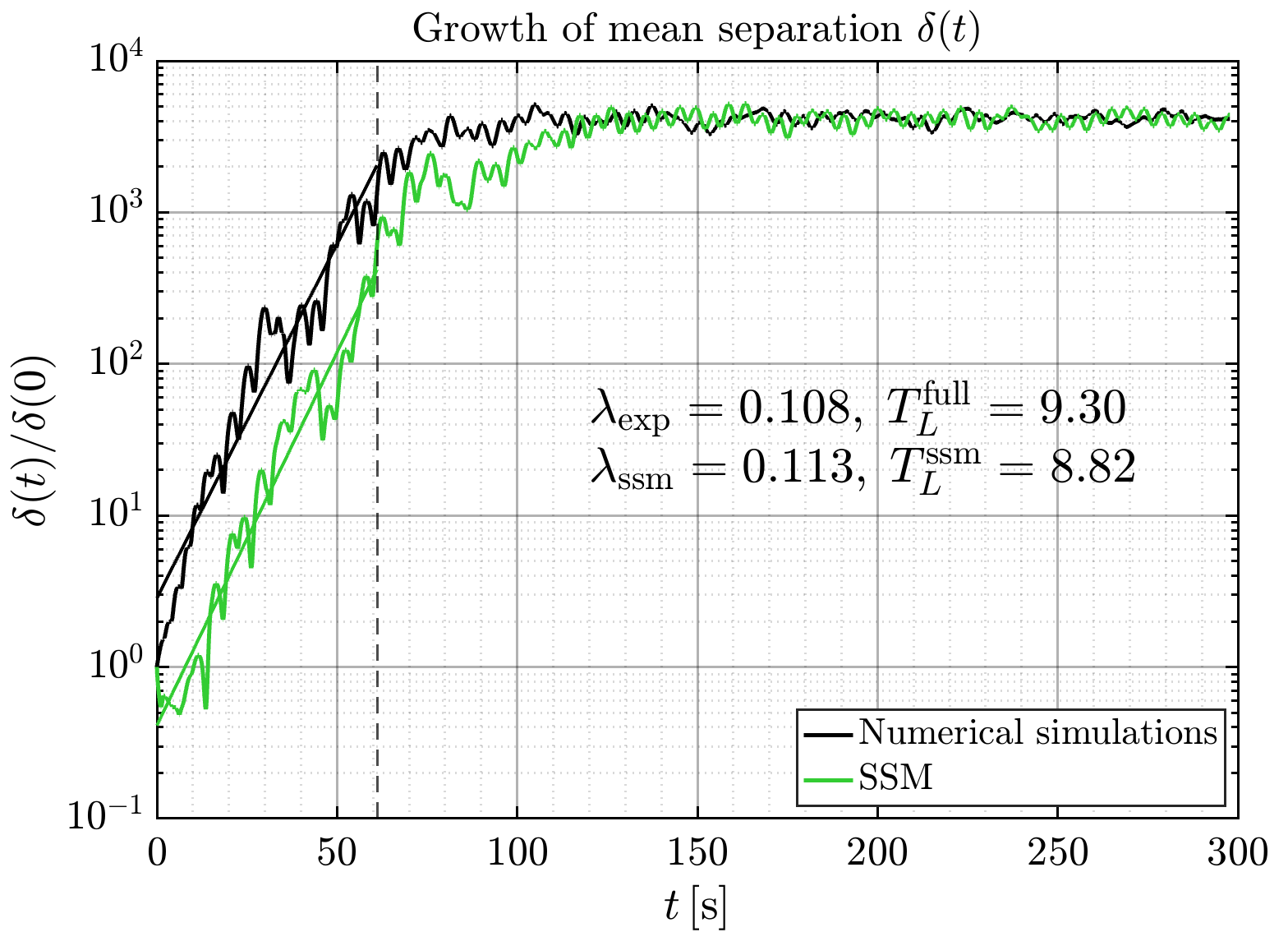}
  \end{minipage}
  \caption{Analysis of the attractor of system \eqref{eq:rossler-double-delay}. Left: Probability density functions of the reduced coordinates, showing agreement between full system's training trajectories (black) and their SSM-based reconstruction (green).
Right: Estimation of the leading Lyapunov exponent based on trajectory mean separation over time $\delta(t)$, comparing the full system (black) and the SSM-based prediction (green).}
  \label{fig:PDF_LE_RO}
\end{figure}

A 3D projection of the delay embedding space (Figure~\ref{fig:RO_3D}) illustrates the attractor and the convergence from the unstable equilibrium of a test trajectory and its prediction.

\newcommand{\imgRO}{0.6\textwidth} 

\begin{figure}[H]
  \centering
  \includegraphics[width=\imgRO]{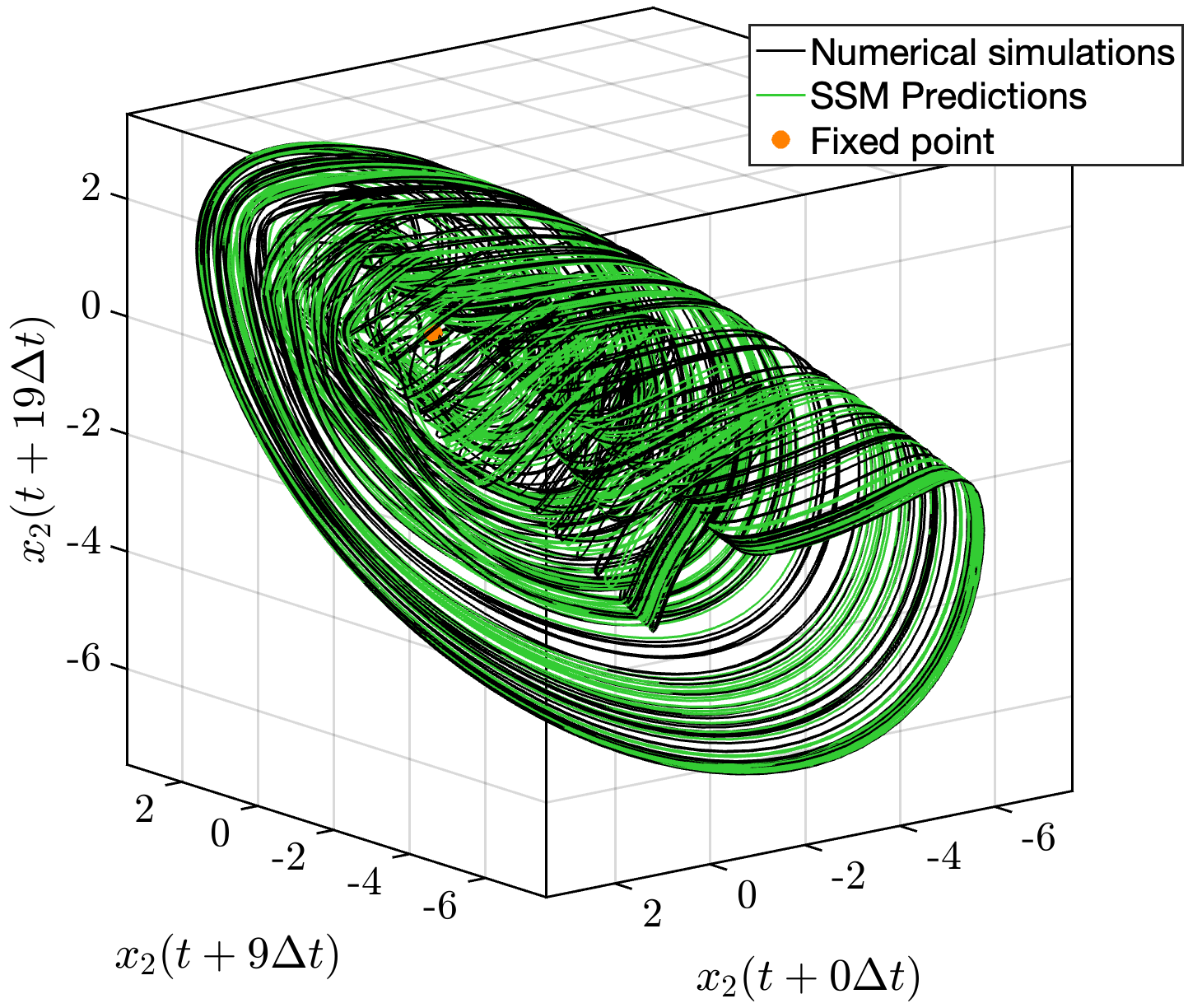}
  \caption{3D projection of the delay embedding space showing the chaotic attractor of system \eqref{eq:rossler-double-delay} and convergence of four test trajectories toward it, for both the full system (black) and the SSM-reduced model (green).}
  \label{fig:RO_3D}
\end{figure}

\subsection{Parametric reduced model for delayed car-following system}
We consider a time-delay traffic model describing the interaction between a human-driven vehicle (HV) and an autonomous vehicle (AV) \parencite{SzakszOroszStepan2024Traffic}, whose experimental relevance for autonomous driving control has been demonstrated in \parencite{OroszMolnar2025ConnectedVehicles}. The system is governed by
\begin{equation}
\begin{aligned}
\dot{\tilde h}(t) &= \tilde v(t)-\tilde v_{-1}(t),\\
\dot{\tilde v}_{-1}(t) &= \alpha\Big(V\big(\tilde h(t-\tau)+h^\ast\big)-\big(\tilde v_{-1}(t-\tau)+v_{\mathrm{ref}}\big)\Big)
+\beta\Big(\tilde v(t-\tau)-\tilde v_{-1}(t-\tau)\Big),\\
\dot{\tilde v}(t) &= \hat{\beta}\Big(v_{\mathrm{ref}}-\big(\tilde v(t-\tau)+v_{\mathrm{ref}}\big)\Big)
+\beta_{-1}\Big(\tilde v_{-1}(t-\tau)-\tilde v(t-\tau)\Big),
\end{aligned}
\label{eq:traffic-dde}
\end{equation}
where $\tilde h$ is the shifted headway between the two vehicles, and $\tilde v_{-1}$ and $\tilde v$ are the shifted velocities of the HV and AV relative to the reference speed $v_{\mathrm{ref}}$. The delay $\tau$ represents the reaction time delay in the control loop of both the HV and the AV. The gains $\alpha$ and $\beta$ describe the HV response, while $\hat{\beta}$ and $\beta_{-1}$ play the analogous role in the AV controller, accounting for speed tracking and velocity-difference feedback. We use the same parameters as in \parencite{SzakszOroszStepan2024Traffic},
\[
\alpha=0.3,\quad
\beta=0.4,\quad
\hat{\beta}=0.6,\quad
\beta_{-1}=-0.4,\quad
v_{\mathrm{ref}}=26.55,
\]
with
\[
h_{\mathrm{stop}}=5,\quad
h_{\mathrm{go}}=55,\quad
v_{\max}=30,
\]
and $h^\ast$ defined implicitly by $V(h^\ast)=v_{\mathrm{ref}}$.

We perform parameter-dependent SSM reduction for the traffic delay system with the equilibrium undergoing a Hopf bifurcation. Such reductions have been shown to accurately capture even global phenomena, including heteroclinic bifurcations of limit cycles \parencite{AbbascianoEndreszStepanHaller2026JSV}, as well as in fluid systems undergoing Hopf bifurcation \parencite{KingEtAl2026ParametricSSM}. We remark that in the DDE setting, we do not focus on uniqueness, since a smoothest nonlinear continuation of the SSM, called the primary SSM in \parencite{Haller2016}, is generally not guaranteed to exist.

Before and after the supercritical Hopf bifurcation, we identify from data, at each fixed value of the delay parameter, one SSM among the infinitely many SSMs. We then interpolate the tangent space, the polynomial parametrization of the manifold, and the polynomial parametrization of the reduced dynamics across the sampled parameter values. This yields a parametric SSM model in the delay-embedding space augmented with the bifurcation parameter, following \parencite{AbbascianoEndreszStepanHaller2026JSV}. Each autonomous SSM reduction at a fixed value of the delay parameter is carried out according to the procedure detailed in Section~\ref{sec:methodologies}. The manifold and the dynamics are not necessarily smooth with respect to the delay parameter; nevertheless, their dependence on the delay can be approximated via smooth spline interpolation, rather than simple linear interpolation, when it provides better accuracy.

Resonances are less disruptive in the data-driven setting, as we do not compute Taylor expansions of the solutions to the invariance PDE, but approximate them directly from data. Nevertheless, they still provide useful guidance: if a good approximation of the manifold cannot be learned, we will restrict the parameter range or reduce the polynomial order, as lower-order approximations cut out resonance-induced differences, making different SSMs less distinguishable, and being that the resonances at lower order occur further away from the critical value of the parameter, as noted in \parencite{KingEtAl2026ParametricSSM}. 

The parametric SSM model is trained at four parameter values, $\tau = \{1.030\,\text{s},\; 1.050\,\text{s},\; 1.070\,\text{s},\; 1.090\,\text{s}\}$, by constructing individual autonomous SSM-reduced models. Each model is trained on $8$ trajectories using only the segments lying approximately on the manifold and evaluated on $2$ unseen trajectories. The results (see Fig.~\ref{fig:Traffic_Parametric}) demonstrate accurate reconstruction across unseen parameter values, with an average NMTE below $1\%$, even as a supercritical Hopf bifurcation occurs at the origin.

We employ a 2D SSM with a third-order polynomial parametrization of the manifold and fifth-order polynomial reduced dynamics. The delay-embedding space has dimension $7$, using all three state variables as observables. We remark that, because interpolation is used to construct the parametric SSM model, all individual SSM reductions must share the same SSM dimension, delay-embedding dimension, SSM approximation order, and reduced-dynamics approximation order.

\newcommand{\imgXYwidth}{0.48\textwidth}
\begin{figure}[H]
\centering

\begin{minipage}{\imgXYwidth}
    \centering
    \includegraphics[width=\linewidth]{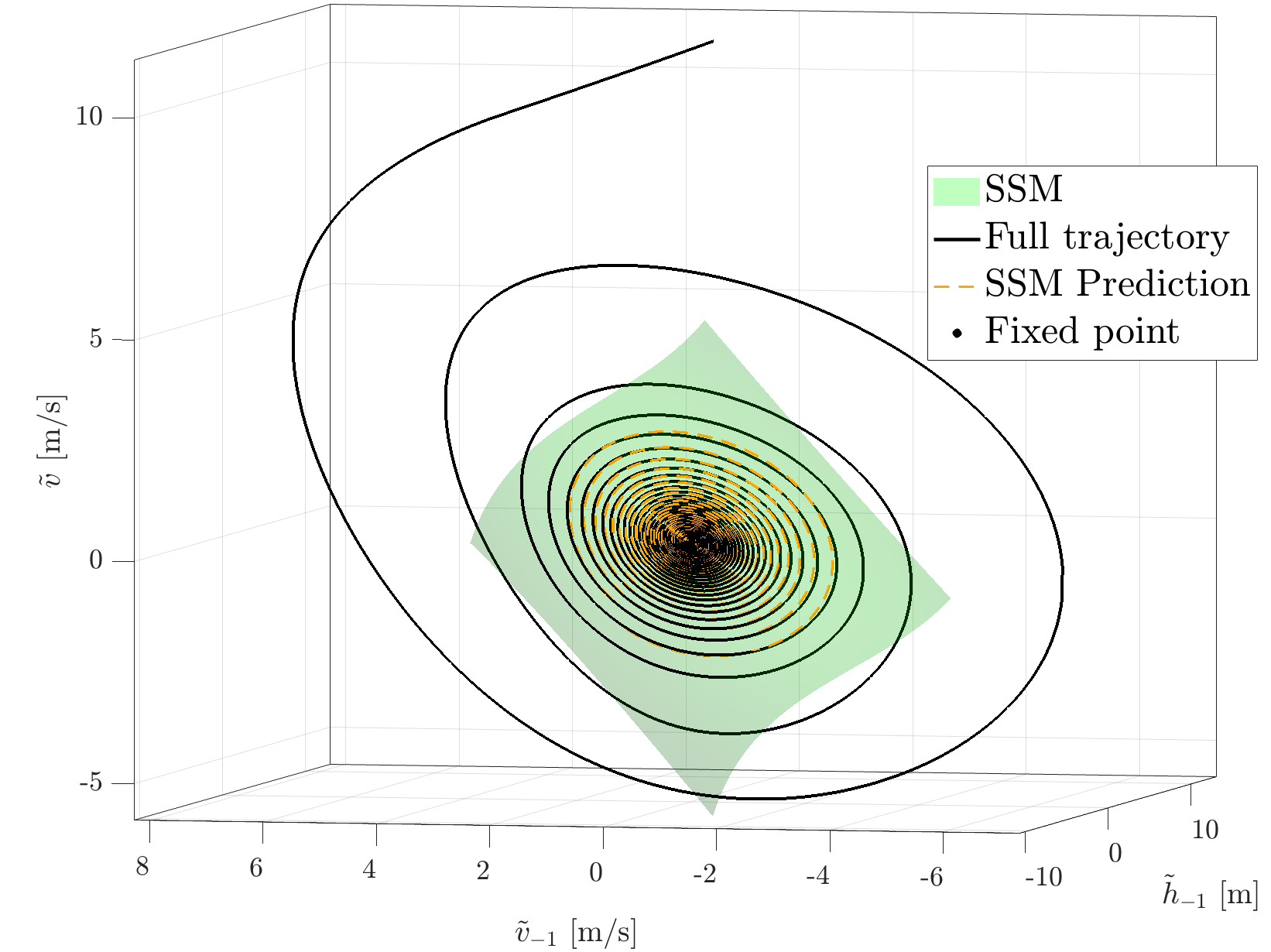}
    \caption*{(a) $\tau = 1.040\,\mathrm{s}$}
\end{minipage}\hfill
\begin{minipage}{\imgXYwidth}
    \centering
    \includegraphics[width=\linewidth]{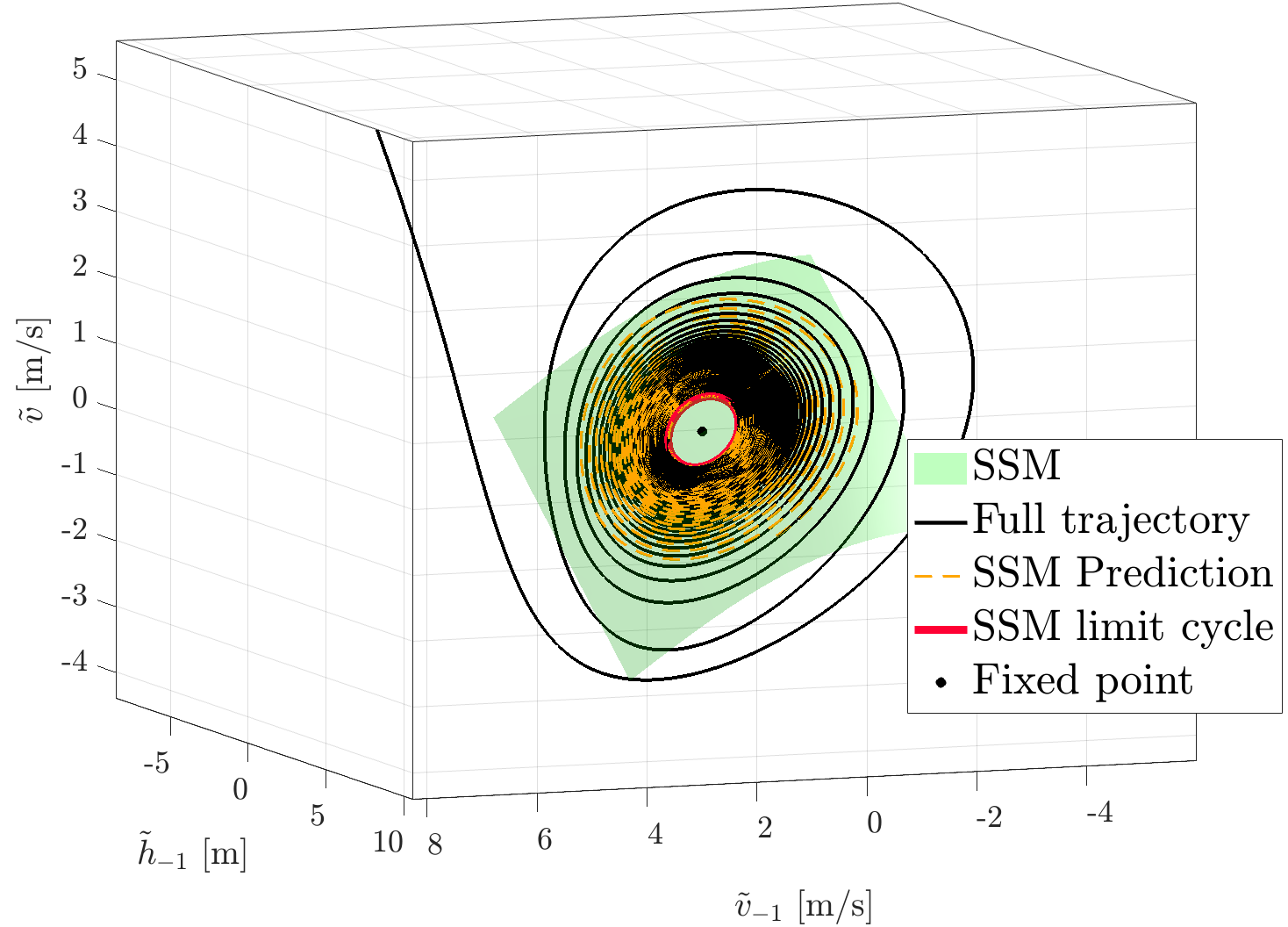}
    \caption*{(b) $\tau = 1.065\,\mathrm{s}$}
\end{minipage}
\vspace{0.5cm}

\begin{minipage}{\imgXYwidth}
    \centering
    \includegraphics[width=\linewidth]{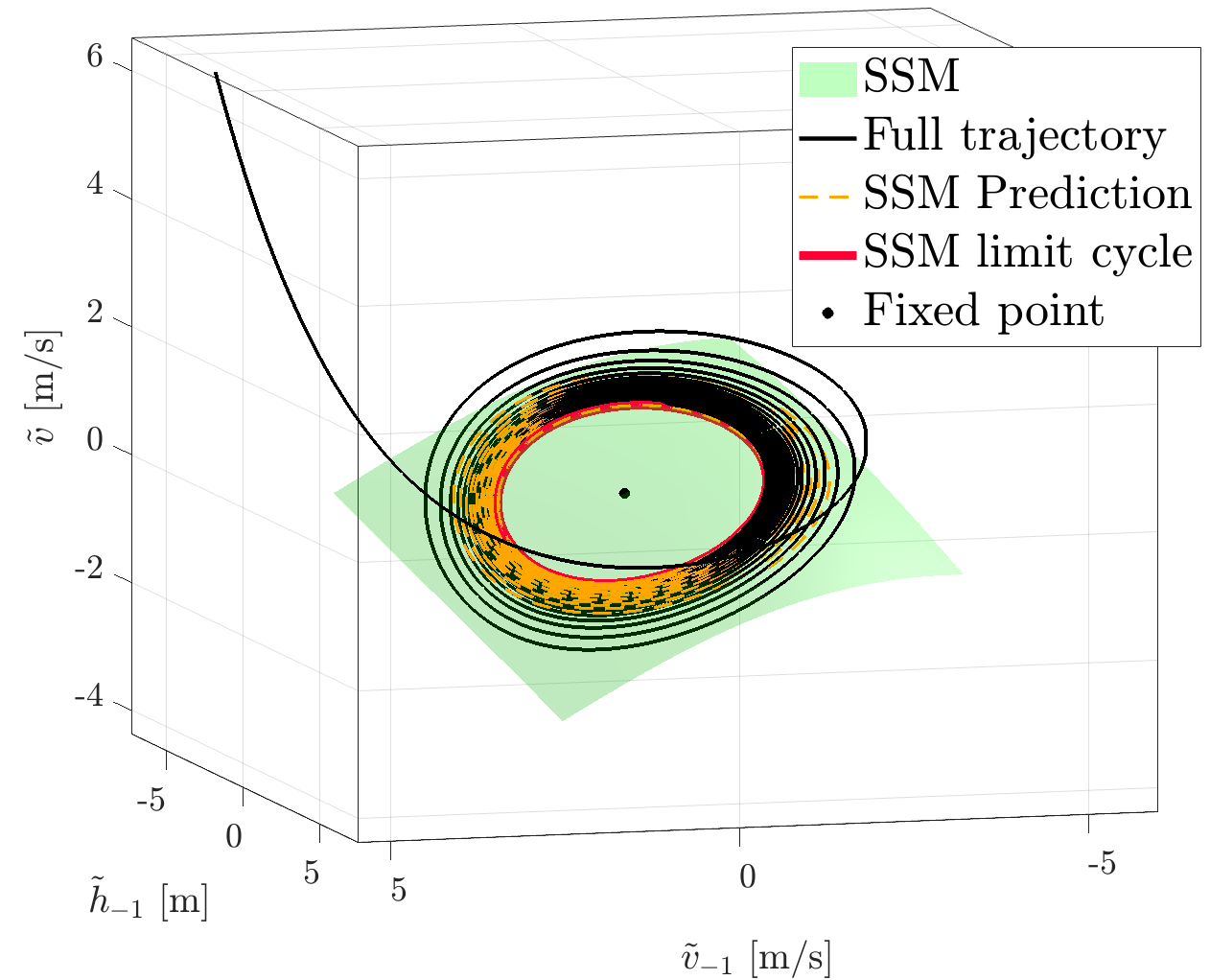}
    \caption*{(c) $\tau = 1.080\,\mathrm{s}$}
\end{minipage}

\caption{For system \eqref{eq:traffic-dde}, 3D projections of the delay embedding space showing the learned SSM (semi-transparent green), the full-system trajectory approaching the SSM (black), and the SSM prediction (orange dashed) for three unseen values of the delay parameter: (a) $\tau = 1.040\,\mathrm{s}$, (b) $\tau = 1.065\,\mathrm{s}$, and (c) $\tau = 1.080\,\mathrm{s}$. The plots illustrate the onset of the Hopf bifurcation, known to occur in the full system at approximately $\tau=1.064\,\mathrm{s}$, as identified by tracking the slowest pair of eigenvalues using the root-finding algorithm \parencite{AppeltansSilmMichiels2022,AppeltansMichiels2023}. They also show the corresponding increase in limit cycle amplitude, while demonstrating that the SSM-based reduced model retains predictive accuracy.}
\label{fig:Traffic_Parametric}
\end{figure}

\subsection{Parametric SSM-reduced model for the analysis of multiple local bifurcations of limit cycles in systems with distributed delay.}

We now consider the Cushing delay equation from \parencite{BuzaHaller2025, Cushing1977}, which is a distributed-delay equation with bounded delay. The parametric SSM-reduced ODE model explains the mechanism underlying qualitative changes in the system dynamics that lead to unusual behavior of observables as the time delay increases.

The system in question is governed by

\begin{equation}
\dot{x}(t)= b\int_{0}^{\tau} x(t-\vartheta)\,d\vartheta + a\bigl(x(t)-\sin(x(t))\bigr),
\label{eq:cushing-distributed}
\end{equation}
where $x(t)\in\mathbb{R}$, $\tau>0$ is the delay, and $a=1,b=-3$ are system parameters. 

An equation-driven reduction for $\tau=1$ was obtained in \parencite{BuzaHaller2025}, whose general formula can be applied directly. Alternatively, the system can be reformulated as a discrete-delay equation by introducing the auxiliary variable if one wishes to carry out the equation-driven reduction following \parencite{Szaksz2025SpectralSystems}:
\[
z(t)=\int_{t-\tau}^{t} x(s)\,ds,
\]
which transforms \eqref{eq:cushing-distributed} into the equivalent delay system
\begin{equation}
\begin{aligned}
\dot{x}(t) &= b\,z(t) + a\bigl(x(t)-\sin(x(t))\bigr),\\
\dot{z}(t) &= x(t)-x(t-\tau).
\end{aligned}
\label{eq:cushing-lifted}
\end{equation}
Thus, the scalar distributed-delay equation is recast as a 2D equation with a single discrete delay. The equilibrium of interest is the origin, $(x,z)=(0,0)$.

The equation is symmetric under the transformation $x \mapsto -x$; therefore, this symmetry is expected to be preserved in the phase portrait of the reduced dynamics. In particular, for instance, every fixed point \((\eta_1^*, \eta_2^*)\) should have a symmetric counterpart \((-\eta_1^*, -\eta_2^*)\). However, because the reduced dynamics are only approximated from data, they are not necessarily exactly odd. This slight symmetry breaking may further stem from the fact that the approximated manifold does not exactly satisfy odd symmetry.

We construct a parametric model following the approach in \parencite{AbbascianoEndreszStepanHaller2026JSV}. In the present setting, however, the construction of the SSMs can be justified directly by Theorem~\ref{thm:main1}, whose conditions are verified. By contrast, the models in \parencite{AbbascianoEndreszStepanHaller2026JSV} were obtained by analogy with the finite-dimensional SSM theory for ODEs. Linear interpolation suffices to obtain an SSM-based model across multiple limit-cycle saddle-node bifurcations near the origin (see Fig.~\ref{fig:parametricPP}). We use eight parameter values, $\tau = \{0.95\,\text{s},\; 0.97\,\text{s},\; 0.99\,\text{s},\; 1\,\text{s},\; 1.01\,\text{s},\; 1.02\,\text{s},\; 1.03\,\text{s},\; 1.035\,\text{s}\}
$, to construct individual autonomous 2D SSM models following the procedure detailed in Section \ref{sec:methodologies}. Cubic-order polynomials are employed for the manifold approximation and 13th-order polynomials for the reduced dynamics, required to ensure accuracy of the SSM-based model at every training parameter point. Each model is trained on 26 trajectories, with 4 reserved for testing. This sampling is more than sufficient to achieve high accuracy.
\newcommand{\imgWZwidth}{0.32\textwidth}

\begin{figure}[H]
\centering

% First row
\begin{minipage}{\imgWZwidth}
    \centering
    \includegraphics[width=\linewidth]{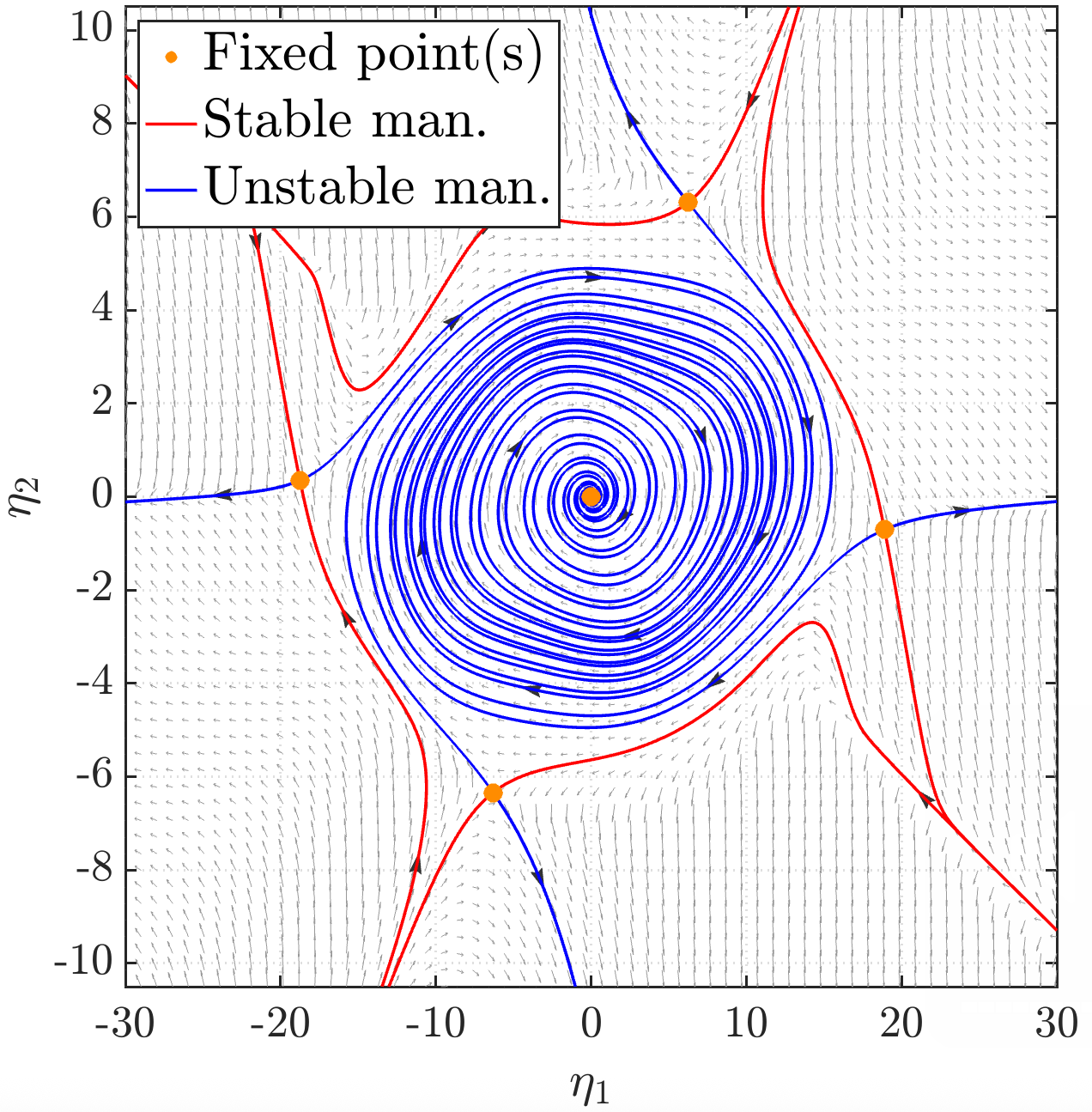}
    \caption*{(a) $\tau=0.98\,\mathrm{s}$}
\end{minipage}\hfill
\begin{minipage}{\imgWZwidth}
    \centering
    \includegraphics[width=\linewidth]{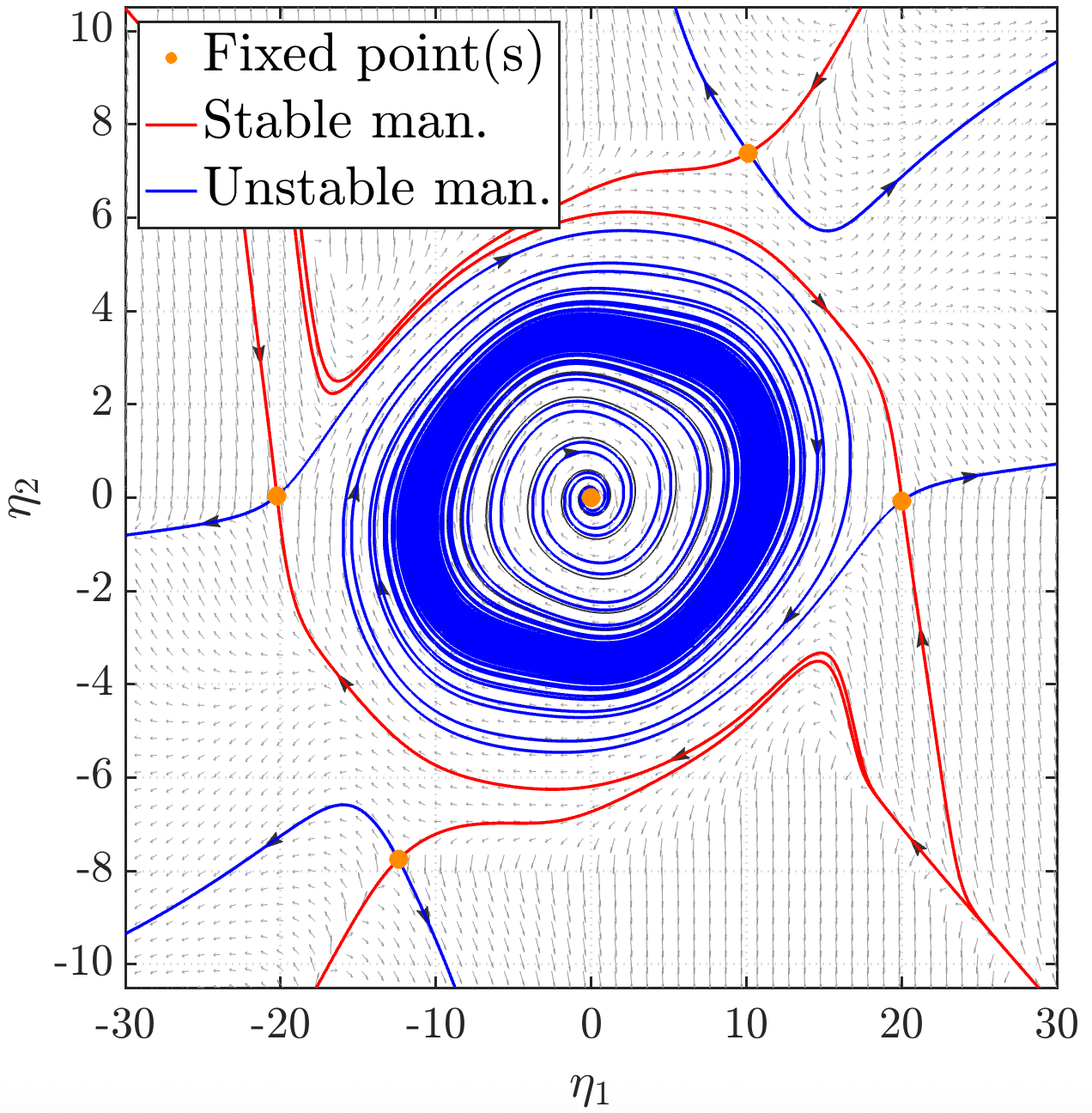}
    \caption*{(b) $\tau=1.0074\,\mathrm{s}$}
\end{minipage}\hfill
\begin{minipage}{\imgWZwidth}
    \centering
    \includegraphics[width=\linewidth]{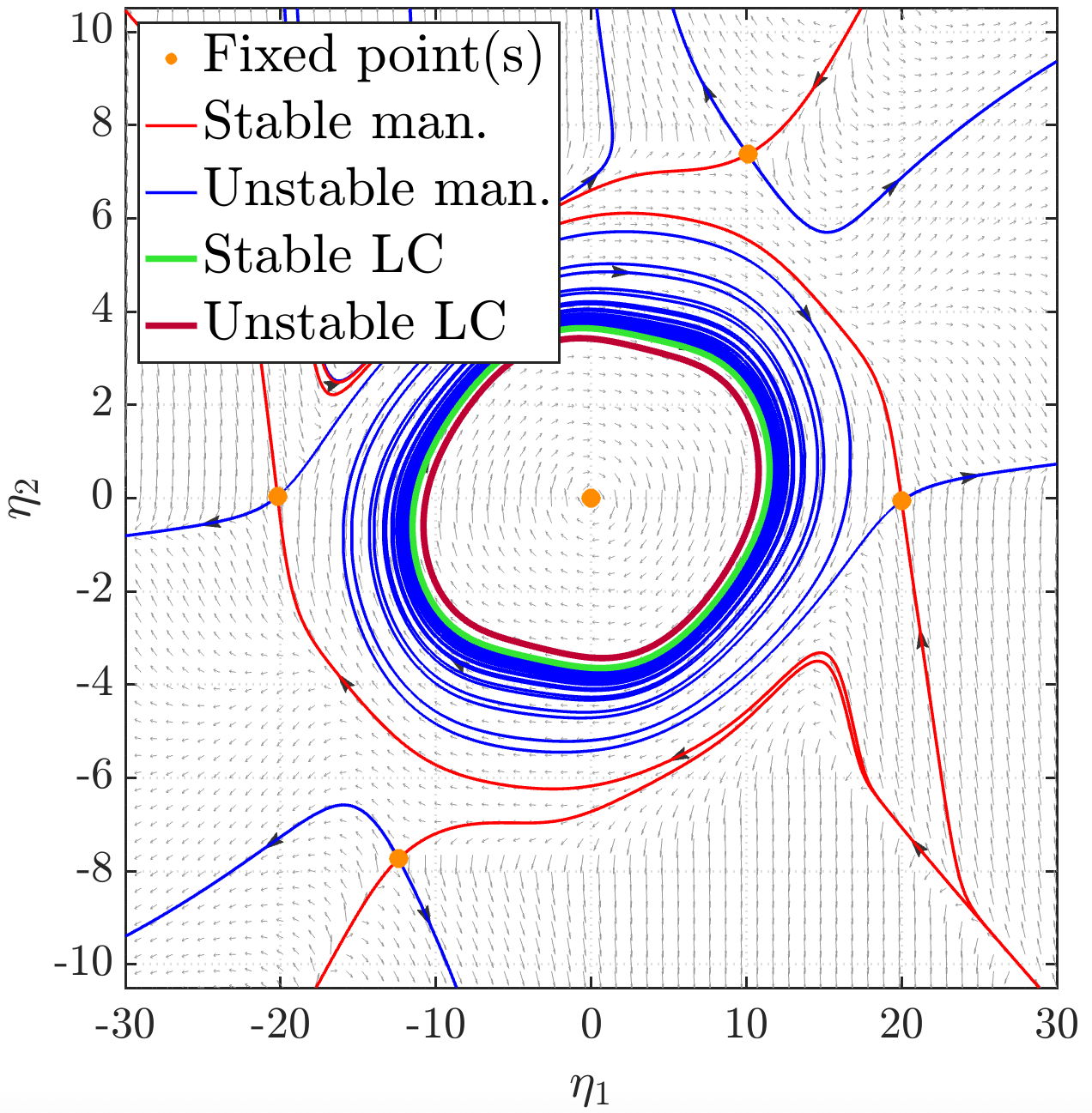}
    \caption*{(c) $\tau=1.0078\,\mathrm{s}$}
\end{minipage}

\vspace{0.5cm}

% Second row
\begin{minipage}{\imgWZwidth}
    \centering
    \includegraphics[width=\linewidth]{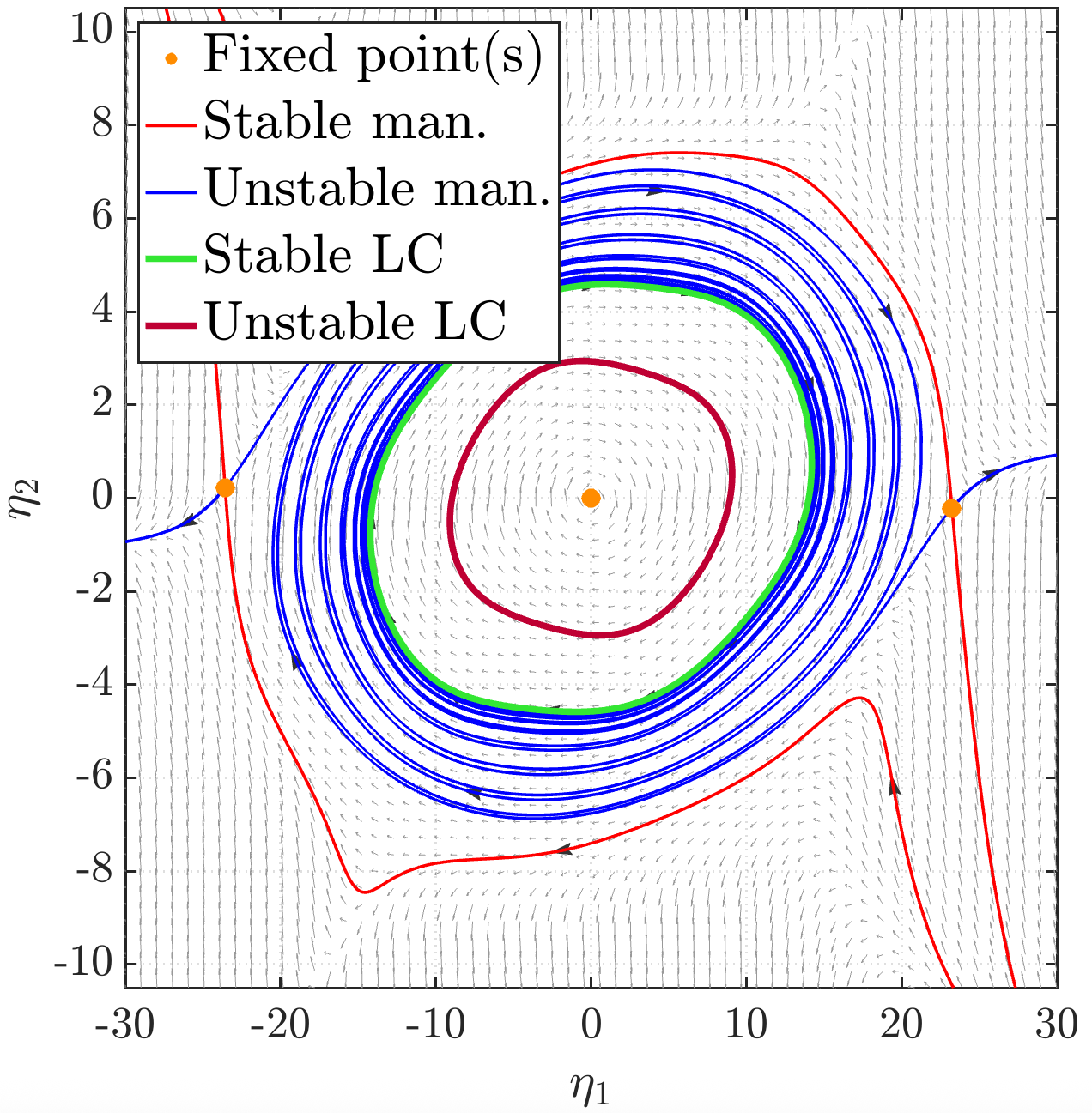}
    \caption*{(d) $\tau=1.025\,\mathrm{s}$}
\end{minipage}\hfill
\begin{minipage}{\imgWZwidth}
    \centering
    \includegraphics[width=\linewidth]{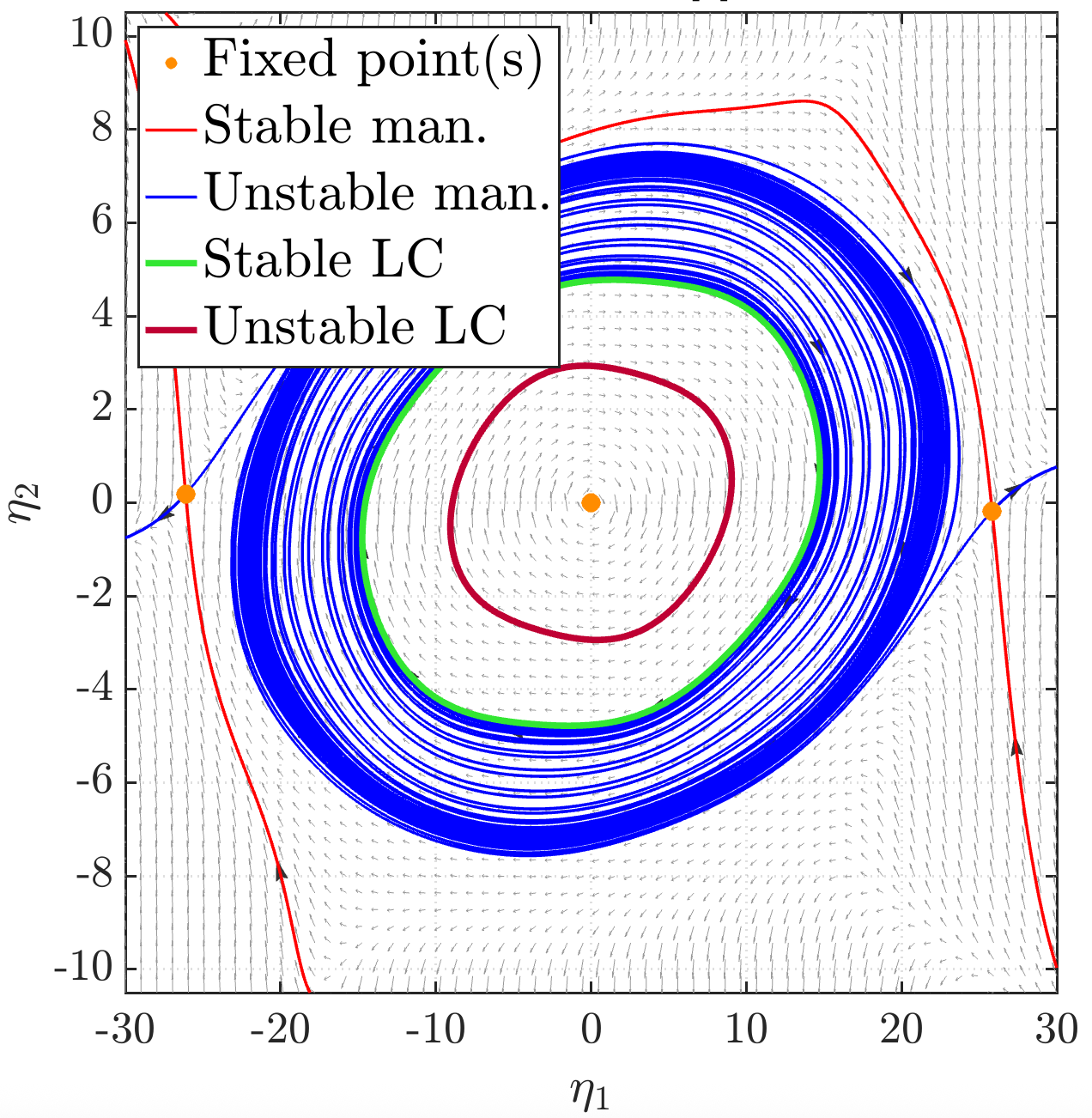}
    \caption*{(e) $\tau=1.0294\,\mathrm{s}$}
\end{minipage}\hfill
\begin{minipage}{\imgWZwidth}
    \centering
    \includegraphics[width=\linewidth]{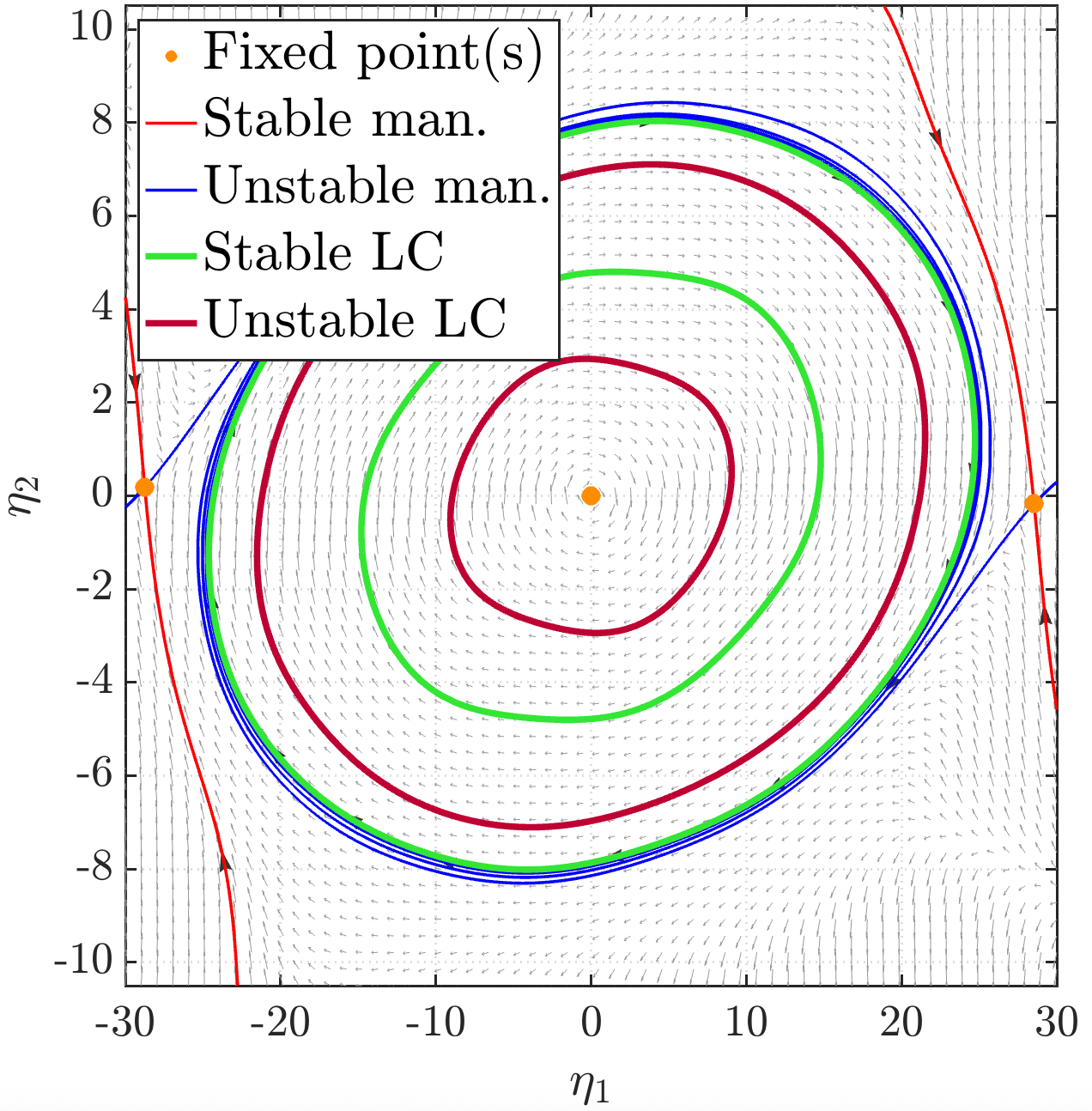}
    \caption*{(f) $\tau=1.0299\,\mathrm{s}$}
\end{minipage}

\caption{Predicted and verified portraits of the parametric SSM-reduced dynamics of system \eqref{eq:cushing-distributed} at six previously unseen values of the delay $\tau$, shown for increasing $\tau$. The reduced dynamics captures a sequence of two saddle-node bifurcations of limit cycles, which explains the qualitative change of the dynamics underlying the unusual behavior of the observables.}
\label{fig:parametricPP}
\end{figure}

Limit cycles are identified as fixed points of the Poincar\'e map, obtained by tracking oriented intersections of trajectories of the SSM-reduced ODE with a chosen codimension-one section in the reduced space, a straight line within the 2D tangent space. This yields a further reduction to a one-dimensional (1D) discrete map of the original infinite-dimensional DDE. However, this reduction loses information on coexisting isolated fixed points, restricting the validity of the reduced dynamics to regions where only the origin equilibrium and limit cycles exist.

A bifurcation diagram can then be constructed from the Poincar\'e map by plotting the positions of limit cycles in the section against the delay parameter, revealing two subsequent saddle-node bifurcations of periodic orbits.

The parametric SSM, corresponding to the 2D slow stable subspace at the origin extended by the parameter direction, approximates well an extended center manifold (i.e. a center manifold in the delay embedding space extended with the delay parameter direction) associated with a saddle-node limit cycle, for the case of both bifurcations, arising near $\tau \approx 1.0075$ and $\tau \approx 1.0295$. Near each bifurcation in the parameter direction, the emerging stable and unstable limit cycles lie on such an extended center manifold emanating from the saddle-node limit cycle, which is locally captured by the SSM. Consequently, a single parametric SSM is able to model the birth and evolution of these newborn limit cycles across the parameter sweep.

We now extend our analysis to a broader range of parameter values, rather than restricting attention to local behavior near critical parameter values. Specifically, by considering the 1D Poincar\'e section in the reduced space and the associated Poincar\'e map, lifted in the delay embedding space in the illustration we provide in Figure \ref{fig:ssm_bifurcation}, we can more clearly observe that a 2D SSM (associated with the stable eigenvalue and the zero eigenvalue corresponding to the parameter direction) emanating from the origin provides a good approximation, at least locally in phase space extended with the parameter and up to a sufficiently high polynomial order of a Taylor expansion, of the 2D spectral submanifold associated with the saddle-node bifurcating eigenvalue and the center eigenvalue (corresponding to the parameter direction)  emanating from the inner saddle-node limit cycle around $\tau_{2,SN}$ and of that emanating from the outer saddle-node limit cycle $\tau_{2,SN}$ (fixed points of the Poincar\'e map). As such, the SSM emanating from the origin captures, over a global range of the parameter, the evolution of both the inner pair of stable and unstable limit cycles and the outer pair of stable and unstable limit cycles, simultaneously.

\begin{figure}[H]
\centering
\begin{adjustbox}{max width=\textwidth}
\begin{tikzpicture}[
    >=Latex,
    font=\Large,
    every node/.style={inner sep=1.5pt},
    every label/.style={font=\Large},
    axisback/.style={->, thin, black},
    axisfront/.style={->, semithick, black},
    axislabel/.style={font=\large, fill=white, inner sep=1pt},
    section/.style={draw=black!55, fill=gray!10, line width=0.45pt},
    ssm/.style={black, line width=1.25pt},
    cmstyle/.style={blue!75!black, line width=2.0pt, line cap=round},
    umstyle/.style={red!75!black, line width=2.0pt, line cap=round},
    smstyle/.style={teal!70!black, line width=2.0pt, line cap=round},
    paramline/.style={black, very thick, ->},
    fp/.style={circle, fill=black, inner sep=1.8pt},
    sndot/.style={circle, fill=blue!75!black, inner sep=2.0pt},
    ulcdot/.style={circle, fill=red!75!black, inner sep=2.0pt},
    slcdot/.style={circle, fill=teal!70!black, inner sep=2.0pt},
    taupt/.style={circle, fill=black, inner sep=2.2pt}
]

% ------------------------------------------------------------
% positions
% ------------------------------------------------------------
\def\xA{-10.4}
\def\xB{-6.2}
\def\xC{-2.0}
\def\xD{2.2}
\def\xE{6.4}

\def\yPanels{3.10}
\def\yLegend{8.15}
\def\yParam{0.2}
\def\Wsec{2.35}

% 3D-like in-plane projection corresponding to a 10 degree rotation
% around the vertical axis x(t+2\Delta t)
\def\cosa{0.9848}
\def\sina{0.1736}

% ------------------------------------------------------------
% legend (three rows)
% ------------------------------------------------------------
\begin{scope}[shift={(0,\yLegend)}]

    % first row
    \node[fp] at (-11.2,0.75) {};
    \node[anchor=west] at (-11.0,0.75) {$\mathrm{stable\ fixed\ point}$};

    \draw[section] (-5.85,0.55) rectangle (-4.25,0.95);
    \node[anchor=west] at (-4.05,0.75) {$S\ \mathrm{(Poincar\acute{e}\ section)}$};

    \node[sndot] at (1.15,0.75) {};
    \node[anchor=west] at (1.35,0.75) {$\mathrm{saddle\text{-}node\ limit\ cycle}$};

    % second row
    \node[ulcdot] at (-11.2,0.10) {};
    \node[anchor=west] at (-11.0,0.10) {$\mathrm{unstable\ limit\ cycle}$};

    \node[slcdot] at (-5.85,0.10) {};
    \node[anchor=west] at (-5.65,0.10) {$\mathrm{stable\ limit\ cycle}$};

    \draw[ssm] (1.15,0.10) -- (2.15,0.10);
    \node[anchor=west] at (2.35,0.10) {$\mathrm{SSM}$};

    % third row
    \draw[cmstyle] (-11.2,-0.55) -- (-10.2,-0.55);
    \node[anchor=west] at (-10.0,-0.55) {$\mathrm{center\ manifold}$};

    \draw[umstyle] (-5.85,-0.55) -- (-4.85,-0.55);
    \node[anchor=west] at (-4.65,-0.55) {$\mathrm{unstable\ manifold}$};

    \draw[smstyle] (1.15,-0.55) -- (2.15,-0.55);
    \node[anchor=west] at (2.35,-0.55) {$\mathrm{stable\ manifold}$};

\end{scope}

% ------------------------------------------------------------
% panel helpers
% ------------------------------------------------------------
\newcommand{\BeginPanel}[1]{%
\begin{scope}[shift={#1}]
\begin{scope}[x={(\cosa cm,\sina cm)},y={(0cm,1cm)}]
    \coordinate (O) at (0,0);
    \coordinate (A) at (0.00,-1.42);
    \coordinate (B) at (\Wsec,-1.42);
    \coordinate (C) at (\Wsec,2.18);
    \coordinate (D) at (0.00,2.18);
    \draw[section] (A)--(B)--(C)--(D)--cycle;
    \node[gray!70!black, font=\large] at ({\Wsec-0.18},1.88) {$S$};
}

\newcommand{\EndPanelLabeled}{%
    \draw[axisback] (O) -- (-1.35,-1.18);
    \node[fp] at (O) {};
    \draw[axisfront] (O) -- (0.00,2.40);
    \draw[axisfront] (O) -- (2.40,0.00);

    \node[axislabel, anchor=south] at (-1.2,2.18)
      {\shortstack{$x(t+3\Delta t)$\\ $\dots$\\ $x(t+k\Delta t)$}};
    \node[axislabel, anchor=east] at (-0.52,-1.83)
      {$x(t+2\Delta t)$};
    \node[axislabel, anchor=west] at (2.52,0.14)
      {$x(t)$};
\end{scope}
\end{scope}
}

\newcommand{\EndPanel}{%
    \draw[axisback] (O) -- (-1.35,-1.18);
    \node[fp] at (O) {};
    \draw[axisfront] (O) -- (0.00,2.40);
    \draw[axisfront] (O) -- (2.40,0.00);
\end{scope}
\end{scope}
}

% ------------------------------------------------------------
% PANEL 1
% ------------------------------------------------------------
\BeginPanel{(\xA,\yPanels)}
    \draw[ssm]
        (0.00,0.00)
        .. controls (0.62,0.00) and (1.55,0.10) .. (2.35,0.74);
\EndPanelLabeled

% ------------------------------------------------------------
% PANEL 2
% ------------------------------------------------------------
\BeginPanel{(\xB,\yPanels)}
    \draw[ssm]
        (0.00,0.00)
        .. controls (0.55,0.00) and (1.05,0.10) .. (1.26,0.31)
        .. controls (1.48,0.47) and (1.88,0.69) .. (2.35,1.02);

    \draw[cmstyle]
        (0.55,0.00)
        .. controls (1.05,0.10) and (1.26,0.31) .. (1.26,0.31)
        .. controls (1.48,0.47) and (1.88,0.69) .. (2.05,0.81);

    \node[sndot] at (1.26,0.31) {};
\EndPanel

% ------------------------------------------------------------
% PANEL 3
% ------------------------------------------------------------
\BeginPanel{(\xC,\yPanels)}
    \draw[ssm]
        (0.00,0.00)
        .. controls (0.48,0.00) and (0.82,0.18) .. (1.04,0.40)
        .. controls (1.16,0.48) and (1.30,0.56) .. (1.46,0.65)
        .. controls (1.70,0.82) and (2.00,1.03) .. (2.35,1.28);

    \draw[umstyle]
        (0.82,0.18)
        .. controls (0.48,0.00) and (0.82,0.18) .. (1.04,0.40)
        .. controls (1.16,0.48) and (1.30,0.56) .. (1.28,0.54);

    \draw[smstyle]
        (1.28,0.54)
        .. controls (1.30,0.56) and (1.46,0.65) .. (1.46,0.65)
        .. controls (1.70,0.82) and (2.00,1.03) .. (1.86,0.91);

    \node[ulcdot] at (1.04,0.40) {};
    \node[slcdot] at (1.46,0.65) {};
\EndPanel

% ------------------------------------------------------------
% PANEL 4
% ------------------------------------------------------------
\BeginPanel{(\xD,\yPanels)}
    \draw[ssm]
        (0.00,0.00)
        .. controls (0.60,0.00) and (0.85,0.08) .. (1.00,0.13)
        .. controls (1.18,0.19) and (1.42,0.36) .. (1.62,0.58)
        .. controls (1.78,0.74) and (1.92,0.91) .. (2.05,1.05)
        .. controls (2.15,1.16) and (2.25,1.30) .. (2.35,1.46);

    \draw[umstyle]
        (0.60,0.00)
        .. controls (0.85,0.08) and (1.00,0.13) .. (1.00,0.13)
        .. controls (1.18,0.19) and (1.42,0.36) .. (1.38,0.35);

    \draw[smstyle]
        (1.38,0.35)
        .. controls (1.42,0.36) and (1.62,0.58) .. (1.62,0.58)
        .. controls (1.78,0.74) and (1.92,0.91) .. (1.95,0.92);

    \draw[cmstyle]
        (1.95,0.92)
        .. controls (1.92,0.91) and (2.05,1.05) .. (2.05,1.05)
        .. controls (2.15,1.16) and (2.25,1.30) .. (2.33,1.42);

    \node[ulcdot] at (1.00,0.13) {};
    \node[slcdot] at (1.62,0.58) {};
    \node[sndot] at (2.05,1.05) {};
\EndPanel

% ------------------------------------------------------------
% PANEL 5
% ------------------------------------------------------------
\BeginPanel{(\xE,\yPanels)}
    \draw[ssm]
        (0.00,0.00)
        .. controls (0.22,0.00) and (0.34,0.10) .. (0.44,0.24)
        .. controls (0.62,0.42) and (0.90,0.64) .. (1.30,0.96)
        .. controls (1.44,1.05) and (1.58,1.13) .. (1.70,1.18)
        .. controls (1.84,1.28) and (1.98,1.40) .. (2.12,1.49)
        .. controls (2.19,1.54) and (2.27,1.60) .. (2.35,1.66);

    \draw[umstyle]
        (0.22,0.00)
        .. controls (0.34,0.10) and (0.44,0.24) .. (0.44,0.24)
        .. controls (0.62,0.42) and (0.90,0.64) .. (0.94,0.67);

    \draw[smstyle]
        (0.94,0.67)
        .. controls (1.04,0.75) and (1.18,0.86) .. (1.30,0.96)
        .. controls (1.44,1.05) and (1.58,1.13) .. (1.52,1.10);

    \draw[umstyle]
        (1.52,1.10)
        .. controls (1.58,1.13) and (1.70,1.18) .. (1.70,1.18)
        .. controls (1.84,1.28) and (1.98,1.40) .. (2.02,1.41);

    \draw[smstyle]
        (1.90,1.34)
        .. controls (1.96,1.38) and (2.00,1.40) .. (2.02,1.41)
        .. controls (2.06,1.44) and (2.12,1.49) .. (2.12,1.49)
        .. controls (2.19,1.54) and (2.27,1.60) .. (2.35,1.66);

    \node[ulcdot] at (0.44,0.24) {};
    \node[slcdot] at (1.30,0.96) {};
    \node[ulcdot] at (1.70,1.18) {};
    \node[slcdot] at (2.12,1.49) {};
\EndPanel

% ------------------------------------------------------------
% parameter line
% ------------------------------------------------------------
\draw[paramline] (-11.8,\yParam) -- (9.0,\yParam) node[right] {$\tau$};

\node[taupt,label={[font=\Large]below:{$\tau<\tau_{1,\mathrm{SN}}$}}] at (\xA,\yParam) {};
\node[taupt,label={[font=\Large]below:{$\tau_{1,\mathrm{SN}}$}}] at (\xB,\yParam) {};
\node[taupt,label={[font=\Large]below:{$\tau_{1,\mathrm{SN}}<\tau<\tau_{2,\mathrm{SN}}$}}] at (\xC,\yParam) {};
\node[taupt,label={[font=\Large]below:{$\tau_{2,\mathrm{SN}}$}}] at (\xD,\yParam) {};
\node[taupt,label={[font=\Large]below:{$\tau>\tau_{2,\mathrm{SN}}$}}] at (\xE,\yParam) {};

\end{tikzpicture}
\end{adjustbox}
\caption{Illustration of the codimension one Poincar\'e section viewed in the delay embedding space extended with the bifurcation delay parameter. Limit cycles can be found as fixed points of the associated 1D Poincar\'e map of the 2D SSM-reduced dynamics via e.g. Newton-Raphson iteration. It is shown how one SSM-reduced ODE model is able to contain manifolds of the limit cycles along the parameter thus giving accurate reduced model of the full, infinite dimensional DDE across multiple saddle node bifurcations of periodic orbits. One should note that the manifolds (unstable/stable/center) of limit cycles are well approximated by the SSM at least locally in the phase space, but far from the fixed point from which they emanate they might in general deviate off the SSM due to higher order terms which are not captured by the fixed utilized order of the SSM.}
\label{fig:ssm_bifurcation}
\end{figure}
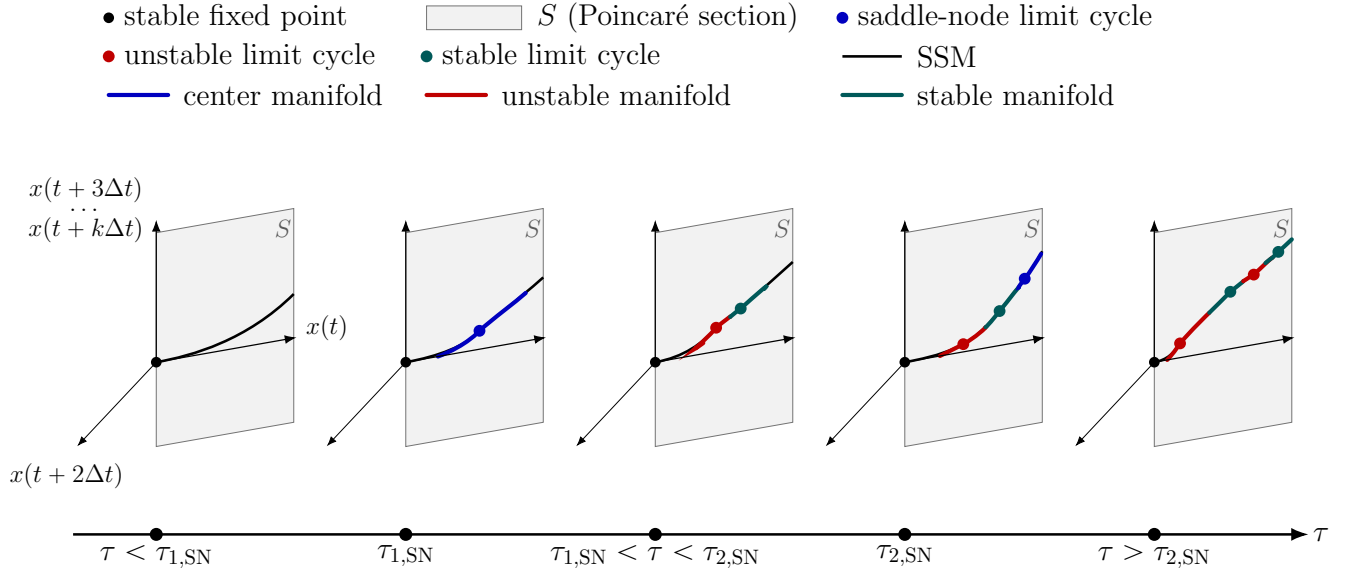

We emphasize that our SSM is only an approximation, as the saddle-node limit cycles, together with the newborn unstable and stable limit cycles, lie on a particular extended center manifold of a family of extended center manifolds. Moreover, these invariant objects also lie on a specific SSM chosen from a family of SSMs, which are not unique in the present setting of delay differential equations. 

If the order of polynomials used to approximate the SSM from data were increased, the discrepancy between the computed SSM and the actual particular SSM on which the saddle-node limit cycles, unstable limit cycles, and stable limit cycles lie would also increase, since higher-order polynomials would select a more specific member of the non-unique family of SSMs. Guided by Taylor series considerations, for lower polynomial orders, this discrepancy is less pronounced and we can conflate all the infinitely many SSMs with the one we have computed. This approximation provides a good representation of the SSM that actually contains the saddle-node limit cycles, unstable limit cycles, and stable limit cycles, as well as of the particular extended center manifold of the saddle-node limit cycle on which the unstable limit cycle and stable limit cycle are going to be born. Finally, we stress that our approach provides a polynomial approximation inferred from data, rather than computating the Taylor series.

More generally, away from the bifurcation values, each stable (respectively unstable) limit cycle possesses a unique 1D stable (respectively unstable) manifold in the Poincar\'e section for a fixed parameter value. Each of these manifolds is captured, i.e. well approximated at least locally in phase space and up to a sufficiently high polynomial order, by the 1D computed SSM emanating from the origin.

An alternative approach to constructing the bifurcation diagram is to compute the extended normal-form transformation and introduce polar coordinates. This yields a 2D ODE in $(\dot{\rho}, \dot{\vartheta})$, where the $\dot{\rho}$ equation decouples from $\dot{\vartheta}$. A parametric study of the $\dot{\rho}$ equation (e.g., via interpolation) then provides a bifurcation diagram of $\rho$ versus $\tau$. However, this approach is difficult to implement from data, as the extended normal-form transformation is highly sensitive and does not interpolate reliably across parameter values. Instead, one could proceed in an equation-driven manner, which is computationally demanding, especially given the high polynomial order (13th order in this case) required to obtain an accurate parametric SSM-reduced model capturing the two local bifurcations of four coexisting limit cycles. 

This behavior is inherently nonlinear and cannot be captured by linear models. For instance, any principal Koopman eigenfunction whose domain contains the fixed point and the inner unstable limit cycle must identically vanish along both, thereby trivializing the dynamics.

Predictions of the parametric SSM model remain accurate for unseen delay values, as shown in Fig.~\ref{fig:unseenPar}, with the NMTE remaining below $1.2\%$. Panel (a) shows the onset of a saddle-node limit cycle, indicated by the gradual decay of oscillation amplitude toward a stable fixed point. Panels (b) and (c) show stable limit cycles after the first bifurcation, while panel (d) exhibits the coexistence of two limit cycles with different amplitudes after the second bifurcation. 

\newcommand{\imgWZZwidth}{0.47\textwidth}

\begin{figure}[H]
\centering

% First row
\begin{minipage}{\imgWZZwidth}
    \centering
    \includegraphics[width=\linewidth]{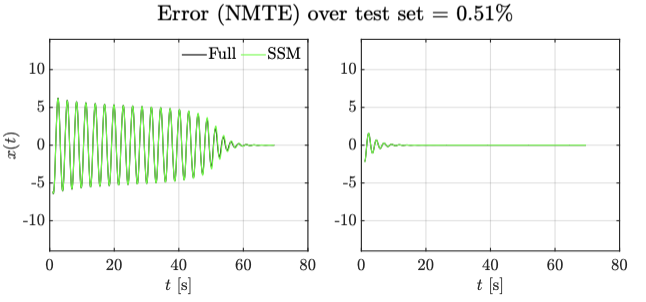}
    \caption*{(a) $\tau=1.005\,\mathrm{s}$}
\end{minipage}\hspace{0.02\textwidth}
\begin{minipage}{\imgWZZwidth}
    \centering
    \includegraphics[width=\linewidth]{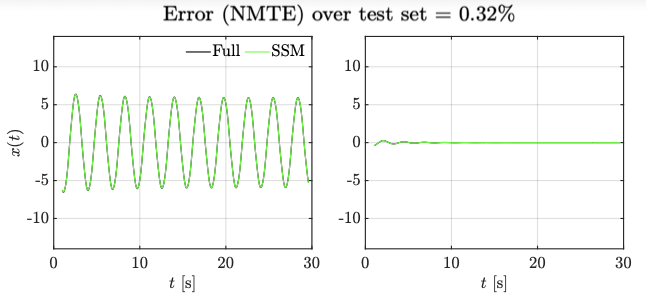}
    \caption*{(b) $\tau=1.0125\,\mathrm{s}$}
\end{minipage}

\vspace{0.5cm}

% Second row
\begin{minipage}{\imgWZZwidth}
    \centering
    \includegraphics[width=\linewidth]{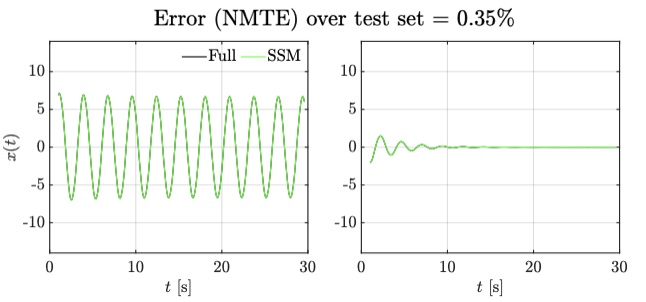}
    \caption*{(c) $\tau=1.025\,\mathrm{s}$}
\end{minipage}\hspace{0.02\textwidth}
\begin{minipage}{\imgWZZwidth}
    \centering
    \includegraphics[width=\linewidth]{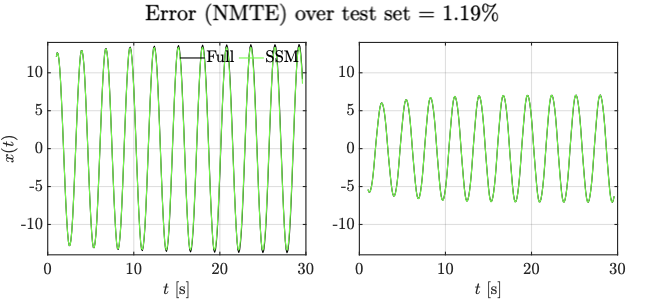}
    \caption*{(d) $\tau=1.032\,\mathrm{s}$}
\end{minipage}

\caption{Validation of the parametric SSM model for system \eqref{eq:cushing-distributed} at unseen delay values. The model captures the onset of a saddle-node limit cycle (a), stable limit cycles (b,c), and the coexistence of two limit cycles (d). The NMTE is reported in each panel.}
\label{fig:unseenPar}
\end{figure}

\subsection{SSM-reduced modeling from periodically delayed experimental data}\label{section5}
Here, we verify the existence of an SSM in the data-driven modeling of a PD-controlled mechanical system with delayed feedback and quantization using experimental data collected from the setup shown in Figure~\ref{fig:expSetups}. The experiments were conducted at the Department of Applied Mechanics, Budapest University of Technology and Economics (BME), under the supervision of Professor Gábor Stépán. These data were previously made publicly available in \parencite{AbbascianoEndreszStepanHaller2026JSV}. In that study, SSMs are employed to model global bifurcations and chaotic dynamics.

\begin{figure}[H]
  \centering
  \includegraphics[width=0.45\textwidth]{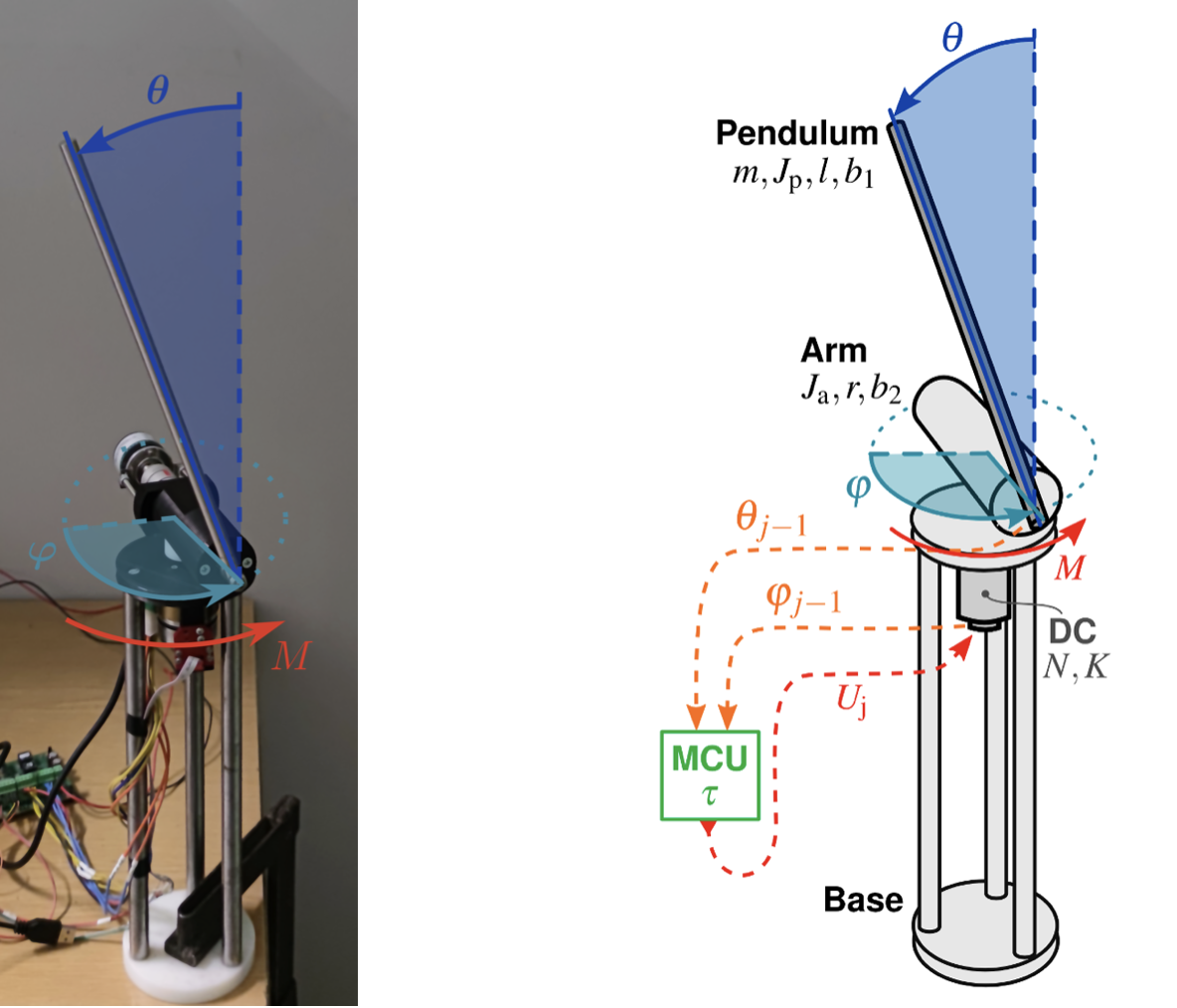}
  \caption{Left: experimental rig used to control the Furuta pendulum. Right: schematic of the experimental setup (adopted from \parencite{Vizi2024}).}
  \label{fig:expSetups}
\end{figure}

In the experimental setting considered there, a periodic delay $\rho(t)$ arises from the ZOH logic of the digital controller implementation:
\begin{equation}\label{eq:rho}
  \rho(t)
  = t - \Delta t\,
    \Bigl\lfloor \tfrac{t}{\Delta t}\Bigr\rfloor + r\,\Delta t,
  \qquad r\in\mathbb{N},
\end{equation}
where $\lfloor\cdot\rfloor$ denotes the integer part. The resulting controller simultaneously accounts for the constant feedback (discrete) delay $r\,\Delta t$ (with $r=1$ for the system studied here) and the additional delay induced by the ZOH mechanism,
\[
t - \Delta t\,
    \Bigl\lfloor \tfrac{t}{\Delta t}\Bigr\rfloor,
\]
which is periodic in time. The closed-loop system can therefore be written as
\begin{equation}\label{eq:digContr}
\dot{x}(t)=
\underbrace{f\bigl(x(t)\bigr)}_{\text{uncontrolled system}} +\quad \underbrace{g(x(t-\rho(t)))}_{\text{linear controller}}
\end{equation}
The control action is updated at the beginning of each sampling interval and then held constant over the interval of length $\Delta t$, as shown in Figure \ref{fig:periodic_delay}. This mechanism produces a periodic sawtooth-like delay, as discussed in \parencite{Insperger2015}, where an expression for its average value is also derived.

\newcommand{\imgPD}{0.6\textwidth}
\begin{figure}[H]
  \centering
  \includegraphics[width=\imgPD]{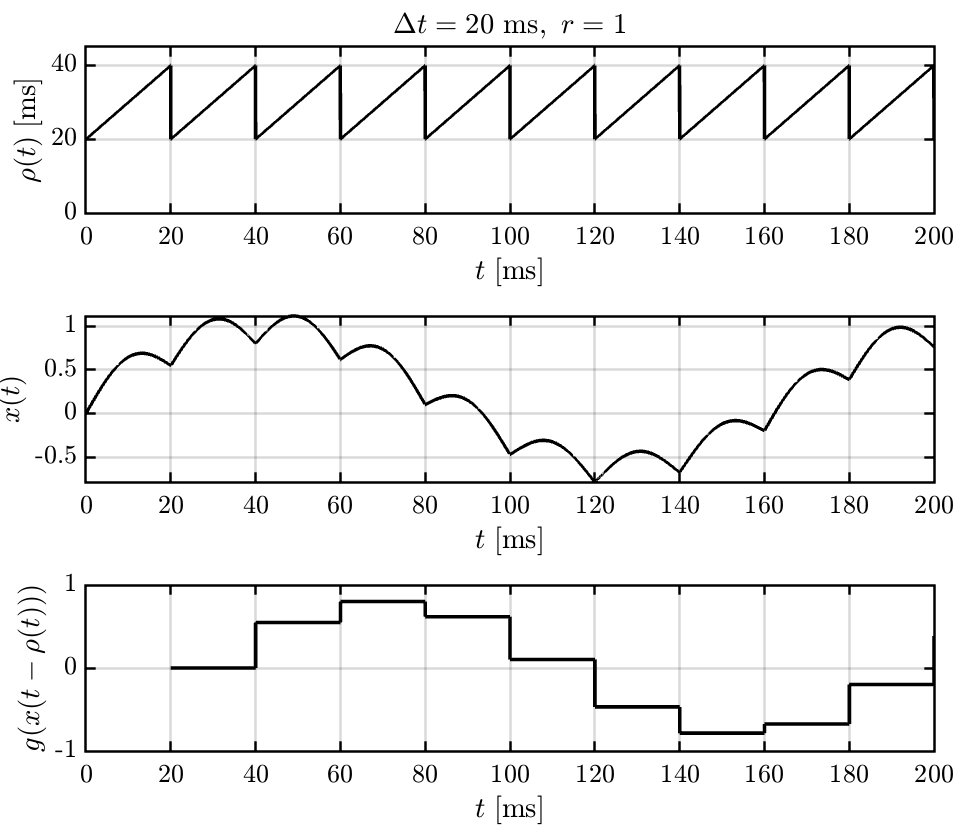}
  \caption{Illustration of the periodic-delay mechanism induced by the ZOH for a scalar DDE of the form \eqref{eq:digContr}, with $f\in C^1$, sampling time $\Delta t=20\,\mathrm{ms}$, and feedback delay $r=1$. Top: periodic delay $\rho(t)$. Middle: state trajectory $x(t)$. Bottom: delayed control action $g(x(t-\rho(t)))$, shown as a piecewise-constant signal due to the sample-and-hold action of the digital controller.}
  \label{fig:periodic_delay}
\end{figure}

In light of the new theoretical results detailed in Theorem~\ref{thm:main}, which extend the Theorem~\ref{thm:main1} \parencite{BuzaHaller2025} to DDEs with periodic delay, we carry out data-driven SSM-reduction following \ref{sec:methodologies}. 

Due to the ZOH logic, the control action $g$ is in theory non-smooth. In reality, however, the control action is smooth with high gradients instead of nonsmooth, as a jump at the beginning of each and every sampling interval would require infinite power to occur, which is not possible. The system is thus $C^1$, allowing us to apply Theorem~\ref{thm:main} to digital controllers \eqref{eq:digContr}.

We first reduce the system to a Poincar\'e section associated with sampling at intervals of length $\Delta t$. This yields an autonomous DDE whose Poincar\'e map is $C^1$ and, by Theorem~\ref{thm:main}, admits infinitely many $d-$dimensional SSMs in the Poincare' section emanating from the equilibrium at $0$, all tangent to a selected spectral subspace and with regularity $C^1$. Infinitely many periodic SSMs and the associated periodic spectral subspace exist with regularity $C^1$ as well in the extended phase space $X\times S^1$. These periodic objects can be reconstructed from data by periodically interpolating the SSMs and spectral subspaces obtained in different Poincar\'e sections.

By subsampling the experimental time series of the observable $\vartheta$  at intervals $\Delta t$, we obtain observations corresponding to intersections of the continuous-time flow with a codimension-one Poincar\'e section $S$ in the delay-embedding space extended with time $\phi \doteq t \bmod \Delta t \in S^1$. The SSMLearn algorithm is then used, following Section \ref{sec:methodologies}, to fit a polynomial approximation of the SSM in $S$, whose existence is guaranteed by Theorem \ref{thm:main}, emanating from the unstable upright equilibrium of the pendulum in \parencite{AbbascianoEndreszStepanHaller2026JSV}, corresponding to the angular position of the vertical arm $\theta=0$.

The reduced dynamics on the manifold are subsequently learned by projecting the data onto the tangent space and identifying a discrete-time map that approximates the Poincar\'e map. In this work, we employ RBFs, as the quantization inherent in the digital controller induces microchaos in the feedback-delayed PD controlled system \parencite{Haller1996, Stepan2017}. Finally, the SSM in the full extended delay-embedding space can be reconstructed by gluing the SSMs computed on each Poincar\'e section, parametrized by the phase variable $\phi\in S^1$.

The chaotic attractor has an estimated fractal dimension of approximately $4.3$. Assuming that the chaotic attractor is contained in one of the SSMs in $X$, we seek an SSM of sufficiently high dimension to ensure a correct embedding of the attractor in the reduced space within the delay-coordinate space \parencite{SauerYorkeCasdagli1991}. In this sense, a $10$-dimensional embedding is adopted as a sufficient, not necessary choice. We therefore fit a first-order, $10$-dimensional SSM, with the reduced dynamics approximated by radial basis functions (RBFs). The model is trained on $4$ trajectories and tested on $4$ additional trajectories (see Figure~\ref{fig:CorrDimTrainTest}).

\begin{figure}[H]
  \centering
  \begin{minipage}[t]{0.50\textwidth}
    \centering
    \includegraphics[width=\linewidth]{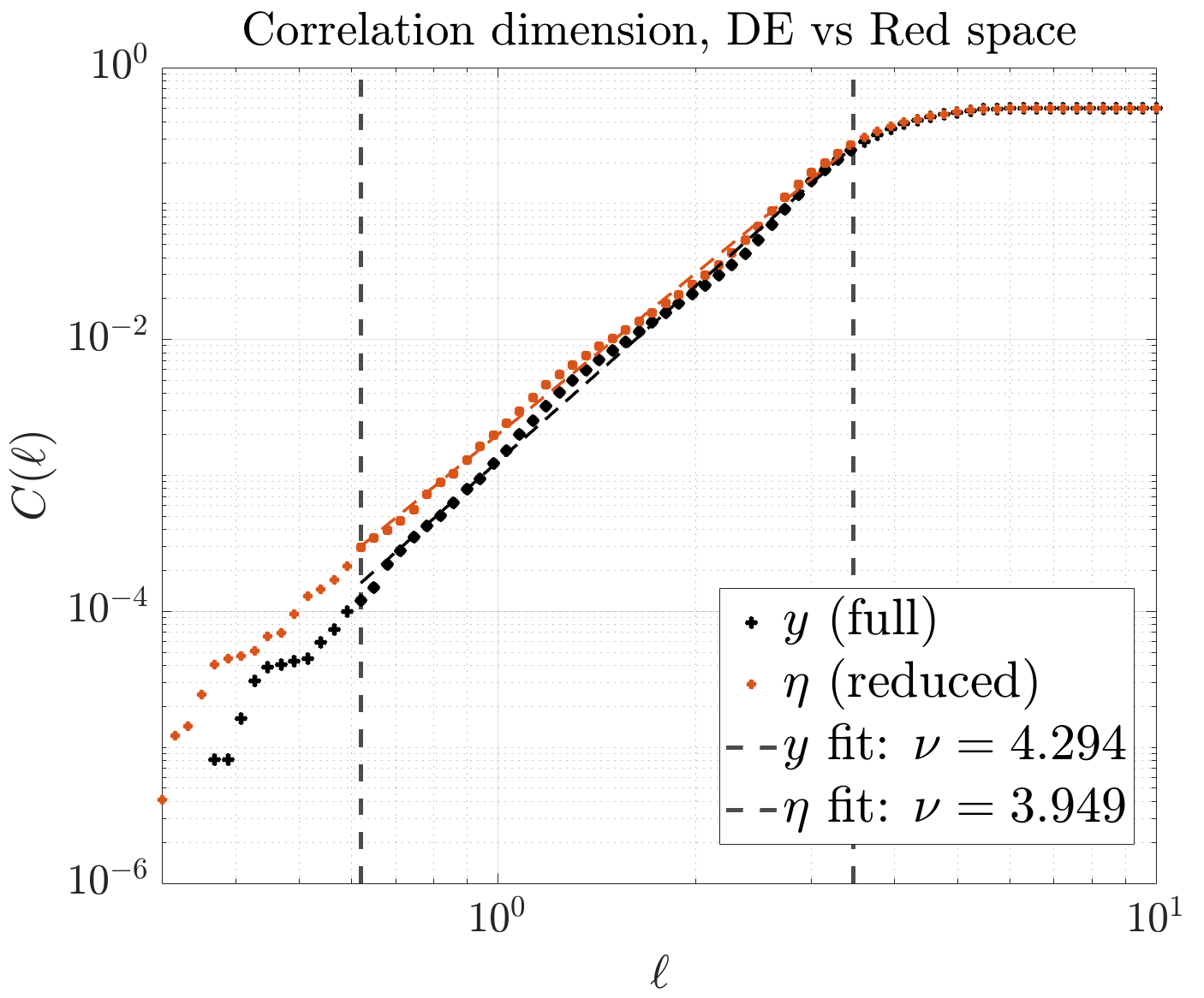}
  \end{minipage}\hfill
  \begin{minipage}[t]{0.48\textwidth}
    \centering
    \includegraphics[width=\linewidth]{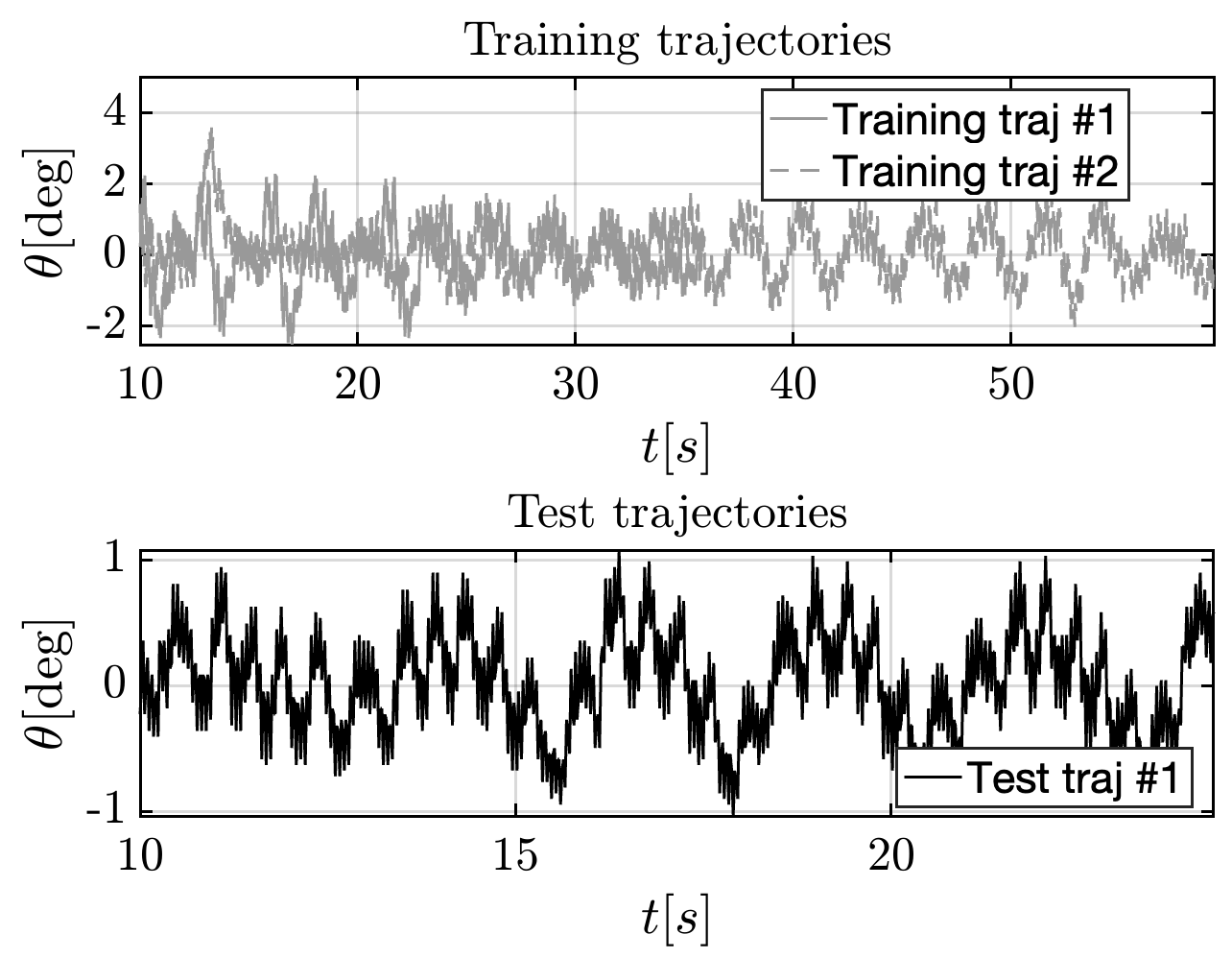}
  \end{minipage}
  \caption{Left: correlation dimension of the chaotic attractor, yielding a value of approximately $4.3$. Right: training and test trajectories used for identification and validation of the first-order, $10$-dimensional SSM with RBF-approximated SSM-reduced dynamics.}
  \label{fig:CorrDimTrainTest}
\end{figure}

For the delay-coordinate map, the lag is chosen to coincide with the zero-order-hold sampling time, namely $\Delta t = 20~\mathrm{ms}$, to retrieve the intersections with the Poincar\'e section. To improve accuracy, we use an overembedding dimension of $15$.

As discussed in the previous examples, Figure~\ref{fig:PDCombined} shows short-time trajectory predictions, the PDFs of reduced coordinates and a three-dimensional projection of the delay embedding, where the two test trajectories are seen to approach the chaotic invariant set.

\begin{figure}[H]
  \centering

  \begin{minipage}[t]{0.54\textwidth}
    \centering
    \includegraphics[width=\linewidth]{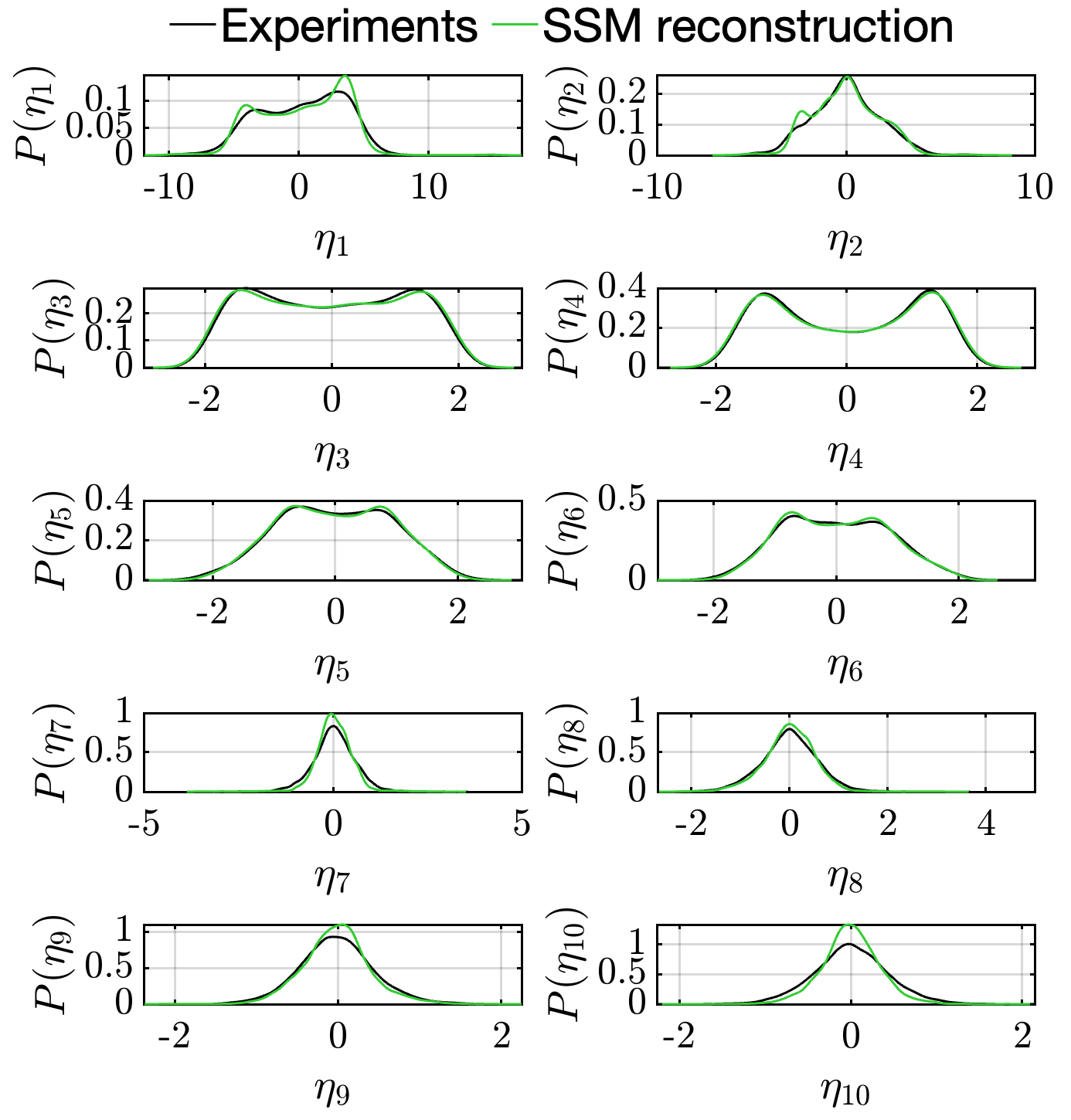}
  \end{minipage}\hfill
  \begin{minipage}[t]{0.45\textwidth}
    \centering
    \includegraphics[width=\linewidth]{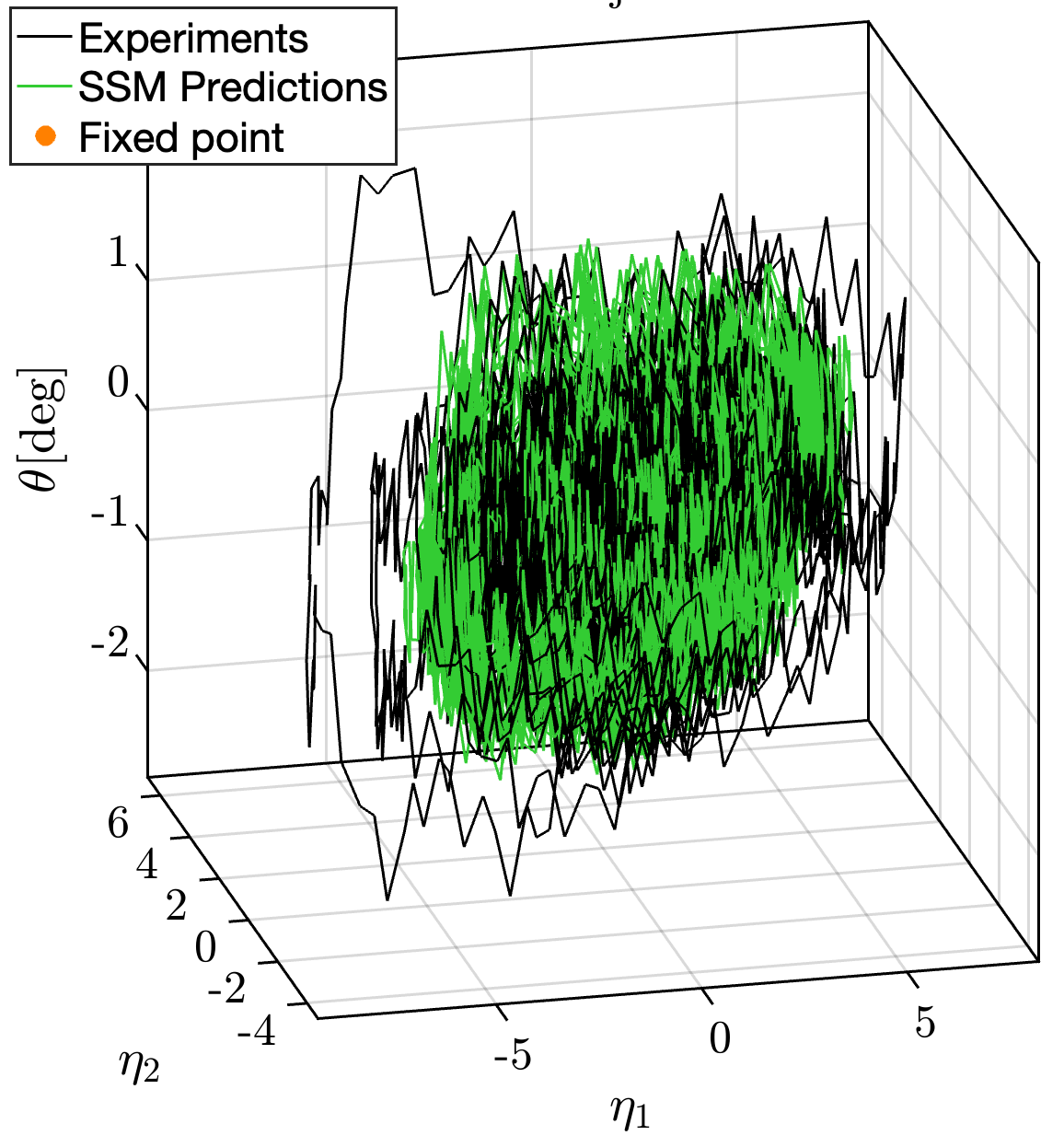}
  \end{minipage}

  \vspace{0.8em}

  \begin{minipage}[t]{0.92\textwidth}
    \centering
    \includegraphics[width=\linewidth]{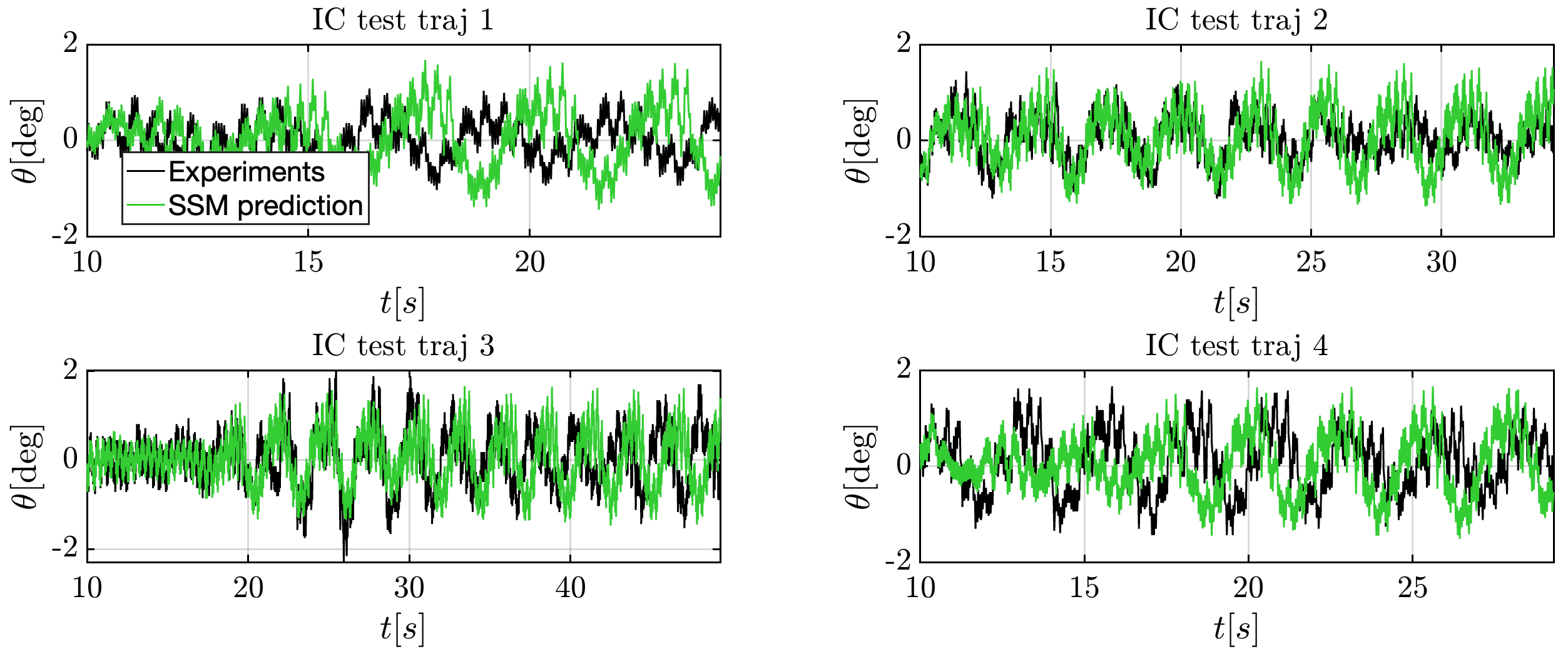}
  \end{minipage}

  \caption{SSM-based model predictions for the Furuta pendulum experiments on previously unseen trajectory data. Top left: probability density functions of the reduced coordinates of training trajectories from experimental data and their reconstruction from the RBF-based SSM model. Top right: Two full system's test trajectories (in black) evolving on the chaotic attractor together with their reduced model predictions (in green). Bottom: short-time predictions for two test trajectories, illustrating attraction toward the chaotic invariant set.}
  \label{fig:PDCombined}
\end{figure}

%% file: Sections/5_Conclusions.tex
\section{Conclusions }\label{section6}
We have shown that a data-driven approach based on a recent extension of spectral submanifold (SSM) theory to delay differential equations (DDEs) yields predictive reduced-order models for autonomous systems with bounded delays. These SSM-reduced models remain accurate even in the presence of chaotic dynamics and without prior knowledge of the number or magnitude of delays.

Furthermore, parametric SSM-reduction provides models that remain predictive across both local and global bifurcations, while revealing the mechanisms underlying loss of stability in the infinite-dimensional phase space of the underlying and possibly unknown DDE.

In the extension of the theory to nonautonomous DDEs with periodic delay, a $d$-dimensional family of SSMs exists, parametrized by time on a periodic interval, all of which are tangent to a time-periodic spectral subspace and with at least $C^1$ regularity. We have recalled the underlying theory in Theorem~\ref{thm:main1} and extended it to systems with periodic delays in Theorem~\ref{thm:main}, with a proof in Appendix C. This extension enables the rigorously justified reduction of experimental data from a PD-controlled mechanical system subject to feedback delay and the quantization effect of the digital controller.

We note that all results presented here concern bounded delays. Extensions to systems with unbounded delay, or memory effects, are an active area of research. Theoretical foundations can be found in \parencite{HinoMurakamiNaito1991}.

%% file: Sections/6_Aknowledgments.tex
\section{Aknowledgments }\label{section7}
This work was supported by the Swiss National Science Foundation (Grant No. $200021\_214908$). The authors are grateful to Professor Gábor Stépán for making the experimental data publicly available.

%% file: Sections/AppendixA.tex
\section{Data-driven spectral submanifolds for autonomous systems }\label{appendixA}
The theory of SSMs and their data-driven reconstruction via delay embedding is well established (see \parencite{Haller2025}). Here, we summarize only the essentials to the present analysis.

\subsection{SSM reconstruction via delay embedding}\label{delayEmb}
Full-state measurements are typically unavailable \parencite{Cenedese2022, Liu2024}, motivating reconstruction from limited observables. By Takens' theorem \parencite{Takens1981}, a $d$-dimensional manifold can be embedded using $m>2d$ delayed samples of a generic scalar observable \parencite{Axs2023}. 

The reduced dynamics
\[
\dot{\eta} = R(\eta), \qquad \eta \in \mathbb{R}^d,
\]
on a $d$-dimensional SSM $\mathcal{W}(E)$, tangent to a spectral subspace $E$, with associated flow $R^t : \eta_0 \mapsto \eta(t)$, are conjugate to the restriction of the full system flow $F^t : \mathbb{R}^n \to \mathbb{R}^n$ to $\mathcal{W}(E) \subset \mathbb{R}^n$. This conjugacy is realized through a smooth parametrization $M : \mathbb{R}^d \to \mathbb{R}^n$, satisfying
\begin{equation}
    F^t \circ M = M \circ R^t.
\end{equation}

Consider a scalar time series $s(t) = \xi(x(t))$, obtained by observing the full system, where $\xi : \mathbb{R}^n \to \mathbb{R}$. The delay-coordinate map $\Psi : \mathbb{R}^n \to \mathbb{R}^m$ is constructed by stacking $m$ uniformly spaced samples of $s(t)$ with sampling interval $T_s$:
\begin{equation}\label{delayEmbCoordMap}
    y(t) = \Psi(x(t))
    = [\,s(t),\, s(t+T_s),\, \dots,\, s(t+(m-1)T_s)\,]^\mathrm{T}
    \in \mathbb{R}^m.
\end{equation}

Denote by $F^t_\Psi : \mathbb{R}^m \to \mathbb{R}^m$ the induced flow in the delay embedding space. If $\xi(0)=0$, then the equilibrium $x=0$ maps to the fixed point $y=0$, i.e., $F^t_\Psi(0)=0$. For a generic observable $\xi$, and assuming nondegeneracy conditions on the SSM-reduced dynamics, Takens' theorem guarantees that for $m>2d$, the restriction of $\Psi$ to $\mathcal{W}(E)$ yields an embedding into $\mathbb{R}^m$ with probability one. As a result, the image $\widetilde{\mathcal{W}}(E) = \Psi(\mathcal{W}(E))$ is diffeomorphic to $\mathcal{W}(E)$.

This results enables the identification of an SSM anchored at $y=0$ in delay space via the conjugacy between the embedded dynamics $F^t_\Psi$ on $\widetilde{\mathcal{W}(E)}$ and the reduced dynamics on $\mathcal{W}(E)$:
\begin{equation}
F^t_\Psi \circ \Psi = \Psi \circ F^t\bigl\lvert_{\mathcal{W}(E)}.
\end{equation}

\subsection{Data‐driven modeling of spectral submanifoldw in a delay embedding space}\label{ManFit}
Following \parencite{Liu2024}, we denote by $y \in \mathbb{R}^m$ the delay-embedded observations, and by $\eta \in \mathbb{R}^d$ the reduced coordinates associated with the spectral subspace $E$ of $D\Psi(0)E$. In a neighborhood of the fixed point, the embedded SSM can be locally represented as a polynomial graph tangent to this linear subspace.

Let $\mathcal{K}$ be the maximal degree of the expansion. We approximate $\mathcal{W}(E)$ through the parametrization
\begin{equation}\label{eq:M_eta}
M(\eta)
= V_1\,\eta + \sum_{|\mathbf{k}|=2}^{\mathcal{K}} V_{\mathbf{k}}\,\eta^{\mathbf{k}},
\qquad 
\mathbf{k}=(k_1,\dots,k_d)\in\mathbb{N}^d,\quad
V_1,V_{\mathbf{k}}\in\mathbb{R}^{m\times d},
\end{equation}
where $\eta^{\mathbf{k}}=\eta_1^{k_1}\cdots\eta_d^{k_d}$ denotes the monomial associated with the multi-index $\mathbf{k}$.

The reduced coordinates are obtained via projection, $\eta = V_1^\mathrm{T} y$, where $V_1 \in \mathbb{R}^{m\times d}$ has orthonormal columns spanning the tangent space of $\widetilde{\mathcal{W}(E)} \subset \mathbb{R}^m$ at $y=0$. The coefficients are identified using SSMLearn \parencite{Cenedese2022} by solving
\begin{equation}\label{eq:Vopt}
V^* = [\,V_1^*,\,V_{2:\mathcal{K}}^*\,]= \argmin_{V_1,\,\{V_{\mathbf{k}}\}_{|\mathbf{k|}\ge2}}
   \sum_j
   \Big\lVert
     y_j 
     - V_1\,V_1^\mathrm{T} y_j 
     - \sum_{|\mathbf{k}|=2}^{\mathcal{K}} V_{\mathbf{k}}\,(\eta_j)^{\mathbf{k}}
   \Big\rVert_2^2  \qquad 
\end{equation}

subject to $V_1^\mathrm{T}V_1 = I$ and $V_1^\mathrm{T}V_{\mathbf{k}} = 0$ for all $|\mathbf{k}|\ge2$.

\subsection{Data‐driven modeling of the SSM-reduced dynamics}\label{RedDynFit}
In the present work, we approximated the reduced dynamics on the SSM using either polynomial regression or radial basis functions (RBFs) interpolation.

\subsubsection*{Polynomial reduced dynamics}
We represent the reduced dynamics in polynomial form as
\begin{equation}\label{eq:reduced-dynamics}
\dot{\eta} 
= R_{1}\,\eta 
+ \sum_{|\mathbf{k}|=2}^{\mathcal{K}} R_{\mathbf{k}}\,\eta^{\mathbf{k}},
\qquad 
\mathbf{k}=(k_1,\dots,k_d)\in\mathbb{N}^d,
\quad
R_{\mathbf{k}}\in\mathbb{R}^{d\times d},
\end{equation}
where the coefficient matrices are identified through a least-squares regression:
\begin{equation}\label{eq:regression}
R^{*}
=\argmin_{R}
\sum_{j}
\Bigl\|
\dot{\eta}_{j} 
-\sum_{|\mathbf{k}|=1}^{\mathcal{K}} R_{\mathbf{k}}\,(\eta_{j})^{\mathbf{k}}
\Bigr\|_2^2,
\qquad 
\mathbf{k}=(k_1,\dots,k_d)\in\mathbb{N}^d, \quad R = \bigl[\,R_{1},\,R_{2:\mathcal{K}}\,\bigr]
\;\in\;
\mathbb{R}^{d\times\sum_{k=1}^{\mathcal{K}}d_{k}}.
\end{equation}
Here, the number of monomials of total degree $k$ is $d_k = \binom{d+k-1}{k}$.

\subsubsection*{RBF reduced dynamics}
The SSM-reduced dynamics can be approximated using RBFs, following \parencite{Xu2024}, in the form of a discrete-time map:
\begin{equation}
\eta_{n+1}
= F(\eta_n)
= \sum_{i=1}^K C_i\,k\bigl(\|\eta_n - \eta_i\|\bigr).
\end{equation}
Here, $k$ is a radial kernel that depends on the distance between $\eta_n$ and the sampled points $\eta_i$. In this work, we adopt the linear kernel $k(r)=r$ due to its simplicity and effectiveness \parencite{Xu2024, KaszasHaller2025, AbbascianoEndreszStepanHaller2026JSV}.

\subsection{Parametric SSM-reduced order models}
About nonresonant points, the SSM and its associated reduced dynamics inherit smooth parameter dependence from the system vector field \parencite{Haller2025, Haller2016}. This justifies constructing parameter-dependent SSMs via polynomial regression or smooth interpolation. This is achieved by interpolating the coefficients of the multivariate polynomial expansions describing the reduced space, the manifold and its reduced dynamics (e.g.\ using splines or linear interpolation, as in this paper), resulting in coefficient functions of $\mu \in \mathbb{R}^p$, following \parencite{AbbascianoEndreszStepanHaller2026JSV}.

The question of persistence and regularity of SSM across resonant points is addressed in \parencite{KingEtAl2026ParametricSSM}. As mentioned therein, the limitations that resonance bring about are less pronounced in the data-driven setting. The corresponding reduced-order models can capture local bifurcations of equilibria and invariant tori, as well as global bifurcation phenomena in a neighborhood of the anchor point.

For illustration, in the case of a two-dimensional SSM ($d=2$), the parameter-dependent reduced dynamics in modal coordinates can be expressed as
\begin{equation}
\begin{aligned}
\dot\eta_1 &= \lambda_1(\mu)\eta_1 
+ a_{11}(\mu)\eta_1^2 
+ a_{12}(\mu)\eta_1\eta_2
+ a_{13}(\mu)\eta_2^2
+ \mathcal{O}\bigl(|\eta|^3\bigr),\\
\dot\eta_2 &= \lambda_2(\mu)\eta_2 
+ a_{21}(\mu)\eta_1^2 
+ a_{22}(\mu)\eta_1\eta_2
+ a_{23}(\mu)\eta_2^2
+ \mathcal{O}\bigl(|\eta|^3\bigr).
\end{aligned}
\end{equation}

%% file: Sections/AppendixB.tex
\section{Equation-driven reduction of the Hutchinson delay system }\label{appendixB}

\subsection{Formal procedure for SSM-reduction employed in \parencite{Szaksz2025}}\label{SectionB1}
We present the equation-driven SSM reduction of the Hutchinson delay system following \parencite{SzakszOroszStepan2024Traffic,Szaksz2025} based on previous work \parencite{StepanHaller1995, Diekmann1995}. The Hutchinson equation, taken from \parencite{BredaEtAl2025DDMethods}, is
\begin{equation}\label{eq:hutch}
x'(t)=r\,x(t)\left(1-\frac{x(t-\tau)}{K}\right),
\qquad r>0,\ K>0,\ \tau>0.
\end{equation}
Its equilibria are $x_*\in\{0,K\}$. Shifting the equilibrium $x_*=K$ to the origin, we obtain
\begin{equation}\label{eq:HutchDDE}
\dot y(t)=L\,y(t)+R\,y(t-\tau)+N(y(t),y(t-\tau)),
\end{equation}
with $L=0$, $R=-r$, and $N(u,v)=-\frac{r}{K}\,u\,v$. 
The state at time $t$ is the history segment
$y_t(\vartheta)=y(t+\vartheta)$, $\vartheta\in[-\tau,0]$,
with phase space $X:=C([-\tau,0],\mathbb{R})$.

For $r=1.8$, $K=10$, and $\tau=1$, the dominant eigenvalues of the linearized system \eqref{eq:HutchDDE} form an unstable complex-conjugate pair. The existence of a $C^\infty$ SSM follows from \parencite{BuzaHaller2025}, since the corresponding spectral subspace is spanned by eigenvectors associated with a simple unstable complex pair.
We construct a two-dimensional SSM associated with the dominant complex-conjugate eigenvalue pair and use one complex reduced coordinate $z\in\mathbb{C}$, corresponding to two real dimensions via $(\Re z,\Im z)$.

The characteristic matrix is
\begin{equation}\label{eq:Delta}
\Delta(\mu):=\mu I-L-Re^{-\mu \tau}.
\end{equation}
Computing the infinitely many roots of the quasi-polynomial~\eqref{eq:Delta} is nontrivial. To this end, we use the method developed in~\parencite{AppeltansSilmMichiels2022,AppeltansMichiels2023} to compute the dominant part of the spectrum, as shown in Figure~\ref{fig:Eigenvalues_H}. This yields
$\lambda_1 = 0.097 - 1.6i, \qquad \lambda_2 = \overline{\lambda_1} = 0.097 + 1.6i$, while all other eigenvalues satisfy $\Re(\lambda_j) < \Re(\lambda_1)$.
To construct a third-order polynomial expansion of the manifold, we must verify that the nonresonance conditions up to cubic order are satisfied, namely,
$\Delta(j\lambda_1+k\lambda_2)\neq 0$ for all integer pairs $(j,k)$ with $j+k\in\{2,3\}$. Equivalently, $\Delta(2\lambda_1)$, $\Delta(\lambda_1+\lambda_2)$, $\Delta(2\lambda_2)$, $\Delta(3\lambda_1)$, $\Delta(2\lambda_1+\lambda_2)$, $\Delta(\lambda_1 + 2\lambda_2)$, and $\Delta(3\lambda_2)$ must all be nonzero. This prevents division by zero when solving the homological equations.

\begin{comment}
The right eigenvector associated with $\lambda_1$, satisfying $\Delta(\lambda_1)q_1=0$ with nontrivial $q_1 \in \mathbb{R}$, can be chosen as any nonzero scalar. We set $q_1=1$ without loss of generality; hence $q_2=\bar{q}_1=1$. Similarly, the left eigenvector associated with $\lambda_1$, satisfying $p_1^*\Delta(\lambda_1)=0$ with nontrivial $p_1^* \in \mathbb{R}$, is any nonzero scalar, with $p_2^*=\bar{p}_1^*$.
\end{comment}

The right eigenfunctions are
\begin{equation}\label{eq:phi12}
\phi_1(\vartheta)=q_1e^{\lambda_1\vartheta},
\qquad
\phi_2(\vartheta)=q_2e^{\lambda_2\vartheta},
\qquad \vartheta\in[-\tau,0], \qquad \text{where } q_1=\bar q_2\in \mathbb{C}\setminus\{0\}.
\end{equation}
For convenience, we set $q_1=1$.
The left eigenfunctions are
\begin{equation}\label{eq:psi1}
\psi_1(\theta)=p_1^*\left(I+R e^{-\lambda_1\tau}\frac{e^{\lambda_1\theta}-1}{\lambda_1}\right), \qquad \psi_2=\bar{\psi}_1,
\qquad \theta\in[0,\tau], \qquad p_1^*\in \mathbb{C}\setminus\{0\}.
\end{equation}

\paragraph{Normalization.}
Imposing $\langle\psi_i,\phi_i\rangle=1$ for $i\in\{1,2\}$ fixes the remaining degree of freedom: $p_i^*=1/(\Delta'(\lambda_i)q_i)$, where the pairing is defined by
\begin{equation}\label{eq:pairing}
\langle \Psi,\Phi\rangle
=
\Psi(0)\,\Phi(0)
+
\int_0^\tau \Psi'(\theta)\,\Phi(-\theta)\,d\theta.
\end{equation}
Henceforth, $\psi_1$ denotes the normalized adjoint eigenfunction.

% ============================================================
\paragraph{SSM parametrization in history space.}
% ============================================================
We parametrize histories on the two-dimensional SSM by $z\in \mathbb{C}$:
\[
y_t(\cdot)=W(z(t),\bar z(t),\cdot).
\]
The reality of the manifold requires
$W( \bar z,{z},\vartheta)=\overline{W(z,\bar z,\vartheta)}$.
We expand up to cubic order,
$W = W_1+W_2+W_3+O(|z|^4)$,
with
\begin{equation}\label{eq:W1W2W3}
\begin{aligned}
W_1(z,\bar z,\vartheta)&=W_{10}(\vartheta)z+W_{01}(\vartheta)\bar z
=\phi_1(\vartheta)z+\phi_2(\vartheta)\bar z,
\\[1mm]
W_2(z,\bar z,\vartheta)&=\frac12\Big(W_{20}(\vartheta)z^2+2W_{11}(\vartheta)z\bar z+W_{02}(\vartheta)\bar z^2\Big),
\\[1mm]
W_3(z,\bar z,\vartheta)&=\frac{1}{6}\Big(W_{30}(\vartheta)z^3+3W_{21}(\vartheta)z^2\bar z+3W_{12}(\vartheta)z\bar z^2+W_{03}(\vartheta)\bar z^3\Big).
\end{aligned}
\end{equation}
With $\dot z=f(z,\bar z)$ and $\dot{\bar z}=\overline{f(z,\bar z)}$, the reduced dynamics truncated at cubic order are
\begin{equation}\label{eq:zdot-trunc30}
\begin{aligned}
\dot z
&=\lambda z
+\beta_{20}z^2+\beta_{11}z\bar z+\beta_{02}\bar z^2
+\beta_{30}z^3+\beta_{21}z^2\bar z+\beta_{12}z\bar z^2+\beta_{03}\bar z^3
+O(|z|^4),
\\
\dot{\bar z} 
&=\bar{\lambda}\bar z
+\overline{\beta_{02}} z^2+\overline{\beta_{11}}z\bar z+\overline{\beta_{20}}\bar{z}^2
+\overline{\beta_{03}} z^3+\overline{\beta_{12}}\bar  z z^2+\overline{\beta_{21}}\bar  z^2 z+\overline{\beta_{30}}\bar z^3
+O(|z|^4).
\end{aligned}
\end{equation}
The invariance of the spectral submanifold can be expressed by the PDE
\begin{equation}
    \frac{\partial}{\partial t}y_t=\frac{\partial}{\partial \vartheta }y_t,
    \end{equation}
which yields equations in the interior, $\vartheta \in [-\tau,0)$, and at the boundary, $\vartheta=0$.

\paragraph{Interior invariance PDE.}
$\partial_\vartheta W = \partial_zW\,\dot z+\partial_{\bar z}W\,\dot{\bar z} \quad , \quad \vartheta\in[-\tau,0).$

The left-hand side, up to cubic order, can be computed from \eqref{eq:W1W2W3} as
\begin{equation}\label{eq:dthW}
\begin{aligned}
\partial_\vartheta W
&=\partial_\vartheta W_1+\partial_\vartheta W_2+\partial_\vartheta W_3+O(|z|^4)
\\
&=\Big(W_{10}'z+W_{01}'\bar z\Big)
+\frac12\Big(W_{20}'z^2+2W_{11}'z\bar z+W_{02}'\bar z^2\Big)+\frac{1}{6}\Big(W_{30}'z^3+3W_{21}'z^2\bar z+3W_{12}'z\bar z^2+W_{03}'\bar z^3\Big)
+O(|z|^4).
\end{aligned}
\end{equation}
For the right-hand side, we obtain
\begin{equation}\label{eq:interior-full-expanded}
\begin{aligned}
\partial_z W\,\dot z
&= \lambda W_{10} z +\Big(
    (W_{10}\beta_{20}+\lambda W_{20}) z^2
  + (W_{10}\beta_{11}+\lambda W_{11}) z\bar z
  + W_{10}\beta_{02}\,\bar z^2
  \Big)+ \Big(
    (W_{10}\beta_{30}+W_{20}\beta_{20}+\tfrac{\lambda}{2} W_{30}) z^3 + \\
&
  + (W_{10}\beta_{21}+W_{20}\beta_{11}+W_{11}\beta_{20}+\lambda W_{21}) z^2\bar z + (W_{10}\beta_{12}+W_{20}\beta_{02}+W_{11}\beta_{11}+\tfrac{\lambda}{2} W_{12}) z\bar z^2 +  \\
  &
  + (W_{10}\beta_{03}+W_{11}\beta_{02}) \bar z^3
  \Big)
+ O(|z|^4),
\\[2mm]
\partial_{\bar z} W\,\dot{\bar z}
&= \bar\lambda W_{01}\,\bar z + \Big(
    W_{01}\bar\beta_{02}\, z^2
  + (W_{01}\bar\beta_{11}+\bar\lambda W_{11}) z\bar z
  + (W_{01}\bar\beta_{20}+\bar\lambda W_{02}) \bar z^2
  \Big) + \Big(
    (W_{01}\bar\beta_{03}+W_{11}\bar\beta_{02}) z^3 + \\
&
  + (W_{01}\bar\beta_{12}+W_{11}\bar\beta_{11}+W_{02}\bar\beta_{02}
     +\tfrac{\bar\lambda}{2} W_{21}) z^2\bar z + (W_{01}\bar\beta_{21}+W_{11}\bar\beta_{20}+W_{02}\bar\beta_{11}
     +\bar\lambda W_{12}) z\bar z^2 +\\
&
  + (W_{01}\bar\beta_{30}+W_{02}\bar\beta_{20}
     +\tfrac{\bar\lambda}{2} W_{03}) \bar z^3
  \Big)
+ O(|z|^4).
\end{aligned}
\end{equation}

\paragraph{Boundary invariance PDE.}
$ \partial_zW(0)\dot z+\partial_{\bar z}W(0)\dot\bar z = L\,W(0)+R\,W(-\tau)+N\big(W(0),W(-\tau)\big) $ at $\vartheta=0$.

The expansion of the nonlinearity up to cubic order is obtained by substituting \eqref{eq:W1W2W3}. 

\begin{comment} Let $N:\mathbb{R}^2\to\mathbb{R}$ be $C^3$, with
$N(u,v)=-\frac{r}{K}\,u\,v$, $N(0,0)=0$, and $DN(0,0)=0$,
and
\end{comment} 

Define $\phi_{1,\tau}=e^{-\lambda\tau}$ and $\phi_{2,\tau}=e^{-\bar\lambda\tau}$. Truncating $W$ at second order is sufficient for a third-order expansion of $N$:
\begin{equation}\label{eq:Nsub-direct-collected-to3}
\begin{aligned}
N\big(& W(0),W(-\tau)\big)
= -\frac{r}{K}
\Big(
    \phi_{1,\tau} z^2
  + (\phi_{1,\tau}+\phi_{2,\tau}) z\bar z
  + \phi_{2,\tau} \bar z^2
\Big)
-\frac{r}{2K}
\Big(
    \big(W_{20}(-\tau)+\phi_{1,\tau}W_{20}(0)\big) z^3 +
\\
&
  + \big(W_{20}(-\tau)+2W_{11}(-\tau)
      +2\phi_{1,\tau}W_{11}(0)
      +\phi_{2,\tau}W_{20}(0)\big) z^2\bar z
      \\
&
+ \big(W_{02}(-\tau)+2W_{11}(-\tau)
      +\phi_{1,\tau}W_{02}(0)
      +2\phi_{2,\tau}W_{11}(0)\big) z\bar z^2
 + \big(W_{02}(-\tau)+\phi_{2,\tau}W_{02}(0)\big) \bar z^3
\Big)
+ O(|z|^4).
\end{aligned}
\end{equation}
We compute approximate solutions of the invariance PDE by solving the associated homological equations obtained from the interior and boundary invariance conditions order by order. This procedure yields the coefficients of both the manifold parametrization and the reduced dynamics up to the desired order, here cubic.

\paragraph{Order 1.}
Solving the interior invariance equation at first order, using \eqref{eq:dthW} and \eqref{eq:interior-full-expanded}, yields
\begin{equation} 
W_{10}(\vartheta)=q_{1}e^{\lambda\vartheta},\qquad
W_{01}(\vartheta)=q_{2}e^{\bar{\lambda}\vartheta}.
\end{equation}
Substitution into the boundary invariance equations, derived from \eqref{eq:interior-full-expanded}--\eqref{eq:Nsub-direct-collected-to3}, leads to
\[
[z]:\quad \lambda = L + R e^{-\lambda\tau}\quad\Longleftrightarrow\quad \Delta(\lambda)=0,
\qquad
[\bar z]:\quad \bar{\lambda}=L+Re^{-\bar{\lambda}\tau}\quad\Longleftrightarrow\quad \Delta(\bar{\lambda})=0,
\]
which ensures tangency to the spectral subspace spanned by the right eigenfunctions in $X$.

\paragraph{Order 2.}
From \eqref{eq:dthW} and \eqref{eq:interior-full-expanded}, we derive the second-order interior invariance equations, solve them by variation of constants, evaluate them at $\vartheta \in\{-\tau,0\}$, and substitute the resulting expressions into the boundary equations. Together with the following projection conditions, which we arbitrarily impose to fix the coefficients $\beta_{ij}$ by removing the second-order component of the manifold along the spectral subspace directions:
\begin{equation}\label{eq:bi-gauge}
\langle\psi_1,W_{jk}\rangle=0,\qquad \langle\psi_2,W_{jk}\rangle=0,
\qquad j+k=2,
\end{equation}
this yields, for each monomial, a $3\times 3$ linear system with unknowns $(W_{jk}(0), \beta_{jk}, \overline{\beta}_{jk})$. The projections \eqref{eq:bi-gauge} are conditions chosen arbitrarily to render the system determined. This choice forces the component tangent to the spectral subspace entirely into $\beta_{jk}$ rather than $W_{jk}$.

\paragraph{Notation for the projections.}
For any scalar exponential, define
\[
\Pi_i(a):=\langle\psi_i,e^{a\vartheta}\rangle=\frac{1}{\Delta'(\lambda_i)}
\left[
1+R\,e^{-\lambda_i\tau}\,
\frac{e^{(\lambda_i-a)\tau}-1}{\lambda_i-a}
\right],
\qquad a\neq \lambda_i,\qquad i=1,2.
\]
In the limit \(a\to\lambda_i\),
\[
\Pi_i(\lambda_i)
=\frac{1}{\Delta'(\lambda_i)}
\left[
1+R\,\tau\,e^{-\lambda_i\tau}
\right]=1.
\]
As an example, we compute $(W_{20}(0), \beta_{20}, \overline{\beta}_{02})$. The boundary equation is
\begin{equation}
    \underbrace{\Big(\lambda-\tfrac{L}{2}-\tfrac{R}{2}\phi_{1,\tau}^{2}\Big)}_{=:a_{20}=\frac{1}{2}\Delta(2\lambda)\rightarrow\text{nonres. cond.}}\,W_{20}(0)
+\underbrace{\Big(1-\tfrac{R}{\lambda}(\phi_{1,\tau}^{2}-\phi_{1,\tau})\Big)}_{=:b_{20}}\,\beta_{20}
+\underbrace{\Big(1-\tfrac{R}{\overline{\lambda}-2\lambda}(\phi_{2,\tau}-\phi_{1,\tau}^{2})\Big)}_{=:c_{20}}\,\overline{\beta_{02}}
=\underbrace{-\tfrac{r}{K}\phi_{1,\tau}}_{=:d_{20}}.
\end{equation}
The projection conditions read
\begin{equation}
0=\langle\psi_i,W_{20}(\vartheta)\rangle
=\Pi_i(2\lambda)\,W_{20}(0)
+\frac{2\beta_{20}}{\lambda}\big(\Pi_i(2\lambda)-\Pi_i(\lambda)\big)
+\frac{2\overline{\beta_{02}}}{\overline{\lambda}-2\lambda}\big(\Pi_i(\overline{\lambda})-\Pi_i(2\lambda)\big),
\quad i\in\{1,2\}.
\end{equation}
This leads to the following linear system for $(W_{20}(0), \beta_{20}, \overline{\beta}_{02})$:
\begin{equation}\label{eq:sys-20-3x3}
\begin{bmatrix}
\frac{1}{2}\Delta(2\lambda) &
b_{20} &
c_{20}
\\[1mm]
\Pi_1(2\lambda) &
\frac{2}{\lambda}\big(\Pi_1(2\lambda)-\Pi_1(\lambda)\big) &
\frac{2}{\overline{\lambda}-2\lambda}\big(\Pi_1(\overline{\lambda})-\Pi_1(2\lambda)\big)
\\[2mm]
\Pi_2(2\lambda) &
\frac{2}{\lambda}\big(\Pi_2(2\lambda)-\Pi_2(\lambda)\big) &
\frac{2}{\overline{\lambda}-2\lambda}\big(\Pi_2(\overline{\lambda})-\Pi_2(2\lambda)\big)
\end{bmatrix}
\begin{bmatrix}
W_{20}(0)\\ \beta_{20}\\ \overline{\beta_{02}}
\end{bmatrix}
=
\begin{bmatrix}
d_{20}\\ 0\\ 0
\end{bmatrix}.
\end{equation}
The interior equation then reconstructs $W_{20}(\vartheta)$. The same procedure applies to $(W_{11}(0),\beta_{11},\overline{\beta}_{11})$ and $(W_{02}(0),\beta_{02},\overline{\beta}_{20})$.

\paragraph{Order 3.}
Proceeding as at second order, the cubic interior equations, derived from \eqref{eq:dthW} and \eqref{eq:interior-full-expanded}, are solved by variation of constants, evaluated at $\vartheta\in\{-\tau,0\}$, and substituted into the boundary equations. Together with \eqref{eq:bi-gauge} applied for $j+k=3$, each monomial yields a $3\times 3$ linear system for $(W_{jk}(0), \beta_{jk}, \overline{\beta}_{jk})$, for all $j+k=3$.

Again, as an example, we present the computation for $(W_{30}(0),\beta_{30},\overline{\beta}_{03})$. The boundary equation is
\begin{equation}
\Big(\tfrac{\lambda}{2}-\tfrac{L}{6}\Big)W_{30}(0)+\beta_{30}+\overline{\beta_{03}}
-\tfrac{R}{6}W_{30}(-\tau)
=
D_{30},
\end{equation}
with projection constraints
\begin{equation}\label{eq:gauge-W30-general}
0=\langle\psi_i,W_{30}\rangle =
\Pi_i(3\lambda)W_{30}(0)+\mathcal J^{(i)}_{30} 
+6\beta_{30}\frac{\Pi_i(3\lambda)-\Pi_i(\lambda)}{2\lambda}
+6\overline{\beta_{03}}
\frac{\Pi_i(\overline{\lambda})-\Pi_i(3\lambda)}
{\overline{\lambda}-3\lambda}.
\end{equation}

Hence, the linear system for $(W_{30}(0),\beta_{30},\overline{\beta}_{03})$ reads
\begin{equation}\label{eq:sys-30-3x3-general}
\begin{bmatrix}
A_{30} & B_{30} & C_{30}\\[1mm]
\Pi_1(3\lambda) &
6\frac{\Pi_1(3\lambda)-\Pi_1(\lambda)}{2\lambda} &
6\frac{\Pi_1(\overline{\lambda})-\Pi_1(3\lambda)}{\overline{\lambda}-3\lambda}
\\[2mm]
\Pi_2(3\lambda) &
6\frac{\Pi_2(3\lambda)-\Pi_2(\lambda)}{2\lambda} &
6\frac{\Pi_2(\overline{\lambda})-\Pi_2(3\lambda)}{\overline{\lambda}-3\lambda}
\end{bmatrix}
\begin{bmatrix}
W_{30}(0)\\ \beta_{30}\\ \overline{\beta_{03}}
\end{bmatrix}
=
\begin{bmatrix}
E_{30}\\[1mm]
-\mathcal J^{(1)}_{30}\\[1mm]
-\mathcal J^{(2)}_{30}
\end{bmatrix}.
\end{equation}
where
\[
\begin{aligned}
A_{30}&:=\Big(\tfrac{\lambda}{2}-\tfrac{L}{6}\Big)-\tfrac{R}{6}\phi_{1,\tau}^3=\cfrac{1}{6}\Delta (3\lambda), \qquad 
B_{30}:=1-R\,\frac{\phi_{1,\tau}^3-\phi_{1,\tau}}{2\lambda},\qquad 
C_{30}:
=1-R\,\frac{\phi_{2,\tau}-\phi_{1,\tau}^3}{\overline{\lambda}-3\lambda}, \qquad 
\\
D_{30}&:=
-\frac{r}{2K}\Big(W_{20}(-\tau)+\phi_{1,\tau}W_{20}(0)\Big)
-\Big(W_{20}(0)\beta_{20}+W_{11}(0)\overline{\beta_{02}}\Big),\qquad E_{30}:=D_{30}+\tfrac{R}{6}M_{30}.
\end{aligned}
\]
Here, $D_{30}$ depends only on second-order quantities, and the known forcing terms are
\begin{equation}
\begin{aligned}
F_{30}(\vartheta) :=6\Big(W_{20}(\vartheta)\beta_{20}+W_{11}(\vartheta)\overline{\beta_{02}}\Big), \quad M_{30}&:=\int_{0}^{-\tau}e^{3\lambda(-\tau-s)}F_{30}(s)\,ds, \quad  \mathcal J^{(i)}_{30}&:=\left\langle\psi_i,\ \int_{0}^{\vartheta}e^{3\lambda(\vartheta-s)}F_{30}(s)\,ds\right\rangle.
\end{aligned}
\end{equation}

The interior equation then reconstructs $W_{30}(\vartheta)$.
After solving all homological equations, we verify the conditions implied by the reality of the manifold, $W_{jk}(\vartheta)=\overline{W_{kj}(\vartheta)}$. Higher-order expansions can be derived analogously, but the number of coefficients grows rapidly, especially with increasing SSM dimension. This motivates data-driven formulations even when the governing equations are available.

\subsection{Alternative procedure for SSM-reduction developed in \parencite{BuzaHaller2025}}\label{B2}
The continuous embedding $\iota :X\hookrightarrow X^{\odot \star}$ defined by $\varphi\mapsto(\varphi(0),\varphi)$, 
where the sun-star subspace $X^{\odot \star}=\mathbb{R}\times L^\infty([-\tau,0];\mathbb{R})$ is introduced in \parencite{Diekmann1995}, allows us to transition from the update rule \eqref{eq:hutch1} into the evolution equation \eqref{eq:eqXodotstar}
\begin{subequations}\label{eq:eqXodotstar}
\begin{align}
\iota \cfrac{d}{dt}y_t &= A^{\odot \star} \iota y_t + R(y_t), \label{eq:eqX_a}\\
y_0 &= \varphi, \qquad \varphi \in O. \label{eq:eqX_b}
\end{align}
\end{subequations}
where the linear part is, explicitly for system \eqref{eq:eqXodotstar},
\begin{subequations}
\begin{align}
\operatorname{dom}\!\left(A^{\odot *}\right)
&=
\left\{ (c,\varphi)\ \middle|\ \varphi \text{ is Lipschitz with } \varphi(0)=c \right\}, \label{eq:a} \\
A^{\odot *}(c,\varphi)
&=
\left(- r \varphi(-\tau),\dot{\varphi}\right). \label{eq:b}
\end{align}
\end{subequations}

while the nonlinear remainder is $R : X \to X^{\odot \star}$:
\begin{equation}
    R(\varphi) = \Big( -\frac{r}{K}\varphi(0)\varphi(-\tau) , 0\Big)
\end{equation}
so that $R\in C^\infty$, $R(0)=0$ and the first derivative vanishes, that is, $DR(0) = 0$.

We define the spectral set $\Sigma = \{\lambda, \bar{\lambda}\}$. For each eigenvalue $\mu$, the corresponding eigenfunction in $X$ has the form $w_\mu(\vartheta) = q_\mu e^{\mu \vartheta}$. Since the DDE is scalar, we fix $q_\mu=1$ and set
\[
w_1(\vartheta) = e^{\lambda \vartheta},\qquad
w_2(\vartheta) = e^{\bar{\lambda} \vartheta}.
\]

We precompute derivatives of the nonlinear part $R$, as they will be useful below. The second derivative is
\begin{equation}
\big(D^2R(0)[u,v]\big)
=
\Big(- \frac{r}{K}
\left(
u(0)v(-\tau) + v(0)u(-\tau)
\right),0\Big) \qquad u,v\in X.
\end{equation}
System \eqref{eq:eqXodotstar} has no cubic term, so $D^3R(0) = 0$. In particular,
\[
\big(D^2R(0)[w_1,w_1]\big)
=
\Big(-\frac{r}{K}
\left(
1\cdot e^{-\lambda\tau}
+
1\cdot e^{-\lambda\tau}
\right),0\Big) = \Big(-\frac{2r}{K} e^{-\lambda\tau},0\Big), 
\]
\[\big(D^2R(0)[w_1,w_2]\big)
=\Big(
-\frac{r}{K}
\left(
e^{-\lambda\tau}
+
e^{-\bar\lambda\tau}
\right),0\Big), \qquad
\big(D^2R(0)[w_2,w_2]\big)
=
\Big(-\frac{2r}{K} e^{-\bar\lambda\tau},0\Big).
\]

We now seek a two-dimensional SSM tangent to the spectral subspace
$X_\Sigma$. As in section \ref{SectionB1}, we wish to parametrize our reduced dynamics via a single complex coordinate $z\in\mathbb{C}$. Lemma 24 in \parencite{BuzaHaller2025}, however, is posed with reduced space taken as the spectral subspace $X_\Sigma$. To bridge this gap, we introduce $\Theta:\mathbb{C}\rightarrow X_\Sigma $ given by
\[
\Theta(z)=(\vartheta\mapsto zw_1(\vartheta) +\bar z w_2(\vartheta) ).
\]

The reduced dynamics may then be posed on $\mathbb{C}$, which will be of the general form 
\begin{equation}\label{eq:reddyn}
\dot z(z,\bar z)
=
g_1 z
+\frac12\big(g_{20}z^2+2g_{11}z\bar z+g_{02}\bar z^2\big)
+\frac16\Big(g_{30}z^3+3g_{21}z^2\bar z+3g_{12}z\bar z^2+g_{03}\bar z^3\Big)
+\mathcal O(4). 
\end{equation}

\subsection*{Computation of $W$}
The formulas provided in Lemma 24, \parencite{BuzaHaller2025}, extract $W$ as a map $X_\Sigma\rightarrow X$. To relate this to our reduced coordinate $z$, we need to precompose with $\Theta$, i.e., our final manifold embedding will be $W\circ\Theta:\mathbb{C}\rightarrow X$. 

\subsection*{Linear term $DW(0)$}
The linear part is given as 
\[
DW(0)=\iota_\Sigma :X_\Sigma \hookrightarrow X
\]
enforcing the graph style of parametrization.

\subsection*{Quadratic term $D^2W(0)$}

As in Lemma 24 of \parencite{BuzaHaller2025},
\begin{equation}\label{eq:D2K0-lemma}
D^2W(0)
=
\iota^{-1}
\int_0^\infty
T^{\odot\star}(s)\,P_{\Sigma'}^{\odot\star}\,
D^2R(0)\big[T(-s)\iota_\Sigma,\;T(-s)\iota_\Sigma\big]
\,ds.
\end{equation}

On the spectral subspace $X_\Sigma$, the linear evolution is given by
$T(-s)\big|_{X_\Sigma} = e^{-sA_\Sigma}$, so that
\begin{equation}\label{eq:Tminus-on-eigs}
T(-s)w_1 = e^{-\lambda s}w_1,\qquad
T(-s)w_2 = e^{-\bar\lambda s}w_2.
\end{equation}

Therefore, using \eqref{eq:Tminus-on-eigs} and bilinearity,
\begin{align}
D^2R(0)\big[T(-s)w_i,\;T(-s)w_j\big]
&= e^{-(\lambda_i+\lambda_j)s}\, D^2R(0)[w_i,w_j], \qquad (i,j)\in\{(1,1),(1,2),(2,2)\},
\end{align}
The three second-order directions in $X_\Sigma$ are $(w_1,w_1)$, $(w_1,w_2)$, and $(w_2,w_2)$. Applying \eqref{eq:D2K0-lemma} to each of these pairs gives
\begin{align}\label{eq:inclusionToBeeliminated}
D^2W(0)[w_i,w_j]
&=
\iota^{-1}\int_0^\infty
T^{\odot\star}(s)\,P_{\Sigma'}^{\odot\star}\,
e^{-(\lambda_i+\lambda_j)s}\, D^2R(0)[w_i,w_j]\;ds,
\qquad (i,j)\in\{(1,1),(1,2),(2,2)\},
\end{align}

Under a suitable spectral gap assumption, so that the integrals converge and the resolvents exist, we use the resolvent identity
\begin{equation}\label{eq:resolvent-trick-expanded}
\iota^{-1}\int_0^\infty e^{-\nu s}\,T^{\odot\star}(s)\,P_{\Sigma'}^{\odot\star}\,ds
=
P_{\Sigma'}\,\iota^{-1}\,(\nu I - A^{\odot\star})^{-1},
\end{equation}
where $\nu\in\mathbb C$ satisfies $\mathrm{Re}(\nu)>\sup\,\mathrm{Re}(\Sigma')$.
Applying \eqref{eq:resolvent-trick-expanded} to \eqref{eq:inclusionToBeeliminated} yields
\begin{align}\label{eq:Eqproj}
D^2W(0)[w_i,w_j]
&=
P_{\Sigma'}\,\iota^{-1}\,
\big((\lambda_i+\lambda_j) I - A^{\odot\star}\big)^{-1}
\,D^2R(0)[w_i,w_j],
\qquad (i,j)\in\{(1,1),(1,2),(2,2)\},
\end{align}

$D^2R(0)[w_i,w_j]$, viewed in $X^{\odot\star}=\mathbb R\times L^\infty([-\tau,0],\mathbb R)$, has the form $(c_{ij},0)$. The resolvent then becomes
\begin{equation}\label{eq:resolvent-on-c0}
(\nu I-A^{\odot\star})^{-1}(c,0)
=
\big(\Delta(\nu)^{-1}c,\ \vartheta\mapsto e^{\nu\vartheta}\Delta(\nu)^{-1}c\big),
\end{equation}
so that applying $\iota^{-1}$ simply extracts the history component:
\begin{equation}\label{eq:iota-inv-resolvent}
\iota^{-1}(\nu I-A^{\odot\star})^{-1}(c,0)
=
\vartheta\mapsto e^{\nu\vartheta}\Delta(\nu)^{-1}c.
\end{equation}

Therefore, defining
\[
U_\nu(\vartheta;c):= e^{\nu\vartheta}\Delta(\nu)^{-1}c,
\]
and using $P_{\Sigma'}=I-P_\Sigma$, equation \eqref{eq:Eqproj} becomes
\begin{align}
D^2W(0)[w_i,w_j](\vartheta)
&=
(I-P_\Sigma)\,
U_{\lambda_i+\lambda_j}\big(\vartheta;c_{ij}\big),
\qquad (i,j)\in\{(1,1),(1,2),(2,2)\},
\end{align}
where $P_\Sigma$ is the spectral projection onto
the 2D spectral subspace $X_\Sigma $.

We now need an explicit formula for $P_\Sigma$, that is, for $P_{\{\lambda\}}$ and $P_{\{\bar\lambda\}}$. 
We recall from \parencite{diekmann2012delay}, exercise IV.3.12,
\begin{equation}\label{eq:Proj}
{P}_{\{\lambda\}} u(\vartheta)
=
\frac{
p^{*}
\left[
u(0)
+
\int_{0}^{\tau} d\zeta(s)
\int_{0}^{s} e^{-\lambda \eta} u(\eta - s)\, d\eta
\right]
q
}{
p^{*}\Delta'(\lambda) q
}
e^{\lambda \vartheta}, \qquad u\in X.
\end{equation}
Here,
\begin{equation}
\int_{0}^{\tau} d\zeta(s)\, u(-s)
=
L u(0) + R u(-\tau)=-ru(-\tau).
\end{equation}
Let $q$ and $p^*$ denote, respectively, the right and left eigenvectors associated with $\lambda$, satisfying
\[
\Delta(\lambda)q = 0
\quad \text{and} \quad
p^*\Delta(\lambda) = 0.
\]

Substituting this into \eqref{eq:Proj}, we obtain
\begin{equation}\label{eq:projector_final_LR}
P_{\{\lambda\}}u(\vartheta)
=
\frac{
p^{*}\!\left[
u(0)-r\int_{0}^{\tau} e^{-\lambda \eta}\,u(\eta-\tau)\,d\eta
\right]\!q
}{
p^{*}\Delta'(\lambda)q
}\,e^{\lambda\vartheta}.
\end{equation}

Define the scalar functionals
\begin{equation}\label{eq:alpha-lambda-def}
\alpha_\lambda(u)
:=
\frac{
p^*\left[
u(0)\;-\;r\displaystyle\int_0^\tau e^{-\lambda s}\,u(s-\tau)\,ds
\right]q
}{
p^*\Delta'(\lambda)q
},
\qquad
\alpha_{\bar\lambda}(u)
:=
\frac{
\overline p^*\left[
u(0)\;-\;r\displaystyle\int_0^\tau e^{-\bar\lambda s}\,u(s-\tau)\,ds
\right]\overline q
}{
\overline p^*\Delta'(\bar\lambda)\overline q
}.
\end{equation}

Therefore, for $u(\vartheta)=C e^{\nu\vartheta}$,
\begin{equation}\label{eq:P-Sigma-exp}
(P_\Sigma u)(\vartheta)
=
\alpha_\lambda(Ce^{\nu\cdot})\,e^{\lambda\vartheta}
+
\alpha_{\bar\lambda}(Ce^{\nu\cdot})\,e^{\bar\lambda\vartheta}.
\end{equation}

\paragraph*{$(w_i,w_j)$ component.}
Thus,

\begin{align}
D^2W&(0)[w_i,w_j](\vartheta)
=
U_{\lambda_i+\lambda_j}(\vartheta;c_{ij})
\\
&\quad-
\Big(
\alpha_\lambda U_{\lambda_i+\lambda_j}(\vartheta;c_{ij}) e^{\lambda\vartheta}
+
\alpha_{\bar\lambda}U_{\lambda_i+\lambda_j}(\vartheta;c_{ij})e^{\bar\lambda\vartheta}
\Big),
\qquad (i,j)\in\{(1,1),(1,2),(2,2)\},
\end{align}

\subsection*{Linear term $DH(0)$}

From Lemma 24 in \parencite{BuzaHaller2025}, we have $DH(0)=A_\Sigma$. This translates to $g_1=\lambda$ in \eqref{eq:reddyn}.

\subsection*{Quadratic term $D^2H(0)$}

From Lemma 24 in \parencite{BuzaHaller2025},
\[
D^2H(0)=\widetilde P_\Sigma^{\odot\star}\, D^2R(0)[\iota_\Sigma,\iota_\Sigma].
\]
That is,
\[
D^2H(0)[w_i,w_j](\vartheta)=\widetilde P^{\odot\star}_\Sigma\big(D^2R(0)[w_i,w_j]\big)(\vartheta)
=\widetilde P^{\odot\star}_\Sigma(c_{ij},0)(\vartheta).
\]

The projection onto $\{\lambda\},\{\bar\lambda\}$ is, for a term $(c,0) \in X^{\odot \star }$,
\[
\widetilde P^{\odot\star}_{\{\lambda\}}(c,0)(\vartheta)
=
\frac{p^*\,c}{p^*\Delta'(\lambda)q}\, q\,e^{\lambda\vartheta},
\qquad
\widetilde P^{\odot\star}_{\{\bar\lambda\}}(c,0)(\vartheta)
=
\frac{\overline p^*\,c}{\overline p^*\Delta'(\bar\lambda)\overline q}\, \overline q\,e^{\bar\lambda\vartheta}.
\]
Moreover,
\[
\widetilde P^{\odot\star}_\Sigma
=
\widetilde P^{\odot\star}_{\{\lambda\}}
+
\widetilde P^{\odot\star}_{\{\bar\lambda\}},
\]
is precisely the projection onto the two-dimensional spectral subspace.

Explicitly,
\begin{align*}
D^2H(0)[w_i,w_j](\vartheta)
&=
\frac{p^*\,c_{ij}}{p^*\Delta'(\lambda)q}\; q\,e^{\lambda\vartheta}
+
\frac{\overline p^*\,c_{ij}}{\overline p^*\Delta'(\bar\lambda)\overline q}\;
\overline q\,e^{\bar\lambda\vartheta},
\\
&\hspace{5.8cm}
(i,j)\in\{(1,1),(1,2),(2,2)\}.
\end{align*}

Any $u\in X_\Sigma$ can be written uniquely as $u=\Theta (z) = (\vartheta\mapsto z w_1(\vartheta) + \bar z w_2(\vartheta))$ for some $z\in\mathbb{C}$. Then, for the quadratic reduced coefficients in $\dot z$,
\begin{equation}\label{eq:g_quadratic_from_D2H}
g_{20}=\frac{p^*\,c_{20}}{p^*\Delta'(\lambda)q},\qquad
g_{11}=\frac{p^*\,c_{11}}{p^*\Delta'(\lambda)q},\qquad
g_{02}=\frac{p^*\,c_{02}}{p^*\Delta'(\lambda)q}.
\end{equation}
Analogously, the coefficients in $\dot{\bar z}$ are obtained from the $w_2$ component by replacing $\lambda$ with $\bar\lambda$ and $(p^*,q)$ with $(\overline p^*,\overline q)$.

\subsection*{3. Cubic term $D^3H(0)$}

From Lemma 24 in \parencite{BuzaHaller2025}, we have
\begin{equation}\label{eq:D3H-guide-correct}
\begin{aligned}
D^3H(0)[u,v,w]
&=
\widetilde P_\Sigma^{\odot\star}
D^3R(0)[DW(0)u,DW(0)v,DW(0)w]+
\widetilde P_\Sigma^{\odot\star}\Big(
D^2R(0)[D^2W(0)[u,v],DW(0)w]
+\\
&\ D^2R(0)[D^2W(0)[u,w],DW(0)v]
+
D^2R(0)[D^2W(0)[v,w],DW(0)u]
\Big).
\end{aligned}
\end{equation}

For the Hutchinson equation, the vector field is quadratic, hence $D^3R(0)=0$, and \eqref{eq:D3H-guide-correct} simplifies accordingly.

\subsubsection*{Cubic contribution $z^3$}

We compute the trilinear term $D^3H(0)[w_1,w_1,w_1]$, which produces the $z^3$ contribution in the reduced dynamics. Evaluating at $(w_1,w_1,w_1)$ and using $DW(0)w_1=w_1$ and
$D^2W(0)[w_1,w_1]=W_{20}$, we obtain
\begin{equation}\label{eq:D3H111-correct}
D^3H(0)[w_1,w_1,w_1]
=
3\,\widetilde P_\Sigma^{\odot\star}
D^2R(0)[W_{20},w_1].
\end{equation}

The corresponding boundary coefficient is
\[
D^2R(0)[W_{20},w_1]\equiv (c_{30},0)\quad\text{in }X^{\odot\star},
\]

where
\begin{equation}
c_{30}
:=
\big(D^2R(0)[W_{20},w_1]\big)(0)
=
-\frac{r}{K}
\Big(
e^{-\lambda\tau}W_{20}(0)
+
W_{20}(-\tau)
\Big).
\end{equation}

Projecting onto $\Sigma=\{\lambda,\bar\lambda\}$ gives
\begin{align}
D^3H(0)[w_1,w_1,w_1](\vartheta)
&=
3\,\widetilde P^{\odot\star}_\Sigma(c_{30},0)(\vartheta)\notag=
3\left(
\frac{p^*\,c_{30}}{p^*\Delta'(\lambda)q}\,q\,e^{\lambda\vartheta}
+
\frac{\overline p^*\,c_{30}}{\overline p^*\Delta'(\bar\lambda)\overline q}\,\overline q\,e^{\bar\lambda\vartheta}
\right).
\label{eq:D3H111-fullproj-correct}
\end{align}

\subsubsection*{Cubic contribution $z^2\bar z$}

We next compute the trilinear term producing the $z^2\bar z$ contribution in the reduced dynamics. Evaluating at $(w_1,w_1,w_2)$ and using the notation
$W_{20}=D^2W(0)[w_1,w_1]$ and
$W_{11}=D^2W(0)[w_1,w_2]$, we obtain
\begin{align}
D^3H(0)[w_1,w_1,w_2]
&=
\widetilde P_\Sigma^{\odot\star}
\Big(
D^2R(0)[D^2W(0)[w_1,w_1],w_2]
+
2\,D^2R(0)[D^2W(0)[w_1,w_2],w_1]
\Big) \\
&= \widetilde P_\Sigma^{\odot\star}
\Big(
D^2R(0)[W_{20},w_2]
+
2\,D^2R(0)[W_{11},w_1]
\Big)=\widetilde P^{\odot\star}_\Sigma(c_{21},0)(\vartheta) =\\
&=
\left(
\frac{p^*\,c_{21}}{p^*\Delta'(\lambda)q}\,q\,e^{\lambda\vartheta}
+
\frac{\overline p^*\,c_{21}}{\overline p^*\Delta'(\bar\lambda)\overline q}\,\overline q\,e^{\bar\lambda\vartheta}
\right).
\label{eq:D3H112-correct-compact}
\end{align}

Similarly, one obtains the terms corresponding to $z\bar z^2$ and $\bar z^3$.

\subsection{Comparison between the two methods}

The two methods presented in this appendix are equivalent. Here we verify their equivalence by comparing the coefficients obtained for the manifold up to second order and for the reduced dynamics up to third order.

\subsubsection{SSM expansion coefficients}

\begin{figure}[H]
  \centering
  \includegraphics[width=0.95\textwidth]{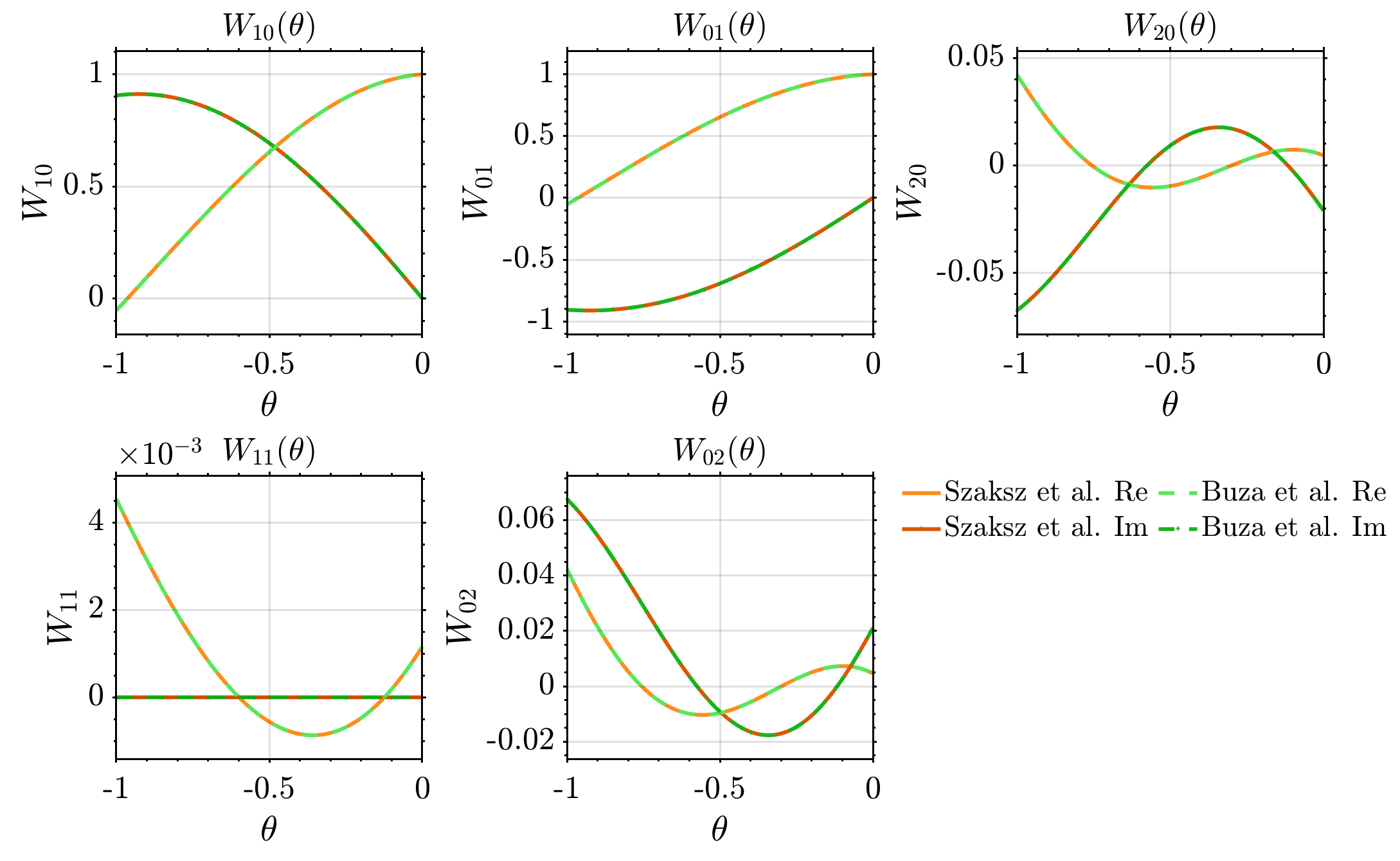}
  \caption{Comparison of the manifold coefficients obtained with the two SSM-reduction approaches, namely the equation-driven method in \parencite{Szaksz2025} and the semigroup-based method of \parencite{BuzaHaller2025}. The figure shows real and imaginary parts of the first- and second-order coefficient functions $W_{10}(\vartheta)$, $W_{01}(\vartheta),W_{20}(\vartheta)$, $W_{11}(\vartheta)$, and $W_{02}(\vartheta)$ over the delay interval $\vartheta\in[-\tau,0]$. The two formulations yield the same manifold parametrization up to second order, consistently with the identical gauge choice used to remove tangential components along the spectral subspace.}
  \label{fig:ManifoldCoeffComp}
\end{figure}

The manifold can be expressed as
\[
W(z,\bar z,\vartheta )=
W_{10}(\vartheta)z + W_{01}(\vartheta)\bar z
+ \frac 12 \big( W_{20}(\vartheta)z^2 + 2W_{11}(\vartheta)z\bar z + W_{02}(\vartheta)\bar z ^2  \big)
+O(|z|^3).
\]

\subsubsection{SSM-reduced dynamics coefficients}

Table~\ref{tab:RedDynCoeffComp} contains the reduced-dynamics coefficients obtained with both methods up to third order, yielding identical results:

\begin{table}[H]
\centering
\renewcommand{\arraystretch}{1.05}
\setlength{\arraycolsep}{2pt}
\caption{Reduced-dynamics coefficients (with a factor $10^{-2}$ collected outside the table).}
\label{tab:RedDynCoeffComp}

\[
\begin{array}{|c|c|c|c|c|c|c|c|}
\hline
  20 & 11 & 02 & 30 & 21 & 12 & 03 \\
\hline
 7.2-4.2\,\mathrm{i} 
& 0.55+0.82\,\mathrm{i} 
& -6.6+5.0\,\mathrm{i} 
& -0.39-0.073\,\mathrm{i} 
& -0.34+0.054\,\mathrm{i} 
& 0.18-0.30\,\mathrm{i} 
& 0.080-0.39\,\mathrm{i} \\
\hline
\end{array}
\]
\end{table}

\paragraph{Use of the table in the reduced dynamics.}

The coefficients reported in Table~\ref{tab:RedDynCoeffComp} are the complex coefficients $b_{jk}$, with an overall factor $10^{-2}$ collected outside the table. They enter directly into the reduced dynamics and its complex conjugate:
\begin{equation}
\dot z = \lambda z 
+ \sum_{j+k=2}^3 b_{jk}\, z^j \bar z^k
+ \mathcal{O}(|z|^4), \qquad 
\dot{\bar z} = \bar{\lambda} \bar z
+ \sum_{j+k=2}^3 \overline{b_{jk}}\, \bar z^j z^k
+ \mathcal{O}(|z|^4),
\end{equation}
The equivalence of the two methods is expected, since the gauge is chosen identically in both cases. In both formulations, the tangential component of the manifold correction along the spectral subspace is removed at every order higher than one. In \parencite{BuzaHaller2025} formulation, this is expressed through the operator $(I-P_\Sigma)$, whereas in \parencite{Szaksz2025} formulation it is imposed through the pairing conditions $\langle W_{jk},\psi_i\rangle=0$. In other words, every higher-order manifold correction is chosen to be normal to the spectral subspace, while the tangential contribution is absorbed into the reduced dynamics coefficients. 

%% file: Sections/AppendixC.tex
\section{Extension of SSM theory to the time-periodic case }\label{appendixC}

%\section{Appendix -- Existence of SSMs for periodic delay DDE}

Let $X : = C([-h,0];\mathbb{R}^n)$; here $h>0$ denotes the maximal delay and $n \in \mathbb{N}$.
Let $O \subset X$ denote an open set containing the origin. 
We consider systems of the form
\begin{subequations} \label{eq:DDE_classical} 
\begin{align}
&\dot{x}(t) = f(t,x_t), \qquad t \geq s \in \mathbb{R}  \label{eq:DDE_classical1} \\ 
& x_s = u,  \label{eq:DDE_classical2}
\end{align}
\end{subequations}
for $f : \mathbb{R} \times O \to \mathbb{R}^n$ continuous, time-periodic, i.e., $f(t+p,u) = f(t,u)$ for some $p > 0$ and all $(t,u) \in \mathbb{R} \times O$, and $f(t,0) = 0$ for all $t \in \mathbb{R}$.
We suppose moreover that $f(t,\cdot)$ is $f(t,\cdot)$ is $C^k$ with $k \geq 1$ for all $t \in \mathbb{R}$.
We shall also suppose $(t,u) \mapsto D_2^jf(t,u)$ is continuous for all $j \leq k$. % this is needed for IMPFT, smoothness of Rhat

\subsection{Sun-star setting}

We cast the system \eqref{eq:DDE_classical} into the sun-star setting of Diekmann et al. \parencite{diekmann2012delay}.\footnote{We assume familiarity with this setting. For the details, we refer to \parencite[Chapter~II]{diekmann2012delay}.}
As usual, we denote by $\{T_0(t)\}_{t \geq 0}$ the shift semigroup acting on $X$; its generator by $A_0$.
The $(\cdot)^{\odot}$ notation should be interpreted with respect to the semigroup $\{T_0(t)\}_{t \geq 0}$.
Recall that $X^{\odot *} \cong \mathbb{R}^n \times L^{\infty}([-h,0];\mathbb{R}^n)$, and
let $\imath : X \xhookrightarrow{} X^{\odot *}$ denote the continuous embedding $u \mapsto (u(0),u)$ defined by the pairing between $X$ and $X^\odot$.
Let 
\begin{displaymath}
    r_i^{\odot *} := (e_i,0) \in X^{\odot *}, \qquad 1\leq i \leq n,
\end{displaymath}
where  $(e_1,\ldots,e_n)$ is the canonical basis of of $\mathbb{R}^n$.
We declare $F : \mathbb{R} \times O \to X^{\odot *}$ to be
\begin{equation}
    F(t,u) := \sum_{i=1}^n \langle f(t,u),e_i\rangle_{\mathbb{R}^n} \, r_i^{\odot *}.
    \label{eq:F}
\end{equation}
We further dissect \eqref{eq:F} into linear and nonlinear parts as
\begin{align*}
    B(t) &: = D_2F(t,0), \\
    R(t,u)&:=F(t,u)-D_2F(t,0)[u].
\end{align*}
The integral equation
\begin{equation}
    x_t = T_0(t) u + \imath^{-1} \int_{s}^t T_0^{\odot *} (t-\tau) F(\tau,x_\tau) \, d\tau, \qquad u \in O, \; t  \geq s,
    \label{eq:AIE}
\end{equation}
determines a \textit{time-dependent semiflow}, i.e., a continuous map $S : \mathcal{D}^S \to X$ defined by
\begin{equation}\label{eq:81}
    S(t,s;u) : = x_t, \qquad \mathcal{D}(S) := \{ (t,s;u) \in [s,\infty) \times \mathbb{R} \times O  \; | \; t \in I_u \},
\end{equation}
where $x_t$ is determined by \eqref{eq:AIE} and $I_u$ is the maximal interval of existence for \eqref{eq:AIE}.
For more details, see \parencite[Section~2.3]{lentjes2025periodic}.
It can be shown that $S(t,s; \cdot )$ is of class $C^k$ via an implicit function argument, as in \parencite[Appenidx~B]{buza2025existence}.
Moreover, solutions of \eqref{eq:DDE_classical} and \eqref{eq:AIE} are in one-to-one correspondence, provided we allow \eqref{eq:DDE_classical} to be satisfied in a weaker sense (see \parencite[Section~VII.6]{diekmann2012delay} for a more elaborate description).

\subsubsection{Linearized semiflow}

In \parencite[Chapter~XII]{diekmann2012delay}, it is shown, under the present assumption of $B$ being strongly continuous, that
\begin{equation}
    U(t,s) u = T_0(t-s) u + \imath^{-1} \int_{s}^t T_0^{\odot *} (t-\tau) B(\tau) U(\tau,s) u \, ds, \qquad u \in X, \; t  \geq s
    \label{eq:AIE_linearpart}
\end{equation}
uniquely determines a strongly continuous forward evolution  system $\{U(t,s)\}_{t \geq s}$.
Since $D_3S(t,s;0)$ also satisfies \eqref{eq:AIE_linearpart}, we must have that $D_3S(t,s;0) = U(t,s)$.

\subsection{Recollection of Floquet theory}

We collect the results we need from \parencite[Chapter~XIII]{diekmann2012delay} concerning Floquet theory in the following proposition.

\begin{proposition} \label{prop:spect}
    The evolution system  $\{U(t,s)\}_{t \geq s}$ defined by \eqref{eq:AIE_linearpart} satisfies the following.
    \begin{enumerate}[label=\upshape{(\roman*)}]
        \item For all $t \geq s$, $U(t+p,s+p) = U(t,s)$.
        \item The spectrum $\sigma(U):=\sigma(U(t+p,t))$ is independent of $t \in \mathbb{R}$, its only accumulation point is $0$, and $\sigma(U(t+p,t)) \setminus\{0\} = \sigma_p(U(t+p,t)) \setminus\{0\}$.
        \item %subbundles of $\mathbb{R}/p \mathbb{Z} \times X$
        For each spectral subset $\Sigma \subset \sigma(U) \setminus \{0\}$ there exists a continuous family projections $\{P_\Sigma(t)\}_{t \in \mathbb{R}}$ on $X$
        of finite dimensional range
        such that $P_\Sigma(t+p) = P_\Sigma(t)$ and $\sigma(U(t+p,t)P_\Sigma(t)) = \Sigma$ for all $t \in \mathbb{R}$. 
        Moreover, for $t \geq s$, $U(t,s)$ maps $\mathrm{im}(P_\Sigma(s))$ onto $\mathrm{im}(P_{\Sigma}(t))$ isomorphically, such that $U(t,s)P_\Sigma(s) = P_\Sigma(t) U(t,s)$.
    \end{enumerate}
\end{proposition}

\begin{proof}
    The first assertion is covered in \parencite[Corollary~XIII.2.2]{diekmann2012delay}.
    The second is a consequence of   \parencite[Theorem~XIII.3.3(i)]{diekmann2012delay} and the fact that $U(t+p,t)^j$ is compact for some $j \in \mathbb{N}$ large enough -- in particular, $\sigma(U(t+p,t))$  is as described in the statement and always contains $0$.
    The third assertion follows from spectral properties of compact operators and  \parencite[Theorem~XIII.3.3(ii)]{diekmann2012delay}; continuity of $t \mapsto P_\Sigma(t)$ follows from \parencite[Theorem~IV.3.16]{kato2013perturbation} and time-periodicity of $U$.
    For a direct proof of continuity, see \parencite[Appendix~A.2]{lentjes2025periodic}.
\end{proof}

\subsection{Modification of the nonlinearity}

To apply invariant manifold theory, one needs to control the global Lipschitz constant of the nonlinearity.
The standard procedure in delay equations is to apply two cut-off functions (in order to preserve smoothness in a desired region),
\begin{displaymath}
    R_{\rho,s} (t,u) : =  \chi_\rho ( |P_\Sigma(s) u |_X) \chi_\rho (|u-P_{\Sigma}(s) u |_X) R(t,u), \qquad (t,u) \in \mathbb{R} \times X,
\end{displaymath}
where $\chi: [0,\infty) \to [0,1]$ is a $C^\infty$ function satisfying
\begin{enumerate}[label=\upshape{(\roman*)}]
    \item $\chi(y) = 1$ for $y \in [0,1]$,
    \item $\chi(y) \in [0,1]$ for $y \in (1,2)$,
    \item $\chi(y) = 0$ for $y \geq 2$,
    \item  $|D\chi (y)| \leq 2$ for all $y \in \mathbb{R}^{\geq 0}$,
\end{enumerate}
and $\chi_\rho(u) : = \chi(u/\rho)$.

We remark that $(t,u) \mapsto R_{\rho,t}(t,u)$ preserves the smoothness of $R$
on the set 
\begin{equation}
     \left\{    (t,u) \in \mathbb{R} \times X     \; \big| \;      |u-P_{\Sigma}(t) u | < \max\{\rho , 2 |P_\Sigma(t) u| \}        \right\} . 
    \label{eq:smoothnessset}
\end{equation}

We quote the following Proposition from \parencite{lentjes2025periodic}.

\begin{proposition}[Proposition~6, \parencite{lentjes2025periodic}] \label{prop:lentjes}
    For $s \in \mathbb{R}$ and sufficiently small $\rho>0$, the operator $R_{\rho,s} (t,\cdot)$ is globally Lipschitz continuous for any $t \in \mathbb{R}$ with Lipschitz constant $\lip(R_{\rho,s} (t,\cdot)) \to 0$ as $\rho \searrow 0$ independent of $s$.
\end{proposition}

Hence, the time-dependent semiflow defined by
\begin{multline}
        S_\rho(t,s;u) = T_0(t) u 
        + \imath^{-1} \int_{s}^t T_0^{\odot *} (t-\tau) \left[B(\tau)S_\rho(\tau,s;u) + R_{\rho,\tau}(\tau,S_\rho(\tau,s;u)) \right] \, d\tau,  \\ \qquad u \in X, \; t  \geq s,
        \label{eq:cutoff_semiflow}
\end{multline}
is globally defined for $\rho>0$ small enough, i.e., $\mathcal{D}(S_\rho) = [s,\infty) \times \mathbb{R} \times X$; moreover $S_\rho(t,s;\cdot) $ is Lipschitz on $X$. % technically you'd need unif Lip cont on R wrt t, but since we are t-periodic its fine .

\begin{lemma} \label{lemma:Sper}
    For $t \geq s$, $S_\rho(t+p,s+p; \cdot) = S_\rho(t,s;\cdot)$.
\end{lemma}

\begin{proof}
    This is shown analogously to \parencite[Proposition~XIII.2.1]{diekmann2012delay} and \parencite[Corollary~XIII.2.2]{diekmann2012delay}, noting that it only relies on the time-periodicity of the time-dependent term, not its linearity.
\end{proof}

\subsection{Proof of Theorem \ref{thm:main}}

\begin{comment}
To make the statement more concise, we set
\begin{displaymath}
    \mathcal{U}(U) = \big\{ (t,s,u) \in \mathcal{D}(S) \; \big| \; S([t,s] \times \{s\} \times \{u \}) \subset U \big\},
\end{displaymath}
and $\mathcal{U}_{s,u}(U) = \{t \in [s,\infty) \; | \; (t,s,u) \in \mathcal{U}(U)\}$.
Moreover, let us denote
\begin{displaymath}
    X = X_\Sigma(t) \oplus X_{\Sigma'}(t) := \mathrm{im}(P_\Sigma(t))\oplus \ker(P_\Sigma(t)), \qquad t \in \mathbb{R},
\end{displaymath}
with projection $P_\Sigma$ as in Proposition~\ref{prop:spect}.
\end{comment}

\begin{proof} [Proof of Theorem \ref{thm:main}].
    We apply \parencite[Theorem~3.1]{chen1997invariant} to the maps $S_\rho(t+p,t;\cdot)$, $t \in \mathbb{R}$.
For this, we set
\begin{displaymath}
    \Phi_t(u) = L_tu + N_t(u), \qquad (t,u) \in \mathbb{R} \times X,
\end{displaymath}
with $L_t := U(t+p,t)$ and $N_{t}(u) := S_\rho(t+p,t;\cdot)-U(t+p,t)$. We need to verify (H.3) and (H.4) of \parencite{chen1997invariant}.
    The prior follows from Proposition~\ref{prop:spect} applied with $\Sigma$ as in the statement and the spectral radius formula \parencite[Theorem~5.2.7]{buhler2018functional}, whereas (H.4) follows from Proposition~\ref{prop:lentjes}, by which we can also guarantee (H.3)-(H.4) hold for all $\{\Phi_t\}_{t \in \mathbb{R}}$ for a single choice of $\rho>0$.

    Let $G_t(\gamma)$ denote the invariant Lipschitz manifold obtained by applying \parencite[Theorem~3.1]{chen1997invariant} to $\Phi_t$ according to the splitting $X = \mathrm{im}(P_\Sigma(t))\oplus \ker(P_\Sigma(t))$,
    \begin{multline*}
        G_t(\gamma) = \big\{ u \in X \; \big| \; \text{there is a negative semiorbit of } \Phi_t, \; \{ u_n\}_{n \leq 0}, \text{ with } u_0 =u  \\ \text{such that } \limsup_{n \to -\infty}\frac{1}{|n|} \ln |u_n| \leq - \ln \gamma \big\}.
    \end{multline*}
    
    We claim that $S_\rho(t,s,G_s(\gamma)) \subset G_t(\gamma)$ for all $t \geq s$.
    Fix $u \in G_s(\gamma)$ and let $\{u_n\}_{n \leq 0}$ denote a negative semiorbit of $\Phi_s$ with $u_0 =u$.
    Set $v:=S_\rho(t,s;u)$ and $v_n := S_\rho(t,s;u_n)$, $n \leq 0$.
    Then $\Phi_t (v_n)= S_\rho(t+p,t; S_\rho(t,s;u_n)) = S_\rho(t+p,s+p;u_{n+1})=v_{n+1}$ for $n \leq -1$ (recall Lemma~\ref{lemma:Sper}), hence $\{v_n\}_{n \leq 0}$ is a negative semiorbit of $\Phi_t$ with $v_0 = v$.
    Since $S_\rho(t,s;\cdot)$ is Lipschitzian, we have $|v_n| \leq C |u_n|$ for all $n \leq 0$ and hence
    \begin{displaymath}
        \limsup_{n \to -\infty}\frac{1}{|n|} \ln |v_n| \leq - \ln \gamma.
    \end{displaymath}
    This shows the claim.

    For the converse inclusion, $S_\rho(t,s,G_s(\gamma)) \supset G_t(\gamma)$, we set $j \in \mathbb{N}$ to be the unique integer for which $t-jp \leq s < t-(j-1)p$.
    Fix $v \in G_t(\gamma)$ and let $\{v_{n}\}_{n \leq 0}$ denote a negative semiorbit under $\Phi_t$ such that $v_0 = v$.
    We set $u:= S_\rho(s,t-jp;v_{-j})$, which is an element of $G_s(\gamma)$ by the previous claim and the fact that $v_{-j} \in G_{t-jp}(\gamma) = G_t(\gamma)$, which in turn follows from \parencite[Theorem~3.1(ii)]{chen1997invariant}.
    Computing  $S_\rho(t,s;u ) = \Phi^j_t(v_{-j}) =v$ shows the converse inclusion.

    Hence, $G_t(\gamma) = S_\rho(t,s,G_s(\gamma))$ for all $t \geq s$.
    This shows that $t \mapsto G_t(\gamma)$ is continuous.
    Recall from \parencite[Equation~(3.16)]{chen1997invariant} and Proposition~\ref{prop:lentjes} that we can control the Lipschitz constant of the function whose graph is $G_t(\gamma)$ via adjusting $\rho$, uniformly with respect to $t$.
    Upon examining \eqref{eq:cutoff_semiflow} and the set \eqref{eq:smoothnessset} on which $(t,u) \mapsto R_{\rho,t}(t,u)$ preserves the smoothness of $R$, we infer that $\Phi_t$ is $C^k$ on a neighborhood of $G_t(\gamma)$.
    Thus, each $G_t(\gamma)$ is $C^\ell$ smooth according to \parencite{irwin1980new}.
    The local invariant manifold statements \ref{st1}-\ref{st2} now follow by observing that $S_\rho=S$ on a small enough neighborhood of $[s,\infty)\times \mathbb{R} \times \{0\}$.

    For statement \ref{st3} we utilize the foliation part of \parencite[Theorem~3.1]{chen1997invariant}.
    For this, recall that leaves of said foliations are defined  by
    \begin{equation}
        M_{t,u}(\gamma) = \big\{ v \in X \; \big| \;  \limsup_{n \to \infty}\frac{1}{n} \ln |\Phi_t^n(u) - \Phi_t^n(v)| \leq \ln \gamma \big\}.
        \label{eq:M_def}
    \end{equation}
    We show $S_\rho(t,s,M_{s,u}(\gamma)) \subset M_{t,S_\rho(t,s;u)}(\gamma)$ for $t \geq s$.
    Fix $v \in M_{s,u}(\gamma) $. 
    We have $ \Phi^n_t \circ S_\rho(t,s;\cdot) = S_\rho(t+np,s;\cdot) = S_\rho(t+np,s+np,\Phi_s^n(\cdot))$, hence
    \begin{displaymath}
        \limsup_{n \to \infty}\frac{1}{n} \ln | \Phi^n_t \circ S_\rho(t,s;v) - \Phi^n_t \circ S_\rho(t,s;u)| \leq \ln \gamma,
    \end{displaymath}
    using the Lipschitzness of $S_\rho(t,s;\cdot)$ once more.
    This shows the desired property, and hence, upon using a contraction mapping argument to define $\pi_t^\rho$ to be the continuous projection mapping $M_{t,u}(\gamma)$ to its unique intersection with $G_t(\gamma)$ (see \parencite[Section~4]{buza2024spectral}), we obtain \eqref{eq:foliation_invariance} globally for the cut-off setting; the local version then follows as above.

    Finally, we observe that \eqref{eq:foliation_decay} follows by \eqref{eq:M_def}.
    Recall from \parencite[Theorem~3.1(iv)]{chen1997invariant} that one can replace $\gamma$ by $\gamma - \varepsilon$ for $\varepsilon > 0$ small enough, since $\gamma \notin |\sigma(U)|$.
    Hence, for $u \in X$ and $s \in \mathbb{R}$ fixed, \eqref{eq:M_def} implies the existence of a $C>0$ such that 
    \begin{displaymath}
        |\Phi_s^n(u) - \Phi_s^n(\pi_s(u))| < C\gamma^n \qquad \text{for all } n \in \mathbb{N}.
    \end{displaymath}
    For $t \geq s$, there exists a unique $n \in \mathbb{N}$ for which $s+np \leq t < s+(n+1)p$.
    Therefore, using Lipschitzness of $S_\rho(t,s;\cdot)$ once more,
    \begin{align*}
        | S_\rho(t,s;u) - S_\rho(t,s;\pi_s(u))| &\leq \widetilde{C} |\Phi_s^n(u) - \Phi_s^n(\pi_s(u))| \\
        & < C \widetilde{C} \gamma^n \\
        & \leq  \hat{C} e^{(t-s) \ln(\gamma)/p}. % depending on  whether \gamma is smaller or greater than 1, one of the two bounds is used the possible remainder is absorbed in C
    \end{align*}
    The local version follows as above.
\end{proof}

\begin{remark}
    Note that, instead of considering \eqref{eq:AIE},
    we could have also considered the integral equation obtained by considering $R$ as a perturbation above $\{U(t,s)\}_{t \geq s}$, using a sun-star calculus for $\{U(t,s)\}_{t \geq s}$.
    This would have enabled a more elegant proof conducted entirely in the time-dependent setting, without the need to appeal to Poincar\'e maps, just as in the case of periodic center manifolds \parencite{Lentjes2025}. 
    % (This setting would be necessary to obtain normal forms, but for external forcing this)
    Since we do not wish to derive rigorous formulas for the computation of manifold coefficients here (our example is dealt with in a data-driven setting), Theorem~\ref{thm:main} suffices.
\end{remark}

%% file: Sections/AppendixD.tex
\section{Mackey-Glass SSM-reduced dynamics}\label{appendixD}

\begin{figure}[H]
  \centering
  \includegraphics[width=1.0\textwidth]{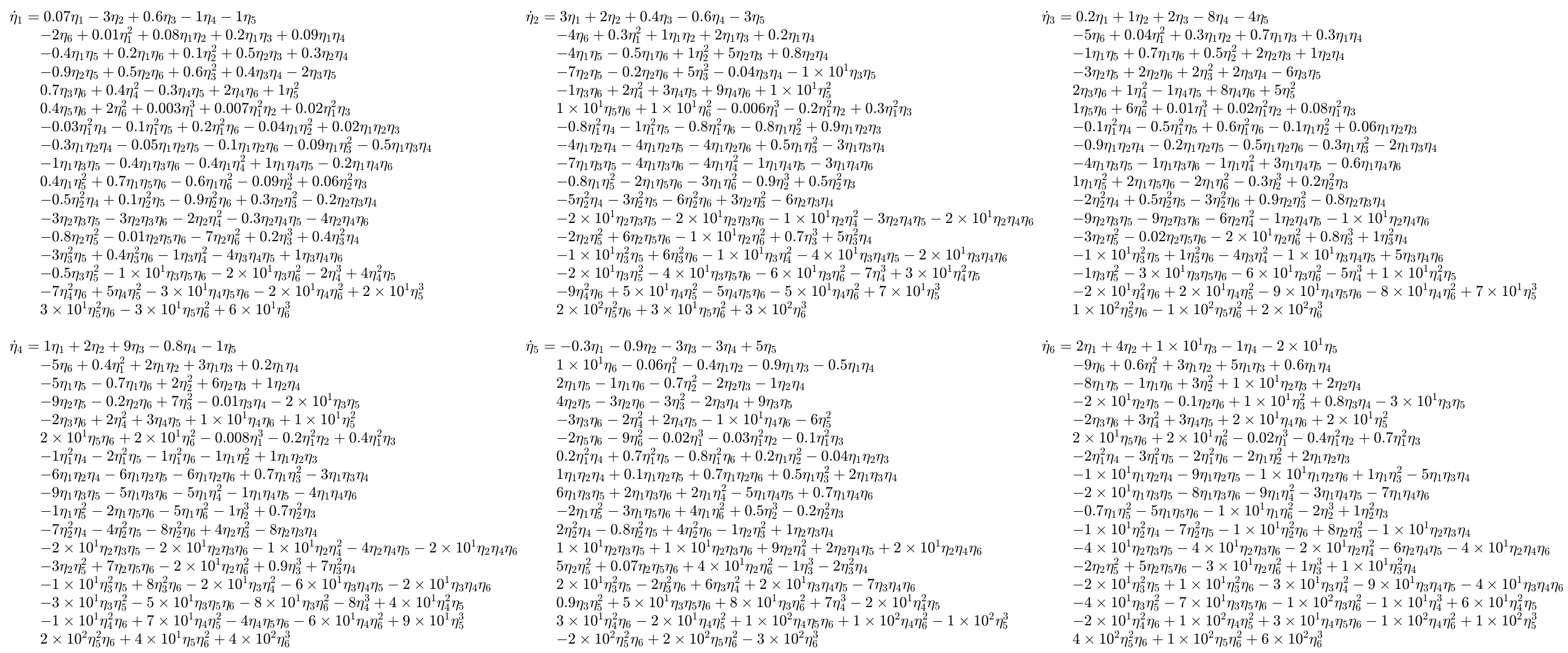}
  \caption{Polynomial 6D ODE SSM-reduced model for the fulll system \eqref{eq:mackey-glass}, with coefficients rounded to one significant digit.}
  \label{fig:DE_MG_redDyn}
\end{figure}